\documentclass[10pt, reqno]{amsart}

\usepackage{enumerate}
\usepackage{aries}
\usepackage[unicode]{hyperref}
 \hypersetup{colorlinks=false}
\usepackage[all,cmtip]{xy}
\usepackage{slashed}

\title[An Index Localization Theorem for Dirac operators]{A Localization Theorem for Dirac operators}
\author[M. Maridakis ]{Manousos Maridakis}
\email{mmaridaki1@gmail.com}

\begin{document}


\vskip.15in

\begin{abstract}
We study perturbed Dirac operators of the form $ D_s= D + s{\mathcal A} :\Gamma(E^0)\rightarrow \Gamma(E^1)$ over a compact Riemannian manifold $(X, g)$ with symbol $c$ and special bundle maps ${\mathcal A} :  E^0\rightarrow E^1$ for $\lvert s \rvert >>0$. Under a simple algebraic criterion on the pair $(c, {\mathcal A})$, solutions of $D_s\psi=0$ concentrate as $\lvert s \rvert \to\infty$ around the singular set $Z_{\mathcal A}$ of ${\mathcal A}$. We prove a spectral separation property of the deformed Laplacians $D_s^*D_s$ and $D_s D_s^*$, for $\lvert s \rvert >>0$. As a corollary we prove an  index localization theorem.
\end{abstract}

\maketitle

 \vspace{1cm}

\setcounter{equation}{0}
\section{Introduction}

 Given a first order elliptic operator $D$ with symbol $c$, one can look for zeroth-order perturbation ${\mathcal A}$  such that all $W^{1,2}$-solutions of the equation
\begin{equation}
\label{eq:1}
D_s\xi\ :=\ (D+s{\mathcal A})\xi\ =\  0
\end{equation}
concentrate along submanifolds $Z_\ell$ as $\lvert s \rvert \to\infty$. In \cite{m} we gave a simple algebraic criterion satisfied from $c$ and ${\mathcal A}$ (see also Definition~\ref{defn:Main_Def_CP}) that ensures localization of solutions to $D_s\xi=0$ in the sense of Proposition~\ref{prop:concentration_Prop}. We also provided a general method for constructing examples of $c$ and ${\mathcal A}$ satisfying this criterion. 

This paper is a sequel to \cite{m} proving two main results:  The first is a  Spectral Decomposition Theorem for operators $D_s$  (Theorem~\ref{Th:mainT}) that satisfy the concentration condition \eqref{eq:cond}, two transversality conditions namely and two degeneration conditions. The Spectral Decomposition Theorem then states that:
\begin{itemize}
\item The eigenvectors corresponding to the  low eigenvalues  of  $D^*_sD_s$ concentrate near the singular set of the perturbation bundle map ${\mathcal A}$ as $\lvert s \rvert \to\infty$.
\item  The eigenvalues of $D^*_sD_s$ corresponding to the eigensections that do not concentrate grow at least as $O(\lvert s \rvert)$ when $s\to\infty$.
\item  The components of the critical set of the perturbation bundle map ${\mathcal A}$ are submanifolds $Z_\ell$ and each determines an associated decomposition of the normal bundle to the $Z_\ell$ giving a precise asymptotic formula for the solutions of \eqref{eq:1} when $\lvert s \rvert$ is large.
\end{itemize}

Our second main result is an Index Localization Theorem, which follows from the Spectral Decomposition Theorem.  It describes an Atiyah-Singer type index formula for how the index of $D$ decomposes as a sum of local indices associated with specific Dirac operators on suitable bundles over the submanifolds $Z_\ell$.

 The concentration condition \eqref{eq:cond} was  previously found by I. Prokhorenkov and K.~Richardson in \cite{pr}. They classified the complex linear perturbations ${\mathcal A}$ that satisfy \eqref{eq:cond} but found few examples, all of which concentrate at points. They proved a spectral decomposition theorem and an index localization theorem in the case where the submanifolds $Z_\ell$ are finite unions of points.  Our result generalize that theorem in the more general case where $Z_\ell$ are submanifolds. Using our general construction of perturbations ${\mathcal A}$ using spinors described in \cite{m}, the Index Localization formula draws some interesting conclusions of the index of $D$ as described by the local information of the zero set of a spinor. Part of our analysis on the asymptotics of the solutions when $s$ is large, is similar to the work \cite{bl} of J. M. Bismut and G. Lebeau. They use adiabatic estimates to examine in detail the case where ${\mathcal A}$ is a Witten deformation on a compact complex manifold with singular set a complex submanifold.

 \medskip

This paper has nine sections. Section~\ref{sec:Concentration_Principle_for_Dirac_Operators} reviews  the concentration condition, improves on some analytic consequences described in \cite{m} and states the main assumptions and results, which are proved in later sections. An important part of the story relies on the vector bundles $S^{0\pm}_\ell$ and $S^{1\pm}_\ell$ over the sub manifolds $Z_\ell$ that are introduced for the statement of Theorem~\ref{Th:mainT} and described in detail later.

Section~\ref{Sec:structure_of_A_near_the_singular_set} is divided into two subsections. In the first subsection, we examine the perturbation term ${\mathcal A}$, regarding it as a section of a bundle of linear maps, and requiring it to be transverse to certain  subvarieties, where the linear maps jump in rank. The transversality condition, allows us to write down a Taylor series expansion of ${\mathcal A}$ in the normal directions along each connected component $Z_\ell$ of $Z_{\mathcal A}$. The 1-jet terms of ${\mathcal A}$ together with the assumptions of Theorem~\ref{Th:mainT} indicate the existence of the vector bundles $S^{0\pm}_\ell$ and $S^{1\pm}_\ell$ and their geometry regarding the kernel bundles $\ker {\mathcal A}$ and $\ker {\mathcal A}^*$. In the second subsection we examine the geometry of the Clifford connection 1-form along the singular set and introduce a more appealing connection that ties better with the geometry of the bundles described earlier.  

In Section \ref{sec:structure_of_D_sA_along_the_normal_fibers},  we examine the geometric structure of $D + s{\mathcal A}$ near the singular set $Z_{\mathcal A}$ of ${\mathcal A}$. The transversality Assumption, allows us to write down a Taylor expansion of the coefficients of the operator $D$ in the normal directions along each connected component $Z_\ell$ of $Z_{\mathcal A}$. The expansion decomposes into the sum of a ``vertical'' operator $\slashed{D}_0$ that acts on sections along each fiber of the normal bundle, and a ``horizontal'' operator $\bar D^Z$ that acts on sections of the bundles $\pi^* S^+_\ell \to N$ on each normal bundle $N \to Z_\ell$ of each component $Z_\ell$ of  $Z_{\mathcal A}$. These local operators are, in turn, used to describe the local models of the solutions we will be using in the construction of the asymptotic formula of eigenvectors of the global low eigenvalues when $s$ is sufficiently large. 

In Section \ref{sec:Properties_of_the_operators_slashed_D_s_and_D_Z.} we review some well known properties of the horizontal and vertical operators introduced in Section \ref{sec:structure_of_D_sA_along_the_normal_fibers}. The vertical operator is well known to associate to a Euclidean harmonic oscillator in the normal directions. 

In Section \ref{sec:harmonic_oscillator} we introduce weighted Sobolev spaces related to the aforementioned harmonic oscillator. The mapping properties of the vertical operator allows to build a bootstrap argument for the weights of the Gaussian approximate solutions. In Section \ref{sec:The_operator_bar_D_Z_and_the_horizontal_derivatives} we examine the mapping properties of the horizontal operator with respect to these spaces.   

The technical analysis needed to prove the Spectral Separation Theorem is carried out in Section~\ref{sec:Separation_of_the_spectrum} and Section~\ref{sec:A_Poincare_type_inequality}. Using the horizontal operator $D^Z_+$ and a splicing procedure, we define a space of approximate eigensections, and work through estimates needed to show that these approximate eigensections are uniquely corrected to true eigensections.

In Section~\ref{sec:Morse_Bott_example} we provide an example from Morse-Bott homology theory where the Euler characteristic of closed manifold is written as a signed sum of the Euler characteristics of the critical submanifolds of a Morse-Bott function.    

Finally, we provide an appendix with the statements and proofs that fall long in the main text. In subsections \ref{subApp:tau_j_tau_a_frames}- \ref{subApp:The_pullback_bundle_E_Z_bar_nabla_E_Z_c_Z} of the Appendix we describe the construction and properties of an adapted connection that is an extension of the connection appeared in Section~\ref{Sec:structure_of_A_near_the_singular_set}, over the total space of the normal bundle $N\to Z$ of the singular component $Z$. The adapted connection is used for the analysis carried out in Sections~\ref{sec:Separation_of_the_spectrum} and \ref{sec:A_Poincare_type_inequality}.

The topology of the subvarieties that produce the singular set $Z_{\mathcal A}$ is the next step in the pursue of a connection of the Index Localization formula with characteristic numbers of these subvarieties. We will pursue this objective in another article. A good reference for them is \cite{k2}.

\section*{Acknowledgements} The author would like to express his gratitude to his advisor, Professor T.H. Parker for introducing him to the subject, for his suggestions and for his continuous encouragement and to \'{A}kos Nagy for numerous engaging conversations on the subject.

\medskip
   
  \vspace{1cm}
  
\setcounter{equation}{0}   
\section{Concentration Principle for Dirac Operators } 
\label{sec:Concentration_Principle_for_Dirac_Operators}

This section reviews some very general conditions described in \cite{m} in which a family $D_s$ of deformed Dirac operators concentrates. Furthermore it describes the necessary assumptions and states the two main theorems of the paper.     
  
\vspace{8mm}

Let $(X,\, g_X)$ be a closed Riemannian manifold and $E$ be a real Riemannian vector bundle over $X$. We say that $E$ admits a $\mathbb{Z}_2$ graded Clifford structure if there exist an orthogonal decomposition $E=E^0\oplus E^1$ and a bundle map $c: T^*X\to \mathrm{End}(E)$ so that that $c(u)$ interchanges $E^0$ with $E^1$ and satisfies the Clifford relations,
\begin{equation}
\label{eq:Clifford_relations}
c(u)c(v)+c(v)c(u)\ =\ -2 g_X(u, v)\ 1_E,
\end{equation}
and  for all $u, v\in T^*X$. We will often denote $u_\cdot = c(u)$. Let also $\nabla : \Gamma(E) \to \Gamma(T^*X\otimes E)$ be a connection compatible with the Clifford structure; by definition $\nabla$ is compatible with the Clifford structure if it is compatible with the Riemannian metric of $E$, if it preserves the $\mathbb{Z}_2$ grading and if it satisfies the parallelism condition $\nabla c = 0$. The composition
\begin{equation}
\label{eq:Def_Dirac}
D = c\circ \nabla : C^\infty(X; E^0)\rightarrow C^\infty(X; E^1),
\end{equation}
then defines a Dirac operator. Given a real bundle map ${\mathcal A}: E^0\to E^1$ we form the family of operators
\begin{align*}
D_s=D+s{\mathcal A},
\end{align*}
where $ s\in\mathbb{R}$.  Furthermore, using the Riemannian metrics on the bundles $E^0$ and $E^1$ and the Riemannian volume form $d\mathrm{vol}^X$, we can form the adjoint ${\mathcal A}^*$, and the formal $L^2$ adjoint $D_s^*= D^*+ s{\mathcal A}^*$ of $D_s$. We recall from \cite{m} the definition of a concentrating pair:

\begin{defn}[Concentrating pairs]
\label{defn:Main_Def_CP}
In the above context, we say that  $(c, {\mathcal A})$  is a {\em concentrating pair} if it satisfies the algebraic condition
\begin{align}
\label{eq:cond}
{\mathcal A}^*\circ c(u) = c(u)^*\circ {\mathcal A},  \qquad \text{for every $u\in T^*X$.}
\end{align}
\end{defn}

It is proven in \cite{m}[Lemma 2.2] that $(c, {\mathcal A})$ being a concentrating pair is equivalent to the differential operator
\begin{equation}
\label{eq:bundle_cross_terms}
B_{\mathcal A} = D^*\circ {\mathcal A} + {\mathcal A}^* \circ D,  
\end{equation}
being a bundle map. 

In the analysis of the family $D+s{\mathcal A}$, a key role is played by the {\em singular set of ${\mathcal A}$}, defined as 
\[
Z_{\mathcal A} : = \left\{p\in X:\, \ker {\mathcal A}(p)\ne 0 \right\},
\]
that is, the set where  ${\mathcal A}$ fails to be injective. Since the bundles $E^0$ and $E^1$ have the same rank, a bundle map ${\mathcal A}:E^0\to E^1$ will necessarily have index zero. However maps satisfying condition \eqref{eq:cond} are far from being generic and they may have a nontrivial kernel bundle everywhere in which case the former definition of $Z_{\mathcal A}$ will be equal to $X$. In particular, the kernel bundle $\ker{\mathcal A} \to X$ and $\ker{\mathcal A}^* \to X$, will be a $\text{Cl}(T^*X)$- submodules of $E^0$ and $E^1$ over $X$. In that case the critical set $Z_{\mathcal A}$ is defined as the jumping locus of of the generic dimension of the kernel in $X$. Note that the dimensions of the kernel and cockernel bundles of ${\mathcal A}$ will jump by the same amount in each of the connected components of $Z_{\mathcal A}$. In this case we can consider the kernel subbundles of $E^0$ and $E^1$ over $X\setminus Z_{\mathcal A}$ and we will assume that we can extend them to subbundles of constant rank over the whole $X$. The Clifford compatible connection can be chosen to preserve these subbundles. Therefore ${\mathcal A}$ and the Dirac equation can be broken to the part of $D$ and ${\mathcal A}$ living in the orthogonal complement of these bundles, reducing the problem to the case where ${\mathcal A} :E^0 \to E^1$ is generically an isomorphism except at the set $Z_{\mathcal A}$.      

 Fix $\delta>0$, set $Z_{\mathcal A}(\delta)$ to be the $\delta$-neighborhood of  $Z_{\mathcal A}$ where the distance function is well defined and set, 
\[
\Omega(\delta)=X\setminus Z_{\mathcal A}(\delta),
\] 
be its complement. The following proposition from \cite{m} shows the importance of the concentrating condition  \eqref{eq:cond}. 

\begin{prop}[Concentration Principle]
\label{prop:concentration_Prop}
There exist $C'=C'(\delta, {\mathcal A}, C)>0$ and $s_0 = s_0(\delta, {\mathcal A})>0$ such that whenever $\lvert s\rvert > s_0$ and $\xi\in C^{\infty}(E)$ is a section with $L^2$ norm 1 satisfying $\|D_s\xi\|^2_{L^2(X)}\leq C \lvert s \rvert $, one has the estimate  
\begin{equation}
\label{eq:concentration_estimate}
\int_{\Omega(\delta)} \lvert\xi\rvert^2\, d\mathrm{vol}^X\ <\frac{C'}{\vert s\rvert}.
\end{equation}
\end{prop}

We have the following improvement: 

\begin{prop}[Improved concentration principle]
\label{prop:improved_concentration_Prop}
Let $\ell \in \mathbb{N}_0,\ 0<\alpha<1$ and $C_1>0$, and choose $\delta>0$ small enough so that $100(\ell+1) \delta$ is a radius of a tubular neighborhood where the distance from $Z_{\mathcal A}$ is defined. Then there exist $s_0 = s_0(\delta, C_1, c, {\mathcal A}),\  \epsilon =  \epsilon(\delta, \ell, \alpha, {\mathcal A})$ and $C'=C'(\delta, \ell, \alpha, {\mathcal A})$ all positive numbers such that whenever $\lvert s \rvert > s_0$ and $\xi\in C^{\infty}(X;E)$ is a section satisfying $D_s^* D_s \xi  = \lambda_s \xi$ with $\lambda_s < C_1 \lvert s\rvert$, one has the estimate  
\begin{equation}
\label{eq:improved_concentration_estimate}
\|\xi\|_{C^{\ell, \alpha}(\Omega(\delta))}\ < C' \lvert s\rvert^{-\tfrac{\ell}{2}} e^{-\epsilon\lvert s\rvert}\|\xi\|_{L^2(X)}.
\end{equation}
\end{prop}
The proof is located in Appendix subsection~\ref{subApp:various_analytical_proofs}.

\begin{rem}
\begin{enumerate}
\item The dependence of the constant $C'$ in \eqref{eq:concentration_estimate} from $\delta$ can be made explicit when one has an  estimate of the form $\lvert{\mathcal A}\xi\rvert^2 \geq r^a \lvert\xi\rvert^2$ on a sufficiently small, tubular neighborhood of $Z_{\mathcal A}$, where  $r$ is the distance  from $Z_{\mathcal A}$ and $a>0$. Assumption~ \ref{Assumption:normal_rates} below gives such an estimate.     

\item  Propositions~\ref{prop:concentration_Prop} and \ref{prop:improved_concentration_Prop} are applicable to general concentration pairs $(\sigma, {\mathcal A})$ where the symbol $\sigma$ is just elliptic.
\end{enumerate}
\end{rem}

Propositions~\ref{prop:concentration_Prop} and \ref{prop:improved_concentration_Prop} also hold for the adjoint operator $D^*_s$ because of the following lemma: 
\begin{lemma}
\label{lem:lr}
The concentration condition \eqref{eq:cond} 
 is equivalent to  
\begin{equation}
\label{eq:cond_version2}
 c(u)\circ {\mathcal A}^* = {\mathcal A}\circ c(u)^*\  \quad  \forall u\in T^*X.
\end{equation}
Hence $D + s{\mathcal A}$ concentrates if and only if the adjoint operator  $D^* + s{\mathcal A}^*$ does.
\end{lemma} 

From Proposition~\ref{prop:concentration_Prop}, the  eigensections $\xi$ satisfying  $D_s^*D_s \xi = \lambda(s) \xi$ with $\lambda (s)=O(s)$, concentrate around $Z_{\mathcal A}$ for large $s$. An interesting question arises as to what extend these localized solutions can  be reconstructed using local data obtained from $Z_{\mathcal A}$. The answer will be given in Theorem~\ref{Th:mainT}. We start by describing the assumptions we will use throughout the paper in proving the main theorem. We start by imposing conditions on ${\mathcal A}$ that will guarantee that the connected components $Z_\ell$ of  the singular set  $Z_{\mathcal A}$ are submanifolds and that the rank of ${\mathcal A}$ is constant on each $Z_\ell$.  For this, we regard ${\mathcal A}$ as a section of a subbundle ${\mathcal L}$ of $\mathrm{Hom}(E^0,E^1)$ as in the following diagram:
\begin{align*}
\label{diag:Ldiagram}
 \xymatrix{
\ \ {\mathcal L}\ \ \ar@{^{(}->}[r]  \ar[d]  & \mathrm{Hom}(E^0, E^1) \supseteq {\mathcal F}^l\\
X \ar@/^1pc/[u]^{{\mathcal A}} & 
}
\end{align*}
Here ${\mathcal L}$ is a bundle that parametrizes some family of linear maps ${\mathcal A}:E^0\to E^1$ that  satisfy  the concentration condition \eqref{eq:cond}  for the operator \eqref{eq:Def_Dirac}, that is, each $A\in{\mathcal L}$ satisfies  $A^*\circ c(u) = c(u)^*\circ A$  for every $u\in T^*X$. Inside the total space of the bundle $\mathrm{Hom}(E^0, E^1)$,  the set of linear maps with $l$-dimensional kernel is a   submanifold $\mathcal{F}^l$;    because $E^0$ and $E^1$ have the same rank, this submanifold  has codimension~$l^2$. In all of our examples in \cite{m} the set ${\mathcal L}\cap {\mathcal F}^l$ is a manifold for every $l$.

\bigskip

\noindent\textbf{Transversality Assumptions:} 

\begin{enumerate}
\addtocounter{enumi}{0}
\item 
\label{Assumption:transversality1}
As a section of ${\mathcal L}$, ${\mathcal A}$ is transverse to  ${\mathcal L}\cap {\mathcal F}^l$ for every $l$, and these intersections occur at points where ${\mathcal L}\cap {\mathcal F}^l$ is a manifold.

\item    
\label{Assumption:transversality2}
 $Z_\ell$ is closed for all $\ell$.
\end{enumerate}

As a consequence of the Implicit Function Theorem, ${\mathcal A}^{-1}({\mathcal L}\cap {\mathcal F}^l)\subseteq X$ will be a submanifold of $X$ for every $l$. The singular set decomposes as a union of these critical submanifolds and, even further, as a union of connected components $Z_\ell$:
\begin{equation}
\label{eq:def_Z_l}
Z_{\mathcal A}\, =\, \bigcup_l {\mathcal A}^{-1}({\mathcal L}\cap {\mathcal F}^l)\, =\,\bigcup_\ell Z_\ell.
\end{equation}
The bundle map ${\mathcal A}$ has constant rank along each $Z_\ell$, and $\ker {\mathcal A}$ and $\ker {\mathcal A}^*$ are well defined bundles over $Z_\ell$.  Fix one critical $m$-dimensional submanifold $Z:= Z_\ell$  with normal bundle $\pi: N \to Z$. As explained in Appendix~\ref{App:Fermi_coordinates_setup_near_the_singular_set}, for small $\varepsilon>0$ the exponential map identify an open tubular neighborhood $B_Z(2\varepsilon)\subset X$ of $Z$ with the neighborhood of the zero section $\mathcal{N}:=\mathcal{N}_\varepsilon \subset N$. There are principal frame bundle isomorphisms and induced bundle isomorphisms,
\begin{equation} 
\label{eq:exp_diffeomorphism}
I:= T\exp: TN\vert_{\mathcal{N}} \to TX\vert_{B_Z(2\varepsilon)}\quad \text{and} \quad  {\mathcal I}: \tilde{E} \to E\vert_{B_Z(2\varepsilon)}.
\end{equation}
Note that when restricted to the zero set of $N$, these bundle maps become identities. Let $S^0$ and $S^1$ denote the bundles obtained by parallel translating  $\ker {\mathcal A}\to Z$ and $\ker {\mathcal A}^* \to Z$ along the rays emanating from the core set of $\mathcal{N}$, with corresponding bundles of orthogonal projections  $P^i: E^i\vert_Z \to S^i\vert_Z,\, i =0,1$. 

The next couple of assumptions involve the conditions for the infinitesimal behavior of ${\mathcal A}$ near each  $Z \subset Z_{\mathcal A}$: 

\bigskip

\noindent \textbf{Non-degeneracy Assumption:} 
\begin{enumerate}
\addtocounter{enumi}{2}
\item  
\label{Assumption:normal_rates}
Set $\tilde {\mathcal A} =  {\mathcal I}^{-1}{\mathcal A}  {\mathcal I}$ and $\bar{\mathcal A}$ the restriction of $\tilde{\mathcal A}$ to the zero set. We require 
\[
\left.\tilde{\mathcal A}^* \tilde{\mathcal A} \right\vert_{S^0} = r^2\left(Q^0 + \left.\frac{1}{2}\bar{\mathcal A}^* \bar A_{rr}\right\vert_{S^0} \right)+ O(r^3),
\]
where $r$ is the distance function from $Z$, and $Q^0$ is a positive-definite symmetric endomorphism of the bundle $S^0$.
\end{enumerate}

Comparing with the expansion of Proposition~\ref{prop:properties_of_perturbation_term_A}\eqref{prop:Taylor_expansions_ of perturbation_term} for $\left.\tilde{\mathcal A}^*\tilde{\mathcal A}\right\vert_{S^0}$, the condition replaces the $\text{Cl}_n^0(TX\vert_Z)$ invariant term $\left.\tfrac{x_\alpha x_\beta}{r^2} (A^*_\alpha A_\beta)\right\vert_{S^0}$ with a matrix $Q^0$. The statement of the non-degeneracy assumption is twofold: 1) For every $v = v_\alpha e_\alpha\in N,\ \lvert v\rvert=1$, by Schur's lemma, there are $\text{Cl}_n^0(TX\vert_Z)$-invariant decompositions
\[
S^0\vert_Z = \bigoplus^{\ell(v)}_{k=1} S_{k,v}\quad \text{and}\quad  \left.v_\alpha v_\beta (A^*_\alpha A_\beta)\right\vert_{S^0} = \sum_k \lambda^2(v) P_{S_{k, v}},
\] 
for some $\lambda(v)\in [0,\infty)$, where $P_{k, v}:S^0\vert_Z \to S_{k,v}$ is a $\text{Cl}_n^0(TX\vert_Z)$-invariant orthogonal projection. The assumption guarantees that the decomposition doesn't depend on the radial directions $v\in N,\ \lvert v\rvert=1$. 2) $Q^0$  has trivial kernel as an element of $\mathrm{End}(S^0)$.  

An eigenvalue of the section $Q^0: Z \to \mathrm{End}(S^0\vert_Z)$  is a function $\lambda^2:Z \to (0,\infty)$ so that the section $Q^0 - \lambda^2 1_{S^0\vert_Z}: Z \to \mathrm{End}(S^0\vert_Z)$ has image consisting of solely singular matrices. By \cite{k}[Ch.II, Th.6.8, pp.122] we can always choose a family of eigenvalues (possibly with repetitions) that are smooth functions $\{\lambda_\ell :Z \to (0, \infty)\}_{\ell=1}^d$. In Lemma~\ref{lem:basic_properties_of_M_v_w} it is proven that Assumption~\ref{Assumption:normal_rates} implies an analogue expansion for the bundle map,  
\[
\left.\tilde{\mathcal A}\tilde{\mathcal A}^*\right\vert_{S^1} = r^2\left(Q^1 + \left.\frac{1}{2}\bar{\mathcal A} \bar A^*_{rr}\right\vert_{S^1} \right)+ O(r^3).
\]
Also from equation \eqref{eq:Q0_vs_Q1} of that lemma follows that $Q^1: S^1\vert_Z \to S^1\vert_Z$ is a $\text{Cl}^0(T^*X\vert_Z)$-invariant map and that the matrices $Q^i,\ i=0,1$ have the same spectrum. By Schur's lemma, we have a decomposition 
\begin{equation}
\label{eq:eigenspaces_of_Q_i}
S^i\vert_Z = \bigoplus_\ell S^i_\ell,
\end{equation}
into the $\text{Cl}^0(T^*X\vert_Z)$-invariant eigenspaces of the distinct eigenvalues $\{\lambda_\ell^2\}_\ell$ of $Q^i$. Our final assumption guarantees that the eigenbundles of this decomposition have constant rank:    

\newpage

\noindent \textbf{Stable degenerations:} 
\begin{enumerate}
\addtocounter{enumi}{3}
\item  
\label{Assumption:stable_degenerations}
We require that every two members of the family $\{\lambda_\ell:Z \to (0,\infty)\}_{\ell=1}^d$ are either identical or their graphs do not intersect.
\end{enumerate}

 The summands of decomposition \eqref{eq:eigenspaces_of_Q_i} are also $\text{Cl}^0(T^*X\vert_Z)$-submodules. By Assumption~\ref{Assumption:stable_degenerations}, the graphs of eignevalues $\{\lambda_\ell\}_\ell$ do not intersect and therefore the bundle of vector spaces $\{(S_\ell^i)_p\}_{p \in Z}$ has constant rank and form a vector bundle over $Z$, for every $\ell$ and every $i=0,1$.

We introduce a $\text{Cl}^0(T^*Z)$-invariant section of $\mathrm{End}(S^i\vert_Z)$ defined by the composition, 
\begin{equation}
\label{eq:IntroDefCp}
 \xymatrix{
  C^i: S^i\vert_Z \ar[r]^-{\nabla{\mathcal A}^i}  & T^*X\vert_Z \otimes E^{1-i}\vert_Z \ar[r]^-{\iota_N^* \otimes P^{1-i}} & N^* \otimes S^{1-i}\vert_Z  \ar[r]^-{- c} & S^i\vert_Z 
},
\end{equation}
where $\iota_N : N \hookrightarrow TX\vert_Z$ is the inclusion and ${\mathcal A}^0={\mathcal A},\ {\mathcal A}^1={\mathcal A}^*$. We set $C = C^0\oplus C^1$. It is proven in Proposition~\ref{prop:properties_of_compatible_subspaces} that $C^i$ is a symmetric operator that respects the decompositions \eqref{eq:eigenspaces_of_Q_i}. It is further proven that $C^i$ has eigenvalues $\{(n-m-2 k)\lambda_\ell\}_{k=0}^{n-m}$ and that the corresponding eigenspaces are $\text{Cl}^0(T^*Z)$-modules of constant rank. It is the eigenspaces with eigenvalues $\{\pm (n-m) \lambda_\ell\}_\ell$ that are important:

\begin{defn}
\label{defn:IntroDefSp}
For each  component  $Z$ of $Z_{\mathcal A}$ with $S\vert_Z= \ker{\mathcal A} \oplus \ker{\mathcal A}^*$ and every $k\in \{0, \dots, n-m\}$, let $S^i_{\ell k}$ denote the eigenspace of $C^i$ with eigenvalue  $(n-m - 2k) \lambda_\ell$ when $i=0,1$ and set $S_{\ell k } = S^0_{\ell k } \oplus S^1_{\ell k}$. In particular, when $k=0$ or $k = \dim N$, we define 
\[
S^{i\pm}_\ell = \{\text{eigenspace of $C^i$ with eigenvalue $\pm (n-m) \lambda_\ell$}\}, 
\]
and set
\[
S^{i\pm} = \bigoplus_\ell S^{i\pm}_\ell, \qquad S^\pm_\ell = S_\ell^{0\pm} \oplus S_\ell^{1\pm}, \qquad S^{\pm} := S^{0\pm} \oplus S^{1\pm},
\]
for every $i=0,1$ with corresponding bundles of orthogonal projections  $P^{i\pm},\  P^{i\pm}_\ell$, etc, where the projection indices follow the same placing as the indices of the spaces they project to. 
\end{defn}  

In particular, $S^\pm=S^{0\pm} \oplus S^{1\pm}$ are bundles of $\mathbb{Z}_2$-graded, $\text{Cl}(T^*Z)$-modules, over $Z$ with  Clifford multiplication defined as the restriction 
\[
c_Z: = c:  T^*Z \otimes S^{i\pm} \to S^{(1-i)\pm},\ i=0,1.
\]
They posses a natural connection, compatible with its Clifford structure given by,
\[
\nabla^\pm:=\sum_\ell P^{i\pm}_\ell\circ \nabla\vert_{TZ} : C^\infty(Z; S^{i\pm}) \to C^\infty(Z;  T^*Z\otimes S^{i\pm}).
\]

\begin{defn}
\label{defn:Dirac_operator_component}
On each $m$-dimensional component $Z\subset Z_{\mathcal A}$ we have a triple  $( S^\pm, \nabla^\pm, c_Z) $. We define a Dirac operator,
\[
D^Z_\pm \ :=\  c_Z\circ \nabla^\pm : C^\infty(Z; S^{0\pm}) \to  C^\infty(Z; S^{1\pm}). 
\]
It's formal $L^2$-adjoint is denoted by $D^{Z*}_\pm$.  Finally let ${\mathcal B}_{i\pm}^Z: S^{i\pm} \to S^{(1-i)\pm}$,
\[
{\mathcal B}_{i\pm}^Z:= \begin{cases}  \sum_{\ell, \ell'} C_{\ell, \ell'} P^{1\pm}_\ell\circ\left( {\mathcal B}^0 + \frac{1}{2(\lambda_\ell + \lambda_{\ell'})} \sum_\alpha \bar A_{\alpha\alpha} \right) \circ P^{0\pm}_{\ell'}, & \quad \text{when $i=0$},
\\
\sum_{\ell, \ell'} C_{\ell, \ell'} P^{0\pm}_\ell\circ\left( {\mathcal B}^1 + \frac{1}{2(\lambda_\ell + \lambda_{\ell'})} \sum_\alpha \bar A_{\alpha\alpha}^* \right) \circ P^{1\pm}_{\ell'}, & \quad \text{when $i=1$},
\end{cases}
 \]
where ${\mathcal B}^i$ are defined in expansions  of Lemma~\ref{lem:Dtaylorexp} and,
\[
C_{\ell, \ell'} = (2\pi)^{\tfrac{n-m}{2}} \frac{(\lambda_\ell \lambda_{\ell'})^{\tfrac{n-m}{4}}}{ (\lambda_\ell + \lambda_{\ell'})^{\tfrac{n-m}{2}}}.
\]
\end{defn}

\vspace{4mm}

The main result of this paper is a  converse of Proposition~\ref{prop:concentration_Prop}.  Recall that Proposition~\ref{prop:concentration_Prop} shows that, for each $C$, the  eigensections $\xi$ satisfying  $D_s^*D_s \xi = \lambda(s) \xi$ with $\lvert\lambda(s)\rvert\le C$, concentrate around $\bigcup_\ell Z_\ell$ for large $\lvert s\rvert$. Let $\mathrm{span}^0(s, K)$ be the span of the eigenvectors corresponding to eignevalues $\lambda(s) \leq K$  of $D^*_s D_s$ and denote the dimension of this space by $N^0(s, K)$. Similarly denote by $\mathrm{span}^1(s, K)$ and $N^1(s,K)$ the span and dimension of the corresponding eignevectors for the operator $D_s D^*_s$.  The following Spectral Separation Theorem  shows that these localized solutions can be reconstructed using local data obtained from $Z_{\mathcal A}$.  

\begin{spectrumseparationtheorem*}
\label{Th:mainT}
Suppose that $D_s=D+s{\mathcal A}$ satisfies Assumptions~\ref{Assumption:transversality1}-\ref{Assumption:stable_degenerations} above. Then there exist $\lambda_0>0$ and $s_0>0$ with the following significance: For every $s>s_0$, there exist vector space isomorphisms
\[ 
\mathrm{span}^0(s, \lambda_0) \,\overset{\cong} \longrightarrow\,\bigoplus_{Z\in \mathrm{Comp}(Z_{\mathcal A})} \ker (D^Z_+ + {\mathcal B}_{0+}^Z),
\]
and
\[
\mathrm{span}^1(s , \lambda_0) \overset{\cong}\longrightarrow\bigoplus_{Z\in \mathrm{Comp}(Z_{\mathcal A})} \ker (D^{Z*}_+ + {\mathcal B}^Z_{1+}),
\]
where $Z$ runs through all the set $\mathrm{Comp}(Z_{\mathcal A})$ of connected components of the singular set $Z_{\mathcal A}$. Furthermore $N^i (s, \lambda_0) = N^i(s, C_1 s^{-1/2})$, for every $s> s_0$ and every $i=0,1$, where $C_1$ is the constant of Theorem~\ref{Th:hvalue} \eqref{eq:est1}. Also, for every $s< -s_0$, the isomorphisms above hold by replacing $\ker (D^Z_+ + {\mathcal B}_{0+}^Z)$ with  $\ker (D^Z_- + {\mathcal B}_{0-}^Z)$ and  $ \ker (D^{Z*}_+ + {\mathcal B}_{1+}^Z)$ with  $\ker (D^{Z*}_- + {\mathcal B}_{1-}^Z)$ and $N^i (s, \lambda_0) = N^i(s, C_1 \lvert s\rvert^{-1/2}),\ i=0,1$ .
\end{spectrumseparationtheorem*}
 
As a corollary we get the following localization for the index : 

\begin{indexlocalizationtheorem*}
\label{Th:index_localization_theorem}
Suppose that $D_s=D+s{\mathcal A}$ satisfies Assumptions~\ref{Assumption:transversality1}-\ref{Assumption:stable_degenerations} above. Then the index of $D$ can be written as a sum of local indices as  
\[
\mathrm{index\,} D = \sum_{Z\in \mathrm{Comp}(Z_{\mathcal A})} \mathrm{index\,} D^Z_\pm.
\]
\end{indexlocalizationtheorem*}
\begin{proof}
By the Spectral Separation Theorem there exist $\lambda_0>0$ and $s_0>0$ so that for every $\lvert s\rvert>s_0$
\begin{align*}
\mathrm{index\,} D_s &= \dim \ker D_s - \dim \ker D_s^* 
\\
&= \dim \ker D^*_sD_s - \dim \ker  D_sD_s^* 
\\ 
&= \dim \mathrm{span}^0(s, \lambda_0) - \dim \mathrm{span}^1(s, \lambda_0)
\\ 
&= \sum_{Z\in \mathrm{Comp}(Z_{\mathcal A})}\mathrm{index\,} D^Z_\pm,    
\end{align*}
where the third equality holds because $D^*_sD_s$ and $D_sD_s^*$ have the same spectrum and their eigenspaces corresponding to a common non zero eigenvalue are isomorphic. We also use the fact that the index remains unchanged under compact perturbations. This finishes the proof.
\end{proof}

\begin{rem}
\begin{enumerate}
\item Due to the conclusion of its last sentence the Spectral Separation Theorem is a direct generalization of the localization theorem proved in \cite{pr}.  

\item By Lemma~\ref{lemma:whew} of the Appendix, the term $\sum_\alpha \bar A_{\alpha\alpha}$ in Definition~\ref{defn:Dirac_operator_component} can be assumed to be zero without affecting the generality of the statement of Theorem~\ref{Th:mainT}.

\item The non - degeneracy assumption~\ref{Assumption:normal_rates} can be weaken for the proof of the Spectral Separation Theorem. It is included for making the construction of the approximation simpler. In general one has to examine the various normal vanishing rates of the eigenvalues of ${\mathcal A}^*{\mathcal A}$ and correct the expansions in Corollary~\ref{cor:taylorexp} up to higher order in $r\to \infty$. The re-scaling $\{w_\alpha= \sqrt{s} x_\alpha\}_\alpha$ required to change the problem into a regular one will change multiplicity.  In particular the bundles of solutions examined here will have a layer structure corresponding to those rates and possibly will have a jumping locus in their dimension.
\end{enumerate}
\end{rem}

The proof of Spectral Separation Theorem relies on the existence of a gap in the spectrum of the operators $D^*_s D_s$ and $D_s D^*_s$ for $\lvert s \rvert >>0$. The proof of the existence of this gap relies in a splicing construction of approximate eigenvectors associated  to the lower part of the spectrum.  Inequalities of Theorem~\ref{Th:hvalue} are used to prove this association. In particular, inequality \eqref{eq:est2} is an analogue of a Poincare inequality and its proof is the heart of the argument. By Lemma~\ref{lemma:local_implies_global}, the proof of inequality \eqref{eq:est2} is essentially local in nature and one has to study the perturbative behavior of the density function $\lvert D_s \eta\rvert^2\, d\mathrm{vol}^X$ for sections $\eta$ supported in tubular neighborhoods of a component $Z$ of the critical set $Z_{\mathcal A}$. Using the maps \eqref{eq:exp_diffeomorphism},  the metric tensor $g_X$ and volume form $d\mathrm{vol}^X$ pulled back to $g = \exp^* g_X$ and $d\mathrm{vol}^\mathcal{N} = \exp^* d\mathrm{vol}^X$ in $TN\vert_\mathcal{N}$ respectively. 

\begin{defn}
\label{eq:tilde_D}
Set $\tilde D$ for the Dirac operator ${\mathcal I}^{-1} D {\mathcal I}$ and $\tilde{\mathcal A}$ for ${\mathcal I}^{-1} {\mathcal A} {\mathcal I}$ and $\tilde c$ for ${\mathcal I}^{-1} \circ c \circ ((I^*)^{-1} \otimes  {\mathcal I})$. Also let $\tilde\nabla^{T{\mathcal N}}$ and $\tilde\nabla^{\tilde E}$ denote the corresponding connections induced by the Levi-Civita connection on $TX$ and the Clifford connection on $E$ respectively.
\end{defn}

The association   
\[
C^\infty\left(B_Z(2\varepsilon\right); E\vert_{B_Z(2\varepsilon)}) \to  C^\infty( \mathcal{N} ; \tilde E),\ \eta \mapsto \tilde\eta:= {\mathcal I}^{-1} \xi \circ \exp,
\]
satisfies $\widetilde{D \eta} = \tilde{D} \tilde \eta$. Also 
\[
\int_{B_Z(2\varepsilon)} \lvert D_s \eta\rvert^2\, d \mathrm{vol}^X = \int_\mathcal{N} \lvert\tilde D_s \tilde \eta\rvert^2\, d \mathrm{vol}^\mathcal{N} = \int_\mathcal{N}\lvert\tau^{-1}(\tilde D_s \tilde \eta)\rvert^2\, d\mathrm{vol}^\mathcal{N},
\]  
where in the last equality, we used the parallel transport map $\tau$ introduced in Appendix~\ref{App:Taylor_Expansions_in_Fermi_coordinates} \eqref{eq:parallel_transport_map}. Recall the  volume element $d\mathrm{vol}^N$ introduced in Appendix~\ref{subApp:The_expansion_of_the_volume_from_along_Z}. By \eqref{eq:density_comparison}, the density $d\mathrm{vol}^\mathcal{N}$ can be replaced by the density $d\mathrm{vol}^N$ and we have to prove the inequalities of Theorem~\ref{Th:hvalue} for the density function $\lvert\tau^{-1}(\tilde D_s \tilde \eta)\lvert^2\, d\mathrm{vol}^N$. We study the perturbative behavior of the operator  $\tau^{-1}\circ\tilde D_s : C^\infty(\mathcal{N}; \pi^*(E^0\vert_Z)) \to C^\infty(\mathcal{N}; \pi^*(E^1\vert_Z))$ in two parts: in Section~\ref{Sec:structure_of_A_near_the_singular_set} we analyze the infinitesimal data of the perturbation term $\tilde {\mathcal A}$ and in Section~\ref{sec:structure_of_D_sA_along_the_normal_fibers} we analyze the perturbative behavior of $\tau^{-1}\circ\tilde D_s$. 

\vspace{10 mm}

\medskip
   
  \vspace{1cm}

\section{Structure of \texorpdfstring{${\mathcal A}$}{} near the singular set}
\label{Sec:structure_of_A_near_the_singular_set}

In proving the Spectral Separation Theorem  we will have to analyze the geometry of the operator 
\[
\tilde D_s = \tilde c\circ \tilde\nabla + s \tilde{\mathcal A}: C^\infty(\mathcal{N}; \tilde E^0)\rightarrow C^\infty(\mathcal{N};\tilde E^1),
\]
near the singular set $Z$, that is the corresponding connected component of $Z_{\mathcal A}$ with tubular neighborhood $\mathcal{N}\subset N$. The idea is to expand into Taylor series along the normal directions $Z$.

\bigskip

\subsection{The data from the 1-jet of \texorpdfstring{$\tilde{\mathcal A}$}{} and \texorpdfstring{$\tilde{\mathcal A}^*\tilde{\mathcal A}$}{} along \texorpdfstring{$Z$}{}.}

Fix a connected component $Z$ of the critical set $Z_{\mathcal A}$, an $m$-dimensional submanifold of the $n$-dimensional manifold $X$. Recall that $\pi:N \rightarrow Z$ is the normal bundle of $Z$ in $X$ and set $S(N) \rightarrow Z$ to be the normal sphere bundle.  Our first task is to understand the perturbation term $\tilde{\mathcal A}$ on a tubular neighborhood $\mathcal{N}$ of $Z$. Note that $\tilde {\mathcal A}\vert_Z = {\mathcal A}\vert_Z$. 

\begin{defn}
\label{defn:kernel_bundles}
We introduce the rank $d$-subbundles,
\[
S_p^0:=\ker {\mathcal A}_p\subset E^0_p,\qquad  S^1_p : =\ker {\mathcal A}^*_p\subset E^1_p \qquad \text{and}\qquad S_p = S^0_p \oplus S^1_p,
\]
as $p$ runs in $Z$, and expand them by parallel transport, along the normal radial geodesics of $Z$, to  bundles $S^0, S^1 \to \mathcal{N}$, defined over a sufficiently small tubular neighborhood $\mathcal{N}$ of $Z$. We denote by $P^i$ and $P_\perp^i$,  the orthogonal projections to $S^i$ and its complement $(S^i)^\perp$, respectively. 
\end{defn}

Since $\mathrm{index\,} {\mathcal A}=0$, the bundles $S^i,\ i=0,1$ are of equal dimension and a consequence of the concentration condition \eqref{eq:cond} is
\begin{align}
\label{eq:spin_action_on_S}
u_\cdot S^i=  S^{1-i}, \qquad  u_\cdot (S^i)^\perp =  (S^{1-i})^\perp, 
\end{align}
for every $i=0,1$ and every $u\in T^*X\vert_\mathcal{N}$. Therefore the  bundles $S^0\oplus S^1$ and $(S^0)^\perp \oplus (S^1)^\perp$ are  both $\mathbb{Z}_2$- graded $\text{Cl}^0(T^*X\vert_Z)$-modules. By derivating relations \eqref{eq:cond} and \eqref{eq:cond_version2} we also get
\begin{align}
\label{eq:dcond}
v_\cdot \nabla_u {\mathcal A} = - \nabla_u{\mathcal A}^* v_\cdot \qquad \mbox{and}\qquad  \nabla_u {\mathcal A}\,  v_\cdot= - v_\cdot\nabla_u{\mathcal A}^*, 
\end{align}
for every $u\in N,\  v\in T^*X\vert_Z$. Since ${\mathcal A}$ is transverse to  ${\mathcal L}\cap {\mathcal F}^\ell$, along the normal directions of $Z_\ell = Z \subset X$, by Proposition~\ref{prop:properties_of_perturbation_term_A} \eqref{eq:transversality_assumption_with_respect_to_connection}, we have
\[
\nabla_u{\mathcal A}(S^0)\subseteq S^1 \qquad \mbox{and}\qquad \nabla_u{\mathcal A}^*(S^1)\subseteq S^0,
\]
for every $u \in N$. Using the preceding relations with equations  \eqref{eq:spin_action_on_S}, we make the following definition:

\begin{defn}
Let $S(N) \rightarrow Z$ be the normal sphere bundle of $Z$ in $X$. For every $v\in N$ let $v^*$ be its algebraic dual. The bundle maps, 
\begin{equation}
\begin{aligned}
\label{defn:IntroDefM}
M^i: S(N) &\rightarrow \mathrm{End}(S^i\vert_Z),\quad   v\mapsto  M_v^i := \begin{cases}- v^*_\cdot \nabla_v{\mathcal A}, & \quad \text{if $i=0$,}
\\ 
 - v^*_\cdot \nabla_v{\mathcal A}^*, & \quad \text{if $i=1$,} \end{cases}
\\
M:  S(N) &\rightarrow \mathrm{End}(S\vert_Z),\quad  v\mapsto  M_v^0 \oplus M_v^1,
\end{aligned}
\end{equation}
will be of the utmost importance.
\end{defn}

\begin{proposition}
\label{prop:spectrum_and_eigenspaces_as_Clifford_submodules}
Assume that $n = \dim X > \dim Z= m>0$. Fix $v\in S(N)$ and $w\in S(TX)\vert_Z$ perpendicular to $v$ and set $p=\pi(v) \in Z$. Given an eigenvalue $\lambda_v:=\lambda$ of $M^i_v$ we denote by $S_\lambda^i \subset S^i_p$, the associated eigenspace and set $S_\lambda = S_\lambda^0 \oplus S_\lambda^1$. The matrices $M^i_v,\ i=0,1$ enjoy the following properties:
\begin{enumerate}
\item The matrix $M_v^0$ is symmetric with spectrum a symmetric finite set around the origin. Moreover, the map 
\[
v^*_\cdot w^*_\cdot :S^0_\lambda \to S_{-\lambda}^0, 
\]
is an isomorphism.

\item  The matrices $M_v^i,\ i=0,1$ have the same spectrum and  the Clifford multiplication
\[
w^*_\cdot : S_\lambda^i \to S^{1-i}_\lambda \quad \text{and}\quad v^*_\cdot  : S_\lambda^i \to  S_{-\lambda}^{1-i},
\]
induce isomorphisms for every $i=0,1$.

\item We have the following submodules of $S_p$: 
\begin{align*}
S_\lambda\oplus S_{-\lambda} &\qquad \text{ is a $\text{Cl}(T^*_pX)$-module,}  
\\
S_\lambda^i\oplus S_{-\lambda}^i &\quad \text{ is a $\text{Cl}^0(N^*_p)$-module,}  
\\
S^i_\lambda \oplus S^{1-i}_\lambda &\quad \text{ is a $\mathbb{Z}_2$-graded $\text{Cl}(T^*Z_p)$-module,} 
\end{align*}
for every $i=0,1$.
\end{enumerate}
\end{proposition}

\begin{proof}
Let $\{e^A\}_{A= j, \alpha}$ be an orhonormal basis of $T^* X_p$ so that $e^1= v^*$ and $e^2=w^*$. For an ordered string $1\leq i_1 < \dots < i_k \leq n$ we set $I = (i_1, \dots, i_k),\ \lvert I\rvert = k$ and $e^I = e^{i_1}_\cdot \dots e^{i_k}_\cdot$. The proof of the proposition follow by the identity,
\begin{equation}
\label{eq:basic_identity}
M_v^i e^I = \begin{cases} e^I M_v^i,& \quad \text{if $\lvert I\rvert$ is even and $1\notin I$},
\\
 - e^I M_v^i,& \quad \text{if $\lvert I\rvert$ is even and $1\in I$},
\\
 e^I M_v^{1-i},& \quad \text{if $\lvert I\rvert$ is odd and $1\notin I$},
\\
-e^I M_v^{1-i},& \quad \text{if $\lvert I\rvert$ is odd and $1\in I$}.
\end{cases}
\end{equation}
The identity is a direct consequence of \eqref{eq:Clifford_relations}, \eqref{eq:dcond} and the definitions of $M^i_v$ in \eqref{defn:IntroDefM}. 
\end{proof}

 We now study the properties of the matrix functions $S(N)\ni v\mapsto M^i_v \in \mathrm{End}(S^i\vert_Z)$. By \cite{k}[Ch.II, Th. 6.8, pp.122], we can always choose a family of eigenvalues (possibly with repetitions) that are smooth functions $\{\lambda_\ell :S(N) \to \mathbb{R} \}_{\ell=1}^d$ so that $M^0_v- \lambda_\ell(v) 1_{S^0\vert_Z}$ is a singular matrix for every $v\in S(N)$ and every $\ell$. Associated to each such $\lambda_\ell$ is a bundle of eigenspaces $\{(S_\ell^0)_v\}_{v\in S(N)}$ that is a subbundle of $\pi^*(S^0\vert_Z)$ with varying dimensions in the fibers, and its jumping locus in $S(N)$ located where the different graphs of the eigenvalues intersect. Assumptions~\ref{Assumption:normal_rates} and \ref{Assumption:stable_degenerations} guarantee two main properties: 1) the bundles of eigenspaces have constant ranks and are pullbacks under $\pi: S(N)\to Z$ of eigenbundles defined on the core set $Z$ and 2) the eigenvalues are functions of the core set $Z$. The following couple of propositions explain why this is the case. The following second order term in the expansion of $\tilde{\mathcal A}^*\tilde{\mathcal A}$ along $Z$ will be important: 

\begin{defn}
\begin{enumerate}
\item We introduce the bundle map,
\begin{align*} 
M^0: N\otimes N &\to \mathrm{Sym}(S^0),\quad  v\otimes w \mapsto \left.\frac{1}{2}\left(\nabla_v{\mathcal A}^*\nabla_w{\mathcal A} + \nabla_w{\mathcal A}^*\nabla_v{\mathcal A}\right)\right\vert_{S^0\vert_Z},
\\
M^1: N\otimes N &\to \mathrm{Sym}(S^1),\quad  v\otimes w \mapsto \left.\frac{1}{2}\left(\nabla_v{\mathcal A} \nabla_w{\mathcal A}^* + \nabla_w{\mathcal A} \nabla_v{\mathcal A}^*\right)\right\vert_{S^1\vert_Z}. 
\end{align*}
We also introduce 
\[
M: N\otimes N \to \mathrm{Sym}(S^0)\oplus \mathrm{Sym}(S^1),\quad  v\otimes w  \mapsto  M^0_{v,w} \oplus M^1_{v,w}.
\]

\item We say that $\{M^i_{v,w}\}_{v,w\in S(N)}$ is compatible with the inner product $g_X$ or simply compatible, if 
\begin{equation}
\label{eq:compatibility_assumption_2}
M^i_{v,w} \equiv 0 \quad \text{whenever} \quad g_X(v,w)=0,
\end{equation}
for every $v,w\in N$. If both $M^0$ and $M^1$ are compatible, we say that $M$ is compatible.
\end{enumerate}
\end{defn}

\begin{lemma}
\label{lem:basic_properties_of_M_v_w}
The term $M_{v,w}$ satisfies the following identities:
 \begin{enumerate}
\item For every $v\in S(N)$, we have  $M_{v,v} = M_v^2$. 

\item For every $v,w\in N$ the equation,
\begin{equation}
\label{eq:commutator_vs_second_order_term}
[M_v, M_w] =  2 w^*_\cdot v^*_\cdot (M_{v,w} - g(v,w) M_v M_w),
\end{equation}
holds.

\item  For every $u\in TX\vert_Z,\ v,w\in N$ and every $i=0,1$, the relation,
\begin{equation}
\label{eq:quadratic_term_vs_adjoint_quadratic_term}
u^*_\cdot M^i_{v,w} = M^{1-i}_{v,w} u^*_\cdot,
\end{equation}
holds.

\item Assumption~\ref{Assumption:normal_rates} is equivalent to  $M^0_{v,w}$ being compatible and to $M^0_{v,v}$ being invertible  for some $v\in S(N)$. 

\item If $\tilde{\mathcal A}^*\tilde{\mathcal A}$ satisfies  Assumption~\ref{Assumption:normal_rates} then so does $\tilde{\mathcal A} \tilde{\mathcal A}^*$ that is, there exist a  positive-definite symmetric endomorphism $Q^1$ of the bundle $S^1$, so that 
 \[
\left.\tilde{\mathcal A}\tilde{\mathcal A}^*\right\vert_{S^1} = r^2\left(Q^1 + \left.\frac{1}{2}\bar{\mathcal A}\bar A^*_{rr}\right\vert_{S^1}\right)+ O(r^3).
\]
Furthermore 
\begin{equation}
\label{eq:Q0_vs_Q1}
u^*_\cdot Q^i = Q^{1-i}u^*_\cdot,
\end{equation}
and $Q^0$ and $Q^1$ share the same spectrum and they are $\text{Cl}^0(T^*X)$-equivariant.
\end{enumerate}
\end{lemma}
\begin{proof}
All the identities are direct consequences of  \eqref{eq:Clifford_relations} and \eqref{eq:dcond}. For the last couple of entries, working in a Fermi chart $(\mathcal{N}_U, (x_j, x_\alpha)_{j,\alpha})$ with $v = \sum_\alpha \tfrac{x_\alpha}{r}\partial_\alpha \in S(N)$, we have that
\[
M^0_{v,v} =  \sum_{\alpha, \beta}\frac{x_\alpha x_\beta}{r^2} M^0_{\alpha, \beta} =\sum_{\alpha, \beta}\frac{x_\alpha x_\beta}{r^2}\left. \bar A_\alpha^*  \bar A_\beta\right\vert_{S^0}. 
\]
By comparing terms in expansion of  Proposition~\ref{prop:properties_of_perturbation_term_A} \eqref{eq:jet1}, it follows that 
Assumption~\ref{Assumption:normal_rates} is equivalent to $M^0_{v,v} = Q^0$ for every $v\in S(N)$ for $Q^0$ a positive definite matrix, a quadratic relation in $S(N)$. By polarization, this is equivalent to $M^0_{v,w}$ satisfying condition \eqref{eq:compatibility_assumption_2} with $i=0$ and to $M^0_{v,v}$ being invertible  for some $v\in S(N)$. Equation \eqref{eq:quadratic_term_vs_adjoint_quadratic_term} then shows that $M^0_{v,w}$ being compatible is equivalent to $M^1_{v,w}$ being compatible. It follows that  if ${\mathcal A}^*{\mathcal A}$ satisfies  Assumption~\ref{Assumption:normal_rates} the so does $\tilde{\mathcal A} \tilde{\mathcal A}^*$ and there exist a  positive-definite symmetric endomorphism $Q^1$ of the bundle $S^1$, so that 
 \begin{equation*}
\left.\tilde{\mathcal A}\tilde{\mathcal A}^*\right\vert_{S^1} = r^2\left(\left.Q^1 + \frac{1}{2}\bar{\mathcal A}\bar A^*_{rr}\right\vert_{S^1}\right)+ O(r^3).
\end{equation*}
Finally for $v=w$, equation \eqref{eq:quadratic_term_vs_adjoint_quadratic_term} becomes \eqref{eq:Q0_vs_Q1}. The last couple of assertions follow then from this resulting equation.
\end{proof}

The following proposition shows that Assumption~\ref{Assumption:normal_rates}, implies that the eigenbundles of $M_v$ are pullbacks under $\pi:S(N)\to Z$ of eigenbundles defined on the core set $Z$ and that the corresponding eigenvalues are functions of the core set $Z$:

\begin{prop}
\label{prop:more_properties}
We use the notation from Proposition~\ref{prop:spectrum_and_eigenspaces_as_Clifford_submodules}.
\begin{enumerate}
\item  \label{lem:commutativity_of_M_v_w}
The family $\{M_{v,w}\}_{v,w\in S(N)}$ commutes if and only if the subfamily $\{M_v^2\}_{v\in S(N)}$ commutes and that happens if and only if, the direct sum of the eigenbundles $S_\lambda\oplus S_{-\lambda} \to S(N)$ can be pushed forward to subbundles of the bundle $S\vert_Z\to Z$, for every eigenvalue $\lambda: S(N) \to \mathbb{R}$ of the family $\{M_v: v\in S(N)\}$. 

\item \label{rem:metric_compatibility1}
If $\{M_{v,w}\}_{v,w \in S(N)}$ is compatible with $g_X$, then there exist $Q: Z \to \mathrm{Sym}\mathrm{End}(S^0) \oplus \mathrm{Sym}\mathrm{End}(S^1)\vert_Z$ so that
\[
M_{v,w} =  g_X(v,w)Q, \qquad \text{for every $v,w \in N$}.
\]

\item \label{rem:metric_compatibility2}
Assume the squares $\{M_v^2: v \in S(N)\}$ are invertible matrices,  independent of $v\in S(N)$ but dependent on $\pi(v)$. Call their common matrix value by $Q$.  Then $Q$ is a symmetric matrix and the family $\{M_v\}_{v\in S(N)}$ respects the eigenspaces of $Q$. Moreover, if $\lambda: S(N) \to \mathbb{R}$ is an eigenvalue of this family then $\lambda$ is a function on the core set $Z$.
\end{enumerate}
\end{prop}

\begin{proof}
In proving \eqref{lem:commutativity_of_M_v_w}, let $v,w\in S(N)$ and set $e= (v+w)/\lvert v+w\rvert$ another unit vector. By polarization of the identity $M_{u,u} = M^2_u$, we obtain 
\[
M_{v,w} = \frac{1}{2}(\lvert v+w\rvert^2 M^2_e -M^2_v - M^2_w). 
\]
Hence, if the family $\{M_v^2\}_{v\in S(N)}$ commutes, then the family $\{M_{v,w}\}_{v,w \in S(N)}$ commutes as well. On the other hand, the family $\{M_v^2\}_{v\in S(N)}$ commutes, if and only if the matrices in the family are simultaneously diagonalizable. Since $M_v^2$ is a direct sum of the $M_v^i, i=0,1$, these matrices satisfy analogue commutativity properties. Thus, given an eigenvalue $\lambda = \lambda(v)$ of $M_v^i$, its eigenspace $S_\lambda^i\oplus S_{-\lambda}^i$, is common for all $(M^i_w)^2$ with $\pi(w)=\pi(v)\in Z$. 

Assume now that $\{M_{v,w}\}_{v,w\in S(N)}$ satisfies condition \eqref{eq:compatibility_assumption_2}. Given $v,w\in S(N)$ perpendicular vectors, we get another pair of unit perpendicular vectors $e^{\pm}= (v \pm  w)/\sqrt{2}$ and by the definition of compatibility we have 
\[
0=M_{e^+, e^-} = \frac{1}{2} (M^2_v - M^2_w) \quad \Longrightarrow \quad M_v^2 = M^2_w. 
\]
It follows that the family $\{M^2_v\}_{v\in S(N)}$ consists of only one positive definite symmetric matrix $Q: S\vert_Z \to S\vert_Z$. Given an orthonormal basis $\{e_\alpha\}$ of $N$ and arbitrary unit vectors $v= \sum_\alpha v_\alpha e_\alpha$ and $w= \sum_\alpha w_\alpha e_\alpha$, we obtain
\[
M_{v,w}= \sum_{\alpha, \beta} v_\alpha w_\beta M_{\alpha, \beta} = \sum_\alpha v_\alpha w_\alpha M_\alpha^2 =  g_X(v,w)Q,
\]
which is \eqref{rem:metric_compatibility1}

Finally the first part of \eqref{rem:metric_compatibility2} follows since $[Q, M_v] = M_v^3 - M_v^3 =0$, for every $v\in S(N)$.  Given  an eigenvalue $\lambda: S(N) \to \mathbb{R}$ of $\{M_v\}_{v\in S(N)}$ we have that $\lambda^2$ is an eigenvalue of $Q$ and therefore a non-vanishing function on the core set $Z$. But $\lambda$ is continuous and maintain sign, as function of $S(N)$ and therefore it is a function on the core set $Z$.   
\end{proof}

From now on we assume Assumptions~\ref{Assumption:normal_rates} and \ref{Assumption:stable_degenerations}. It is evident from Proposition~\ref{prop:more_properties} \eqref{rem:metric_compatibility2} that the family $\{M_v\}_{v\in S(N)}$ respects the decomposition  \eqref{eq:eigenspaces_of_Q_i} into eigenbundles of $Q$. Recall the definitions of the bundles $S^i_{\ell k}$ in Definition~\ref{defn:IntroDefSp}. We proceed to describe the structure of a single $S_\ell = S^0_\ell \oplus S^1_\ell$ for fixed eigenvalue $\lambda_\ell:Z \to (0, \infty)$:

\begin{prop}[Structure of compatible subspaces]
\label{prop:properties_of_compatible_subspaces}
Assume $0<m=\dim Z< n = \dim X$.

\begin{enumerate}
\item 
We have an alternate description of $S_\ell^{i\pm}$ as,
\begin{equation}
\label{defn:positive_negative_espaces1}
S_\ell^{i\pm} = \bigcap_{v\in S(N)}\{ \xi\in S^i_\ell: M_v^i\xi = \pm\lambda_\ell \xi\},
\end{equation}
for every $i=0,1$.

\item 
 $S^\pm_\ell$ are both, $\mathbb{Z}_2$-graded, $\text{Cl}(T^*Z)$-modules. Furthermore
\begin{equation}
\label{prop:structure_of_compatible_subspaces2}
\bigcap_{v,w \in S(N)} \ker[M^i_v, M^i_w] = S_\ell^{i+} \oplus S_\ell^{i-}, 
\end{equation}
for every  $i=0,1$.
\item 
\label{prop:structure_of_compatible_subspaces3}
We view $\text{Cl}(N^*)$ as an irreducible $\mathbb{Z}_2$-graded $\text{Cl}(N^*)$-left module  and $S^+_\ell = S_\ell^{0+} \oplus S_\ell^{1+}$ as an irreducible $\mathbb{Z}_2$-graded $\text{Cl}(T^*Z)$-module. Then there is a bundle isomorphism of $\text{Cl}(N^*)\hat\otimes \text{Cl}(T^*Z)$-modules
\begin{equation}
\label{eq:graded_tensor_product_decomposition}
c_\ell :   \text{Cl}(N^*) \hat\otimes S^+_\ell\to S_\ell,
\end{equation}
induced by the the Clifford multiplication of $S$ as a $\text{Cl}(T^*X\vert_Z)$-module, where $\hat \otimes$ represents the $\mathbb{Z}_2$-graded tensor product. Moreover the principal $\mathrm{SO}(\dim S_\ell)$- bundle of $S_\ell$ is reduced to the product of the principal $\mathrm{SO}$-bundles of $N$ and $S^+_\ell$ respectively.

\item  Under the identification of bundles $\Lambda^* N^* \simeq \text{Cl}(N^*)$, the bundle isomorphism \eqref{eq:graded_tensor_product_decomposition} restricts to a bundle isomorphism
\[
c_\ell: \Lambda^k N^* \otimes S^{+ i}_\ell \to \begin{cases} S^i_{\ell k},& \quad \text{if $k$ is even,}
\\
S^{1-i}_{\ell k},  & \quad \text{if $k$ is odd,}
\end{cases}
\]
for every $i=0,1$. In particular, there is an orthogonal decomposition 
\begin{equation}
\label{eq:decompositions_of_S_ell}
S^i_\ell =\left( \bigoplus_{k\, \text{even}} S^i_{\ell k} \right) \oplus \left( \bigoplus_{k\, \text{odd}} S^{1-i}_{\ell k} \right), 
\end{equation}
for every $i=0,1$.
\end{enumerate} 
\end{prop}

\begin{proof}
Throughout the proof, we fix an orthonormal basis $\{e_\alpha\}_\alpha$ of $N$ with dual basis $\{e^\alpha\}_\alpha$, we work with $i=0$ and set $M_\alpha^0:= M^0_{e_\alpha}$. Similar proofs holds for $i=1$. 
Combining equation from  Proposition~\ref{prop:more_properties} \eqref{rem:metric_compatibility1} with equation \eqref{eq:commutator_vs_second_order_term}, we have that,
\begin{equation}
\label{eq:com_1}
[ M_v^0, M_w^0 ] =  2 w^*_\cdot v^*_\cdot g(v, w)(Q^0 - M_v^0 M_w^0) = 2 w^*_\cdot v^*_\cdot g(v, w)M_v^0(M_v^0 -  M_w^0),
\end{equation}
for any $v,w\in S(N)$. Hence
\begin{equation}
\label{eq:dihotomy}
\xi \in \ker[M_v^0, M_w^0] \qquad \Longleftrightarrow \qquad v\perp w \quad \text{or if}\quad  M_v^0 \xi = M_w^0\xi.
\end{equation}
In particular we have that $[M_\alpha^0, M_\beta^0] = 0$ for every $\alpha, \beta$ so that the family $\{M_\alpha^0\}_\alpha$ commutes. Recall the bundle map $C^0$ introduced in \eqref{eq:IntroDefCp}. By Proposition~\ref{prop:properties_of_perturbation_term_A} \ref{eq:transversality_assumption_with_respect_to_connection}, $C^0 =  \sum_\alpha M_\alpha^0$. In particular
\[
[C^0, M_\alpha^0]= \sum_\beta [M_\beta^0, M_\alpha^0] =0,
\]
and since $v= e_\alpha$ is an arbitrary vector completed to a basis of $N$, we have that $[C^0, M_v^0]=0$ and therefore $M^0_v$ preserves the eigenspaces of $C^0$, for every $v \in S(N)$. We calculate these eigenspaces by using the description of $C^0$ with respect to the orthonormal base $\{e_\alpha\}_\alpha$.

The members of the family $\{M_\alpha^0\}_\alpha$ are commuting symmetric matrices that preserve $S_\ell^0$ so that 
\[
(M^0_\alpha)^2 = Q^0\rvert_{S_\ell^0} = \lambda_\ell^2 1_{S^0_\ell}. 
\]
Hence there exist a common eigenvector $\eta \in S^0_\ell$ and a string $I$, with 
\[
M_\alpha^0 \eta =\begin{cases} \lambda_\ell \eta,& \quad \text{if $\alpha\notin I$,}\\ - \lambda_\ell \eta,& \quad \text{if $\alpha\in I$.} \end{cases}
\]
Recall that for every string $I= (\alpha_1, \cdots,\alpha_h)$ of ordered positive integers we denote $e^I_\cdot = e^{\alpha_1}_\cdot \dots e^{\alpha_h}_\cdot \in \text{Cl}(N^*)$. Define
\begin{equation}
\label{eq:xi_minus_from_xi_plus}
\xi^{0+} = \begin{cases} e^I_\cdot \eta,& \quad \text{if $\lvert I\rvert$ is even,}\\  e^I_\cdot e_\cdot \eta,& \quad \text{if $\lvert  I\rvert$ is odd,} \end{cases} \quad \text{and}\quad  \xi^{0-} = \begin{cases} d\mathrm{vol}^N_\cdot \xi^{0+},& \quad \text{if $\dim N$ is even,}\\  d\mathrm{vol}^N_\cdot e_\cdot \xi^{0+},& \quad \text{if $\dim N$ is odd,} \end{cases}
\end{equation}
so that, by \eqref{eq:basic_identity}, $M_\alpha^0 \xi^{0\pm} = \pm \lambda_\ell \xi^{0\pm}$, for every $\alpha$, where $e\in T^*Z$ is a unit covector (since $m>0$). In particular $\xi^{0\pm} \in S^{0\pm}_\ell$ so that $S^{0\pm}_\ell$ are nontrivial subspaces of $S^0_\ell$ and therefore eigenspaces of $C^0$.

Given the string $I= (\alpha_1, \cdots,\alpha_h)$, we define a $\text{Cl}^0(T^*Z)$-module by
\[
S^0_{\ell I} := \bigcap_{j=1}^{\lvert  I\rvert}\{\xi \in S^0_\ell: M_{\alpha_j}^0 \xi =  -\lambda_\ell \xi\},
\]
for every $1\leq \lvert  I\rvert \leq \dim N-1$. Notice that $S^0_{\ell I}$ and $S^0_{\ell J}$ are orthogonal when $I\neq J$. A similar construction holds for $S^1_\ell$, constructing $S^1_{\ell I}$. Finally, there are isometries 
\[
A_I : = \begin{cases}e^I_\cdot : S^{i+}_\ell \to S^i_{\ell I} &\quad \text{when $\lvert  I\rvert$ is even,}\\  e^I_\cdot : S^{i+}_\ell \to S^{1-i}_{\ell I}  &\quad \text{when $\lvert  I\rvert$ is odd,} \end{cases}
\]
for every $i=0,1$. It follows that the $S^i_{\ell I}$ eigenspaces are nontrivial, they all have equal dimensions and we obtain decompositions  
\begin{equation}
\label{eq:decomposition_of_S_0_given_frame_e_a}
S^i_\ell\vert_Z  = \bigoplus_I  S^i_{\ell I} = \bigoplus_{\{I: \lvert  I\rvert= \text{even}\}} A_I S^{i+}_\ell \oplus \bigoplus_{\{I: \lvert  I\vert= \text{odd}\}} A_I S^{(1- i)+}_\ell, 
\end{equation}
for every $i=0,1$. This decomposition depends on our choice of frame $\{e^\alpha\}_\alpha$. However, for each $0\leq k \leq n-m$, the component $  \bigoplus_{\{I:\lvert  I\rvert=k\}}  S^i_{\ell I}$ is the eigenspace $S^i_{\ell k}$ of the matrix $C^i$, corresponding to the eigenvalue $(n-m-2k)\lambda_\ell$. Therefore for every $k$ the component is independent of the choice of the frame. Decomposition \eqref{eq:decompositions_of_S_ell} follows. When $k=0$ or $k=\dim N$, then every $\xi \in S^0_{\ell k}$ satisfies 
\[
M^0_\alpha \xi = \begin{cases} \lambda_\ell \xi, & \quad \text{if $k=0$,} \\  -\lambda_\ell \xi , & \quad \text{if $k=\dim N$,} \end{cases}
\]
for every $\alpha$. Given $v= \sum_\alpha v_\alpha e_\alpha \in N$ with $\lvert  v\rvert=1$ and using the expansion, 
\begin{equation}
\label{eq:M_v_expansion}
M_v = \sum_\alpha v_\alpha^2M_\alpha + \sum_{\alpha<\beta} v_\alpha v_\beta e^\alpha_\cdot e^\beta_\cdot (M_\alpha - M_\beta), 
\end{equation}
we obtain 
\[
M_v^0 \xi = \begin{cases} \lambda_\ell \xi, & \quad \text{if $k=0$,} \\  -\lambda_\ell \xi , & \quad \text{if $k=\dim N$,} \end{cases}
\]
for every $v\in S(N)$. This  proves \eqref{defn:positive_negative_espaces1}. 

By \eqref{eq:dihotomy} 
\[
\xi \in \bigcap_{v,w\in S(N)} \ker[M^0_v, M^0_w]  \qquad \Longleftrightarrow \qquad  \{M_v^0\xi\}_{v\in S(N)} \quad \text{is a singleton.}
\]
This implies the first inclusion in \eqref{prop:structure_of_compatible_subspaces2}. For the reverse inclusion of \eqref{prop:structure_of_compatible_subspaces2}, assume $\xi \in S^0_\ell$ so that $[M_v^0, M^0_w]\xi=0$ for all $v,w$. Using decomposition \eqref{eq:decomposition_of_S_0_given_frame_e_a}, there exist $a_I^i\in \mathbb{R}$ and $\xi_I^i\in S^{i+}_\ell$, linearly independent vectors, so that
\[
\xi = \sum_{\lvert  I\rvert\ \text{even}} a^0_I e^I_\cdot \xi^0_I +  \sum_{\lvert  I\rvert\ \text{odd}} a^1_I e^I_\cdot \xi^1_I. 
\]
Clearly, whenever $\lvert  I\rvert\neq 0 , \dim N$ is even and $a_I \neq 0$, we have that there exist $a\in I$ and $\beta\notin I$. But then 
\[
M^0_\alpha e^I_\cdot \xi^0_I = \lambda_\ell e^I_\cdot \xi^0_I\quad \text{and} \quad M^0_\beta e^I_\cdot \xi^0_I = -\lambda_\ell e^I_\cdot \xi^0_I,
\]
so that $\{M_\gamma^0\xi\}_\gamma$ is not singleton and cannot belong to the left hand side of \eqref{prop:structure_of_compatible_subspaces2}. A similar statement is true when $\lvert  I\rvert\neq 0 , \dim N$ and $\lvert  I\rvert$ is odd. Therefore $\lvert  I\rvert=0$ or $\lvert  I\rvert= \dim N= n-m$ and
\[
\xi = \begin{cases} a_0^0 \xi_0^0 + a_{n-m}^0 d\mathrm{vol}^N_\cdot\xi^0_{n-m} ,& \quad \text{if $n-m$ is even,}
\\
a_0^0 \xi_0^0 + a_{n-m}^1 d\mathrm{vol}^N_\cdot \xi^1_{n-m} ,& \quad \text{if $n-m$ is odd.}
\end{cases}
\]
Observe that in each case $ d\mathrm{vol}^N_\cdot\xi^0_{n-m},\  d\mathrm{vol}^N_\cdot \xi^1_{n-m}\in S^{0-}_\ell$. This finishes the proof of the other inclusion in \eqref{prop:structure_of_compatible_subspaces2}.

To prove \eqref{eq:graded_tensor_product_decomposition} we observe the inclusion bundle maps,
\begin{equation*}
(\text{Cl}^0(N^*) \otimes S^{0+}_\ell) \oplus (\text{Cl}^1(N^*)\otimes S^{1+}_\ell) \hookrightarrow 
(\text{Cl}^0(T^*X\vert_Z) \otimes S^{0+}_\ell) \oplus (\text{Cl}^1(T^*X\vert_Z)\otimes S^{1+}_\ell ) \rightarrow S^0_\ell\vert_Z,
\end{equation*}
and
\begin{equation*}
(\text{Cl}^0(N^*) \otimes S^{1+}_\ell) \oplus (\text{Cl}^1(N^*)\otimes S^{0+}_\ell) \hookrightarrow 
(\text{Cl}^0(T^*X\vert_Z) \otimes S^{1+}_\ell) \oplus (\text{Cl}^1(T^*X\vert_Z)\otimes S^{0+}_\ell ) \rightarrow S^1_\ell\vert_Z,
\end{equation*}
define isomorphisms on the fibers as a consequence of the decomposition \eqref{eq:decompositions_of_S_ell}  and its analogue for $S^1_\ell\vert_Z$. Under the identification $\Lambda^* N^* \simeq \text{Cl}(N^*)$ we have that $\Lambda^k N^*\otimes S^{0+}_\ell$ is a bundle isomorphic to the eigenbunlde $S^0_{\ell k}$, for every $0\leq k \leq n-m$. This finishes the proof of \eqref{eq:graded_tensor_product_decomposition} and \eqref{eq:decompositions_of_S_ell} and the proof of the proposition.  
\end{proof}

\medskip

\subsection{Connection \texorpdfstring{$\bar\nabla$}{}  emerges from the 1-jet of the connection \texorpdfstring{$\bar \nabla^{E\vert_Z}$}{} along \texorpdfstring{$Z$}{}.}

Recall now the Fermi coordinates $(\mathcal{N}_U, (x_j, x_\alpha)_{j,\alpha})$ and the frames $\{e_j, e_\alpha\}_{j,\alpha}$ of $N\vert_U$, centered at $p\in Z$, introduced in Appendix~\ref{App:Fermi_coordinates_setup_near_the_singular_set}. Let $\{\sigma_\ell\}_\ell,\ \{f_k\}_k$ orthonormal frames trivializing $E^0\vert_U$ and $E^1\vert_U$. Using Proposition~\ref{prop:properties_of_compatible_subspaces}, we choose the frames $\{\sigma_\ell\}_\ell,\ \{f_k\}_k$ to respect the decomposition \eqref{eq:decompositions_of_S_ell} and the decomposition of each $S^i_{\ell k}$ into simultaneous eigenspaces of $\{M_\alpha\}_\alpha$. In particular we obtain a trivialization of $S^{i+}_\ell\vert_U$.  

From Proposition~\ref{prop:properties_of_compatible_subspaces}, the bundle maps, 
\begin{equation}
 \label{eq:decompositions}   
\begin{aligned}
S^0\vert_Z &= \bigoplus_\ell S_\ell^0, \qquad  c_\ell : \text{Cl}^0(N^*)\otimes S^{0+}_\ell\oplus \text{Cl}^1(N^*)\otimes S^{1+}_\ell \rightarrow S_\ell^0  , 
\\
S^1\vert_Z &= \bigoplus_\ell S^1_\ell, \qquad c_\ell: \text{Cl}^0(N^*)\otimes S^{1+}_\ell\oplus \text{Cl}^1(N^*)\otimes S^{0+}_\ell \rightarrow  S_\ell^1. 
\end{aligned}
\end{equation}
are isometric isomorphisms and the structure group of $S^i\vert_Z$ is reduced to the product of $SO(n-m)$ and the structure group of $S^{i+}$. The Clifford bundles $\text{Cl}^i (N^*)$ admit an $SO(n-m)$ connection, induced by $\nabla^N$. Based on the decompositions \eqref{eq:decompositions}, we introduce a new connection on $E^i\vert_Z$:

\begin{defn}
\label{eq:connection_bar_nabla}
Let $v\in TZ$. By decompositions \eqref{eq:decompositions}, a given section $\xi \in C^\infty(Z; S^0_\ell \oplus S^1_\ell)$ is a sum of elements of the form $c(w_i) \xi^{i+}$ for some uniquely defined $w_i\in C^\infty( Z ; \text{Cl}(N^*))$ and some $\xi^{i+} \in C^\infty(Z; S^{i+}_\ell),\, i =0,1$. We define the connection 
\[
\bar\nabla_v (c(w_i) \xi^{i+}):=  c(\nabla^{N^*}_v w_i) \xi^{i+} + ( c(w_i) \circ  P_\ell^{i+})(\nabla^{E^i\vert_Z}_v \xi^{i+}),
\]
for every $i=0,1$. When $\xi \in C^\infty(Z; (S^i\vert_Z)^\perp)$, we define
\[
\bar\nabla_v \xi =  (1_{E^i\vert_Z}- P^i)(\nabla^{E^i\vert_Z}_v\xi).
\]
\end{defn}

The connection $\bar\nabla$ satisfies the following basic properties:
\begin{prop}
 \label{prop:basic_restriction_connection_properties}
\begin{enumerate}
\item By definition, $\bar\nabla$ preserves the space of sections of the bundles $S^i\vert_Z,\ (S^i\vert_Z)^\perp$ and $S^i_\ell,\ S^{i+}_\ell$, for every $\ell$ and every $i=0,1$. Moreover it is compatible with the metrics of $E^i\vert_Z, \ i=0,1$ and reduces, by definition, to a sum of connections, each of which is associated to each summand of the decomposition of $S^i\vert_Z$ into eigenbundles $S^i_{\ell k}$ introduced in decompositions \eqref{eq:eigenspaces_of_Q_i} and \eqref{eq:decompositions_of_S_ell}. .

\item
Let $\xi\in C^\infty(Z; S^i\vert_Z)$. Then $\bar\nabla$ satisfies,
\begin{equation}
\label{prop:basic_restriction_connection_properties2}
[\bar\nabla,  c(w)] \xi = \begin{cases} c( \nabla^{N^*} w) \xi,& \quad \text{if $w\in C^\infty(Z; N^*)$,}
\\ 
 c(\nabla^{T^*Z}w)\xi,& \quad \text{if $w\in C^\infty(Z; T^*Z)$.}
\end{cases}
\end{equation}
When $\xi\in C^\infty(Z; (S^i\vert_Z)^\perp)$ and $w\in C^\infty(Z; T^*\mathcal{N}\vert_Z)$, then
\begin{equation}
\label{prop:basic_restriction_connection_properties1}
[\bar\nabla,   c(w)] \xi = c(\nabla^{T^*\mathcal{N}\vert_Z} w) \xi.
\end{equation}
\end{enumerate}
\end{prop}
The proof is provided in Appendix subsection \ref{subApp:The_expansion_of_the_Spin_connection_along_Z}

Finally, we consider the difference
\begin{equation}
\label{eq:remainder_term}
B^i:  TZ \to \mathrm{End}(E^i\vert_Z),\quad v \mapsto \nabla_v - \bar\nabla_v, 
\end{equation} 
for every $i=0,1$. By the definition of $\bar \nabla$, it follows that
\begin{equation}
\label{eq:properties_of_remainder_term} 
 B_v^i(S^{i+}_\ell) \perp S^{i+}_\ell \qquad \text{and} \qquad    B_v^i ((S^i\vert_Z)^\perp) \subset S^i\vert_Z,
\end{equation}
for every $\ell$ and every $i=0,1$. By the Proposition \ref{prop:basic_restriction_connection_properties}, $B_v^i$ is a skew-symmetric bundle map satisfying
\[
 B_v^i( w_\cdot \xi) = (p_Z\nabla^{T^*X\vert_Z}_v w)_\cdot \xi +  w_\cdot B_v^{1-i} \xi,
\]
for every $v\in TZ,\ i=0,1$ and every $w\in C^\infty(Z; N^*)$ and $\xi\in  C^\infty(Z; S^i\vert_Z)$, where $p_Z:T^*X\vert_Z \to T^*Z$ is the orthogonal projection. The bundle maps $B^i,\ i=0,1$ appear in the definition of the terms ${\mathcal B}^i$ in the expansions  of Lemma~\ref{lem:Dtaylorexp} which will be the main objective of the following section.

\medskip

\vspace{1cm}

\setcounter{equation}{0}
\section{Structure of \texorpdfstring{$D + s{\mathcal A}$}{} along the normal fibers} 
\label{sec:structure_of_D_sA_along_the_normal_fibers}

Recall that $Z\subset X$ is an $m$-dimensional submanifold with normal bundle $\pi:\mathbb{R}^{n-m} \hookrightarrow N\to Z$.  We study the perturbative behavior of the operator  $\tau^{-1}\circ\tilde D_s : C^\infty(\mathcal{N}; \pi^*(E^0\vert_Z)) \to C^\infty(\mathcal{N}; \pi^*(E^1\vert_Z))$. The expansions will be carried out for the diffeomorphic copies $g = \exp^* g_X$, connection $\tilde \nabla$,  volume form $d\mathrm{vol}^N$ and Clifford structure $\tilde c = {\mathcal I}^{-1} \circ c \circ (I \otimes {\mathcal I})$ and $\tilde {\mathcal A}$ as defined in Definition~\ref{eq:tilde_D}. Throughout the section, we use bundle coordinates $(N\vert_U, (x_j, x_\alpha)_{j,\alpha})$ of the total space $N$ that are restricted to the tubular neighborhood $\mathcal{N}_\varepsilon = \mathcal{N}$ and frames $\{\sigma_\ell\}_\ell$ of $E^0\vert_U$ and $\{f_k\}_k$ of $E^1\vert_U$ that were introduced in Appendix~\ref{App:Taylor_Expansions_in_Fermi_coordinates}, obeying relations \eqref{eq:Clifford_connection_one_form_local_representations_in_on_frames} and \eqref{eq:comparing_tilde_nabla_orthonormal_to_bar_nabla_orthonormal}.

Assumption~\ref{Assumption:normal_rates} guarantees that the singular behavior of the perturbation term $\tilde {\mathcal A}$  operator $ \tau^{-1} \circ\tilde D_s$ becomes regular after rescaling $\{w_\alpha = \sqrt{s} x_\alpha\}_\alpha$. Recall the decomposition $TN = {\mathcal V} \oplus {\mathcal H}$ into vertical and horizontal distributions introduced in Appendix~\ref{subApp:The_total_space_of_the_normal_bundle_N_to_Z}. The terms in the perturbation series of $\tilde D$ involving  derivation fields in the normal distribution will re-scale to order $O(\sqrt{s})$, after the re-scaling is applied, introducing operator $\slashed D_0$. Since the metric in the fibers of $N$ becomes Euclidean in the blow up, this is a Euclidean Dirac operator in the normal directions. The terms involving derivation fields in the horizontal distribution will contain the fields of order $O(1)$ introducing the horizontal operator $\bar D^Z$. Each of these differential operators is originally defined to act on sections of the bundles $\pi^*(E\vert_Z)\to \mathcal{N}$ of the tubular neighborhood $ \mathcal{N}\subset N$. However these operators make sense on sections of the bundles $\pi^*(E\vert_Z)\to N$ over the total space of the normal bundle $\pi : N \to  Z$. Recall the operators $\mathfrak{c}_N,\ \nabla^{\mathcal V}$ and $\nabla^{\mathcal H}$ introduced in Appendix from subsections~ \ref{subApp:tau_j_tau_a_frames} to \ref{subApp:The_pullback_bundle_E_Z_bar_nabla_E_Z_c_Z}. We have the following definitions:  

\begin{defn}
\label{defn:vrtical_horizontal_Dirac}

 We define $\slashed D_0$ by composing 
\begin{equation*}
 \xymatrix{
 C^\infty(N; \pi^*(E^0\vert_Z)) \ar[r]^-{\bar\nabla^{\mathcal V}} & C^\infty(N; {\mathcal V}^* \otimes \pi^*(E^0\vert_Z)) \ar[r]^-{\mathfrak{c}_N} & C^\infty(N; \pi^*(E^1\vert_Z)),
}
\end{equation*}
and $\bar D^Z$ by composing 
\begin{equation*}
 \xymatrix{
 C^\infty(N; \pi^*(E^0\vert_Z)) \ar[r]^-{\bar\nabla^{\mathcal H}} & C^\infty(N; {\mathcal H}^* \otimes \pi^*(E^0\vert_Z)) \ar[r]^-{\mathfrak{c}_N} & C^\infty(N; \pi^*(E^1\vert_Z)).
}
\end{equation*}
We restrict $\slashed D_0$ to sections of the sub-bundles $\pi^*(S^i\vert_Z)$ of $\pi^*(E^i\vert_Z),\, i =0,1$ and recall the term $\bar A_r$ introduced in \eqref{eq:1st_jet_2nd_jet_of_A}. Define  
\[
\slashed{D}_s  : C^\infty(N; \pi^*(S^0\vert_Z)) \rightarrow C^\infty(N; \pi^*(S^1\vert_Z)), \quad \xi\mapsto \slashed D_0\xi + s r \bar A_r\xi, 
\]
for every $s\in \mathbb{R}$.
\end{defn}

\begin{rem}
\label{rem:properties_of_horizontal_and_vertical_operators}
 Given a section $\xi: N \to \pi^*(E\vert_Z)$, we use the same letter $\xi=(\xi_1,\dots, \xi_d)$ to denote is coordinates with respect to the frames $\{\sigma_\ell\}_\ell,\ \{f_k\}_k$. Also recall the lifts $\{h_A = H e_A\}_{A=j,\alpha}$ with their coframes $\{h^A = (H^*)^{-1}(e^A) \}_{A=j,\alpha}$ and the matrix expressions $c^A$ of $\mathfrak{c}_N(h^A)$ with respect to these frames introduced in Appendix from subsections~ \ref{subApp:tau_j_tau_a_frames} to \ref{subApp:The_pullback_bundle_E_Z_bar_nabla_E_Z_c_Z}:

\begin{enumerate}
\item  
\label{rem:local_symbol_expressions}
Let $v = v_j h^j + v_\alpha h^\alpha\in T^*N$. The symbol maps are described in local coordinates by
\begin{align*}
\sigma_{\slashed D_0}: T^*N\otimes \pi^*(E^0\vert_Z) &\to \pi^* (E^1\vert_Z), \quad v\otimes \xi \mapsto v_\alpha c^\alpha \xi,
 \\
\sigma_{\bar D^Z}: T^*N\otimes \pi^*(E^0\vert_Z) &\to \pi^* (E^1\vert_Z), \quad v\otimes \xi \mapsto (v_j -  v_\alpha x_\beta \bar\omega_{j\beta}^\alpha) c^j \xi. 
\end{align*}

\item \label{rem:local_expression_of_slashed_D_0} For fixed $z\in Z$, the operator $\slashed D_0:  C^\infty(N_z; E_z^0) \to C^\infty(N_z; E^1_z)$ is a Euclidean Dirac operator because its expression in local coordinates is 
\[
\mathfrak{c}_N(h^\alpha) \bar\nabla_{h_\alpha} = c^\alpha (h_\alpha \xi) = c^\alpha \partial_\alpha \xi. 
\]
Recalling that $ r \bar A_r= x_a \bar A_\alpha = c^\alpha M^0_\alpha$, we have the local expression
\[
\slashed D_s \xi = c^a( h_\alpha + s x_\alpha M_\alpha^0) \xi.
\]

\item \label{rem:local_expression_of_D_Z} By using \eqref{eq:pullback_bar_connection_components}, the operator $\bar D^Z$ has local expression 
\[
\bar D^Z \xi = \mathfrak{c}_N(h^j) \bar\nabla_{h_j}\xi = c^j( h_j + \phi^0_j)\xi.
\]

\item \label{rem:respectful_eigenspaces} Recall the eigenbundles $S^i_{\ell k} \to Z$ of $C^i,\, i =0,1$ in Definition~\ref{defn:IntroDefSp}. Recall also the construction of the covariant derivative $ \bar\nabla_{h_j}$ in Appendix~\ref{subApp:The_pullback_bundle_E_Z_bar_nabla_E_Z_c_Z}  By Proposition~\ref{prop:basic_extension_bar_connection_properties} \eqref{prop:basic_extension_bar_connection_properties3}, the operators $\slashed D_s$ and $\bar D^Z$ decompose into blocks with each block being a differential operator on sections $C^\infty(N; \pi^*S^0_{\ell k}) \to C^\infty(N; \pi^*S^1_{\ell k})$, for every $\ell$ and every $k=0,\dots, n-m$. 
\end{enumerate}
\end{rem}

The following expansions are proven in \cite{bl}[Theorem 8.18, p. 93]. We include the proof for completeness:
 
\begin{lemma}
\label{lem:Dtaylorexp}
Let $\eta = \tau \xi$ where $\xi : \mathcal{N}_U\rightarrow \pi^*(E^i\vert_Z)$. At the fibers of $\mathcal{N}_U \to U$, we have the expansions
\[
\tau^{-1} \tilde D \eta = \slashed D_0 \xi + \bar D^Z \xi + {\mathcal B}^0\xi + O(r^2 \partial^{\mathcal V} + r \partial^{\mathcal H} + r)\xi,
\]
where $r$ is the distance function from the core set.
\end{lemma}

\begin{proof}
We use Hitchin's dot notation $\tilde c(v) = v_\cdot$ throughout the proof. By linearity, it suffices to work with $\eta=\tau\xi = f \sigma_k$, for some $f:\mathcal{N}_U \to \mathbb{R}$ and use the expression of $\tilde D$ in local frames,
\[
\tilde D \eta = \tau^\alpha_\cdot \tilde \nabla_{\tau_\alpha}  \eta + \tau^j_\cdot \tilde\nabla_{\tau_j} \eta.
\] 
Since the Clifford multiplication is $\tilde\nabla^\mathcal{N}$-parallel, we have
\[
\tau^\alpha_\cdot \sigma_k = \tau( c^\alpha \sigma_k) \quad \text{and}\quad \tau^j_\cdot \sigma_k = \tau ( c^j \sigma_k). 
\]
By \eqref{eq:on_frames_expansions},
\[
\tau_\alpha =  h_\alpha +  O(r^2\partial^{\mathcal V} + r\partial^{\mathcal H}),
\]
and by \eqref{eq:Clifford_connection_one_form_local_representations_in_on_frames},
\begin{align*}
\tilde\nabla_{\tau_\alpha} \eta &= (\tau_\alpha f) \sigma_k + f\theta_\alpha^0 \sigma_k
\\
&=(h_\alpha f) \sigma_k + O( (r^2\partial^{\mathcal V} + r\partial^{\mathcal H}+ r) (f \sigma_k).
\end{align*}
 
Hence, for the normal part, we estimate
\begin{align*}
\tau^\alpha_{\cdot}\tilde \nabla_{\tau_\alpha}\eta &=  ( h_\alpha  f) \tau^\alpha_\cdot \sigma_k \, +\, + O( (r^2\partial^{\mathcal V} + r\partial^{\mathcal H}+ r) (f \sigma_k) 
\\ 
&=  (h_\alpha f) \tau  c^\alpha \sigma_k +\, O(r^2\partial^{\mathcal V} + r\partial^{\mathcal H} + r) (f \sigma_k)
\\
&= \tau \left[ \slashed D_0 \xi +\,  O(r^2\partial^{\mathcal V} + r\partial^{\mathcal H} + r)\xi \right],
\end{align*}
where the last equality follows by Remark~\ref{rem:properties_of_horizontal_and_vertical_operators} \ref{rem:local_expression_of_slashed_D_0}.

By  \eqref{eq:on_frames_expansions},
\[
\tau_j = h_j +  O(r^2\partial^{\mathcal V} + r\partial^{\mathcal H}),
\]
and by \eqref{eq:Clifford_connection_one_form_local_representations_in_on_frames} and \eqref{eq:comparing_tilde_nabla_orthonormal_to_bar_nabla_orthonormal}, the connection on the horizontal directions expands as
\begin{align*}
\tilde \nabla_{\tau_j}\eta &= (\tau_j f) \sigma_k + f\theta_j^0 \sigma_k 
\\
&= (h_j f) \sigma_k + f{(\phi_j^0 + B_j^0)}_k^l\sigma_l  +   O(r^2\partial^{\mathcal V} + r\partial^{\mathcal H}+ r)(f\sigma_k),
\end{align*}
where the term $B_j^0 = B_{jk}^{0l} \sigma^k\otimes \sigma_l$ is introduced in \eqref{eq:remainder_term}.
It follows that
\begin{align*}
\tau^j_{\cdot} \tilde \nabla_{\tau_j}\eta &=  (h_j f) \tau^j_\cdot\sigma_k + f{(\phi_j^0 + B_j^0)}_k^l \tau^j_\cdot\sigma_l  +   O(r^2\partial^{\mathcal V} + r\partial^{\mathcal H}+ r)(f\sigma_k) 
\\
 &=  \tau[  (h_j f)  c^j\sigma_k +  f{(\phi_j^0 + B_j^0)}_k^l c^j \sigma_l + O(r^2\partial^{\mathcal V} + r\partial^{\mathcal H}+ r)(f\sigma_k)],
\end{align*}
that is, by  using Remark~\ref{rem:properties_of_horizontal_and_vertical_operators} \ref{rem:local_expression_of_D_Z},
\begin{align*}
\tau^j_{\cdot}\tilde\nabla_{\tau_j}\eta = \tau \left[ \bar D^Z\xi\,+  {\mathcal B}^0\xi  +\,  O(r^2 \partial^{\mathcal V} + r\partial^{\mathcal H} + r)\xi \right].
\end{align*}
Here ${\mathcal B}^0:= c^j B_j^0\in \mathrm{Hom}(E^0\vert_Z; E^1\vert_Z)$. Adding up the two preceding expressions and using the local expressions for $\slashed D_0$ and $\bar D^Z$ in Remark~\ref{rem:properties_of_horizontal_and_vertical_operators}, we obtain the required expansion.
\end{proof}
 
In Proposition~\ref{prop:Weitzenbock_identities_and_cross_terms}, the properties of these vertical and horizontal operators are presented. The horizontal operator satisfies a Weitzenboock formula. This is a naturally occurring formula if one introduces the metric $g^{TN}$ and the Clifford action $\mathfrak{c}_N$ on $TN$. The connections $\bar\nabla^{\pi^*(E\vert_Z)}$ and $\bar\nabla^{TN}$  introduced in Appendix subsection~\ref{subApp:The_total_space_of_the_normal_bundle_N_to_Z}  become compatible with the Riemannian metric and the Clifford multiplication per Proposition~\ref{prop:basic_extension_bar_connection_properties}.  

The following proposition calculates in local frames, the formal adjoints of  $\slashed D_s$ and $\bar D^Z$ on the total space $(N, g^{TN}, d\mathrm{vol}^N)$. 

\begin{proposition}
\label{prop:cross_adjoint}
The formal adjoints of the operators $\slashed D_s$ and $\bar D^Z$ with respect to the metric on $\pi^*(S^0\vert_Z)$ and volume form $d\mathrm{vol}^N$ are computed in local coordinates by  
\begin{equation}
\label{eq:local_expressions_for_formal_adjoints}
\slashed D^*_s = c^\alpha (\partial_\alpha + sx_\alpha M^1_a) \qquad \text{and}\qquad \bar D^{Z*} = c^j(h_j + \phi_j^1).  
\end{equation}
\end{proposition}

\begin{proof}
Let $\xi_i: N \to \pi^*(S^i\vert_Z),\ i=0,1$ smooth sections so that at least one of them is compactly supported. We define vector field $Y\in C^\infty(N; TN)$ by acting on covectors,
\[
Y: T^*N \to \mathbb{R},\ v \mapsto \langle \mathfrak{c}_N(v) \xi_1, \xi_2\rangle, 
\]
and we decompose it into its vertical and horizontal parts $Y = Y_{\mathcal V} + Y_{\mathcal H}$. We then define the $n-1$-forms
\[
\omega^{\mathcal V} = \iota_{Y_{\mathcal V}} d \mathrm{vol}^N \quad \text{and} \quad \omega^{\mathcal H} = \iota_{Y_{\mathcal H}} d \mathrm{vol}^N.
\]
We have the equations,
\begin{align*}
d \omega^{\mathcal V} &= ( \langle \slashed D_0 \xi_1, \xi_2\rangle - \langle \xi_1, \slashed D_0^* \xi_2 \rangle)\, d \mathrm{vol}^N,
\\
d \omega^{\mathcal H} &= ( \langle \bar D^Z \xi_1, \xi_2\rangle - \langle \xi_1, \bar D^{Z*} \xi_2 \rangle)\, d \mathrm{vol}^N.
\end{align*}
We prove the second identity. Calculating over $N_p$, as in the proof of Proposition~\ref{prop:basic_extension_bar_connection_properties}, we have ${\mathcal L}_{Y_{\mathcal H}} d \mathrm{vol}^N = h_j(Y_j)\, d\mathrm{vol}^N$ and 
\[
h_j Y_j  = \langle \mathfrak{c}_N(h^j) \bar\nabla_{h_j} \xi_1, \xi_2\rangle + \langle \mathfrak{c}_N(h^j)  \xi_1, \bar\nabla_{h_j}\xi_2\rangle = \langle \bar D^Z \xi_1, \xi_2\rangle - \langle \xi_1, \bar D^{Z*} \xi_2\rangle.
\]
The first identity is proven analogously. The proof of the proposition follows. 
\end{proof}

Combining Lemma~\ref{lem:Dtaylorexp} with expansion in Appendix~\ref{subApp:The_expansion_of_A_A*A_nabla_A_along_Z} \eqref{eq:jet0} we obtain the expansions for $\tilde D_s$. The same computation and the local expresions given in Proposition~\ref{prop:cross_adjoint} give analogue expansions for $\tilde D_s^*$, the $L^2(\mathcal{N})$ formal adjoint of $\tilde D_s$ with respect to the density function $d\mathrm{vol}^\mathcal{N}$. The expansions are described in the following corollary:

\begin{cor}
\label{cor:taylorexp}
$\tilde D + s\tilde {\mathcal A}$ expands along the normal directions of the singular set $Z$ with respect to $\xi_i \in C^\infty(\mathcal{N}; \pi^*(S^i\vert_Z)),\, i=0,1$, as 
\begin{equation}
\label{eq:taylorexp}
(\tilde D+ s \tilde{\mathcal A})\tau\xi_0 =  \tau \left(\slashed{D}_s+ \bar D^Z + {\mathcal B}^0 + \frac{1}{2} sr^2 \bar A_{rr}\right) \xi_0
+  O( r^2\partial^{\mathcal V} + r\partial^{\mathcal H }+r + sr^3)\tau\xi_0,
\end{equation}
and
\begin{equation*}
(\tilde D^*+ s\tilde {\mathcal A}^*)\tau\xi_1 =  \tau \left(\slashed{D}_s^*+ \bar D^{Z*} + {\mathcal B}^1 + \frac{1}{2} sr^2 \bar A^*_{rr}\right) \xi_1
+  O( r^2\partial^{\mathcal V} + r\partial^{\mathcal H} +r + sr^3)\tau\xi_1,
\end{equation*}
as $r \to 0^+$. When $\xi_i\in C^\infty(\mathcal{N} ; \pi^*(S^i\vert_Z)^\perp)$ then,
\begin{equation}
\label{eq:taylorexp1}
(\tilde D + s \tilde {\mathcal A}) \tau \xi_0 = \tau(\slashed D_0 + \bar D^Z + {\mathcal B}^0) \xi_0 + s\bar{\mathcal A} \tau \xi_0
+ O( r^2\partial^{\mathcal V}+ r\partial^{\mathcal H} +(1+s)r)\tau\xi_0,
\end{equation}
and
\begin{equation*}
(\tilde D + s\tilde {\mathcal A})^* \tau \xi_1 = \tau(\slashed D_0^* + \bar D^{Z*} + {\mathcal B}^1) \xi_1 + s \bar{\mathcal A}^* \tau \xi_1
+O( r^2\partial^{\mathcal V} + r\partial^{\mathcal H} +(1+s)r)\tau\xi_1,
\end{equation*}
as $r\to 0^+$.
\end{cor}

\begin{rem}
\begin{enumerate}
\item  We note that in the preceding expansions, the $L^2$-adjoint of $\tilde D_0 : C^\infty (\mathcal{N}; \tilde E^0\vert_\mathcal{N}) \to C^\infty(\mathcal{N}; \tilde E^1\vert_\mathcal{N})$ on the left hand side, denoted by $\tilde D^*_0$ is computed with respect to the metrics on $\tilde E$ and $T^*\mathcal{N}$ and with respect to the volume form $d\mathrm{vol}^\mathcal{N}$. On right hand side, the adjoints $\slashed{D}_0^*$ and $\bar D^{Z*}$ are computed with respect to the pullback metric, on $\pi^*E$, the metric $g^{TN}$ and the volume form $d\mathrm{vol}^N$.
\end{enumerate}
\end{rem}

\medskip
   
  \vspace{1cm}
\setcounter{equation}{0}
\section{Properties of the operators \texorpdfstring{$\slashed D_s$}{} and \texorpdfstring{$D^Z$}{}.}
\label{sec:Properties_of_the_operators_slashed_D_s_and_D_Z.}

In this section we review some well known Weitzenbock type formulas from \cite{bl} for the operators $\slashed D_s$ and $\bar D^Z$ introduced in the preceding paragraph. In particular the Weitzenbock identity of the operator $\slashed D_0$ involves a harmonic oscillator. As proven in Proposition~\ref{prop:basic_extension_bar_connection_properties} \eqref{prop:basic_extension_bar_connection_properties1} of the Appendix, the connection $\bar\nabla$ on the bundle $\pi^*E \to N$ is Clifford compatible to the connection $\bar\nabla^{TN}$ (introduced in Appendix~\ref{subApp:The_total_space_of_the_normal_bundle_N_to_Z}). The latter has nontrivial torsion $T$ that is described explicitly in Proposition~\ref{prop:basic_extension_bar_connection_properties}. 
Throughout the section we will use the constructions frames of Appendices~\ref{App:Fermi_coordinates_setup_near_the_singular_set} and \ref{App:Taylor_Expansions_in_Fermi_coordinates}. We work on a bundle chart $(N\vert_U, (x_j, x_\alpha)_{j,\alpha})$ where the normal coordinates $(U, (x_j)_j)$ are centered at $p \in U$. Recall the frames $\{e_j\}_j$ being $\nabla^{TZ}$-parallel at $p\in Z$, the frames $\{e_\alpha\}_\alpha$ that are $\nabla^N$-parallel at $p$ and the horizontal lifts $\{h_A = {\mathcal H}(e_A)\}_{A=j,\alpha}$. By using relations \eqref{eq:local_components_for_bar_nabla_TN} and \eqref{eq:connection_comp_rates}, the connection components of $\bar\nabla^{TN}$ vanish over $N_p$ so that $\bar\nabla_{h_A} = h_A, \ A=j, \alpha$. 

Recall also the frames $\{\sigma_k, f_\ell\}_{k,\ell}$ trivializing $\tilde E_{\mathcal{N}_U}$, introduced in Appendix~\ref{App:Taylor_Expansions_in_Fermi_coordinates}. Similarly, by using relations \eqref{eq:pullback_bar_connection_components} and \eqref{eq:comparing_tilde_nabla_orthonormal_to_bar_nabla_orthonormal}, the connection components of $\bar\nabla^{\pi^*(E^i\vert_Z)},\ i=0,1$ vanish over $N_p$ so that $\bar\nabla_{h_A} = h_A,\ A=j,\alpha$. 

Recall $Q= Q^0 \oplus Q^1$ introduced in Proposition~\ref{prop:more_properties} \eqref{rem:metric_compatibility1} and $C = C^0 \oplus C^1$ introduced in \eqref{eq:IntroDefCp}. Recall also that $M_\alpha =M_\alpha^0 \oplus M_\alpha^1$ satisfies $C = \sum_\alpha M_\alpha$ and $M_\alpha^2 = Q$, for every $\alpha$. 

Recall the Clifford map $\mathfrak{c}_N$, introduced in Definition~\ref{defn:definition_of_mathfrac_Clifford} with local representation as the matrix $c^A,\ A=j,\alpha$. By employing relations \eqref{eq:horizontal_lift}, \eqref{eq:basic_extension_bar_connection_properties1} and \eqref{eq:connection_comp_rates} over $N_p$,   
\begin{equation}
\label{eq:c_N_commutator_identities}
\begin{aligned} 
 \relax[\bar\nabla_{h_j}, c^k] &= [\bar\nabla_{h_j}, \mathfrak{c}_N(h^k)] = \bar\omega_{jk}^l(p)c^l =0,
\\
[\bar\nabla_{h_j}, c^\alpha] &= [\bar\nabla_{h_j}, \mathfrak{c}_N(h^\alpha)] = \bar\omega_{j\alpha}^\beta(p) c^\beta =0, 
\\
[\bar\nabla_{h_j} , x_\alpha c^\alpha ] &= [ h_j, x_\alpha ] c^\alpha+ x_\alpha[\bar\nabla_{h_j}, c^\alpha]=  x_\beta\bar\omega_{j\beta}^\alpha(p) c^\alpha =0,
\\
[\bar\nabla_{h_\alpha}, c^A]&= [\bar\nabla_{h_\alpha}, \mathfrak{c}_N(h^A)] = \begin{cases} \bar\omega_{\alpha j}^k(p)c^l, & \quad \text{if $A=j$},\\ \bar\omega_{\alpha \beta}^\gamma(p) c^\gamma,&\quad \text{if $A=\beta$} \end{cases} =0.
\end{aligned}
\end{equation}

Finally recall the expressions in local frames of the operators $\slashed D_0$ and $\bar D^Z$ and their symbols described in Remark~\ref{rem:properties_of_horizontal_and_vertical_operators}.
In the aforementioned frames, above the fiber $N_p$, they become 
\begin{equation}
\label{eq:local_expressions_over_N_p}
\begin{aligned}
\sigma_{\slashed D_0}(v)\xi &= v_\alpha c^\alpha\xi, \quad \slashed D_0 \xi = \mathfrak{c}_N(h^\alpha) \bar\nabla_{h_\alpha}\xi = c^\alpha \partial_\alpha\xi,   
\\
\sigma_{\bar D^Z}(v)\xi &= v_j c^j\xi, \quad \bar D^Z\xi = \mathfrak{c}_N(h^j)\bar\nabla_{h_j}\xi = c^jh_j \xi,
\end{aligned}
\end{equation}
where $v= v_jh^j+v_\alpha h^\alpha\in T_p^*N$ and $\xi\in C^\infty(N; \pi^*(E^0\vert_Z))$.

\begin{proposition}
\label{prop:Weitzenbock_identities_and_cross_terms}
\begin{enumerate}
\item  For every $\xi \in C^\infty(N; \pi^* (S^0\vert_Z))$,
\begin{equation}
\label{eq:slashed_D_s_Weitzenbock}
\slashed D_s^* \slashed D_s \xi =  (- \Delta  +  s^2 r^2 Q^0 - sC^0)\xi,
\end{equation}
where $\Delta = \sum_\alpha \partial^2_\alpha$ is the Euclidean Laplacian in the fibers $N_z,\, z \in Z$ and
\begin{equation}
\label{eq:bar_D_z_Weitzenbock}
\bar D^{Z*}\bar D^Z\xi = \bar\nabla^{{\mathcal H}*}\bar\nabla^{\mathcal H} \xi - (c_T\bar\nabla) \xi + F\xi , 
\end{equation}
where $F$ is the Clifford contraction of the curvature of the bundle $(S^0\vert_Z,  \bar\nabla) \to Z$ and $c_T \bar\nabla$ is the Clifford contraction $\frac{1}{2}\mathfrak{c}_N(h^j) \mathfrak{c}_N(h^k) T_{jk}^\alpha \bar\nabla_{h_\alpha}$.

\item On sections of the bundle $\pi^* (E^0\vert_Z))\to N$,  the term 
\begin{equation}
 \label{eq:cross_terms3}
(\slashed D_0 + \bar D^Z)^* \circ \bar{\mathcal A} + \bar{\mathcal A}^* \circ(\slashed D_0 + \bar D^Z),
\end{equation}
is a zeroth order operator.

\item For every $\xi \in C^\infty(N; \pi^*(S^{0+}_\ell))$,
\begin{equation}
\label{eq:cross_terms}
(\slashed  D_s^*  \bar D^Z  +  \bar D^{Z*} \slashed D_s )\xi=- s  x_\alpha \mathfrak{c}_N(h^\alpha) \mathfrak{c}_N( \pi^*d\lambda_\ell)\xi.
\end{equation}

\item On sections of the bundle $\pi^* (E^0\vert_Z))\to N$,  the term 
\begin{equation}
\label{eq:cross_terms1}
\slashed  D_s^*  \bar D^Z  +  \bar D^{Z*} \slashed D_s, 
\end{equation}
 is 0-th order operator with coefficients of order $sO(r)$. In particular
\begin{equation}
\label{eq:cross_terms2}
\slashed D^*_0  \bar D^Z  +  \bar D^{Z*} \slashed D_0 \equiv 0.
\end{equation}

\item 
For every $v\in C^\infty(N;TN)$ and every $\xi\in C^\infty(N;\pi^*E^1)$,
\begin{equation}
\label{eq:cross_terms25}
[\slashed D_s^*, \bar\nabla_v] \xi = - s[\bar\nabla_v, r \bar A_r^*]  \xi. 
\end{equation}
\end{enumerate}
Analogous facts fold for  the operators $ \slashed D_s \slashed D_s^*,\ \bar D^Z \bar D^{Z*}$ and $\slashed  D_s  \bar D^{Z*}  +  \bar D^Z \slashed D_s^*$.

\end{proposition}

\begin{proof}
Most of the calculations involved in the proof of the proposition are carried over the fiber $N_p$ of the total space $N$ where expressions \eqref{eq:local_expressions_over_N_p} hold. To prove \eqref{eq:slashed_D_s_Weitzenbock} we use \eqref{eq:local_expressions_over_N_p} over $N_p$ and calculate
\begin{align*}
\slashed D_s^* \slashed D_s \xi &= \sum_{\alpha, \beta}(c^\alpha\partial_\alpha + s x_\alpha \bar A^*_\alpha)  (c^\beta \partial_\beta + s x_\beta \bar A_\beta) \xi
\\
&=\sum_{\alpha, \beta} ( c^\alpha c^\beta \partial^2_{\alpha \beta}  + sx_\alpha(c^\beta  \bar A_\alpha + \bar A^*_\alpha c^\beta) \partial_\beta  + s^2 x_\alpha x_\beta \bar A^*_\alpha \bar A_\beta ) \xi + s \sum_\alpha c^\alpha \bar A_\alpha \xi
\\
&= \frac{1}{2}\sum_{\alpha, \beta} [(c^\alpha c^\beta + c^\beta c^\alpha) \partial^2_{\alpha \beta}  + s^2 x_\alpha x_\beta (\bar A^*_\alpha \bar A_\beta +\bar A_\beta^* \bar A_\alpha)] \xi - s\sum_\alpha M^0_\alpha \xi,
\end{align*}
where in the last line we used equation \eqref{eq:dcond} to eliminate the cross terms. By Proposition~\ref{prop:more_properties} \ref{rem:metric_compatibility1}, we have $M_{\alpha,\beta}^0 = \frac{1}{2}( \bar A^*_\alpha \bar A_\beta +\bar A_\beta^* \bar A_\alpha) = \delta_{\alpha \beta}Q$ and by the Clifford relations \eqref {eq:Clifford_relations} we obtain $c^\alpha c^\beta + c^\beta c^\alpha = -2 \delta_{\alpha \beta}$. Identity \eqref{eq:slashed_D_s_Weitzenbock} follows.

Identity \eqref{eq:bar_D_z_Weitzenbock} is treated as the usual calculation for the Weitzenbock type formula. We use the local expressions \eqref{eq:local_expressions_over_N_p} of $\bar D^Z$ and $\bar D^{Z*}$ over $N_p$, together with \eqref{eq:c_N_commutator_identities} and calculate:
\begin{align*}
\bar D^{Z*} (\bar D^Z\xi) &=  \mathfrak{c}_N(h^j) \bar\nabla_{h_j} (\mathfrak{c}_N(h^k) \bar\nabla_{h_k}\xi)
\\
&= \mathfrak{c}_N(h^j) \mathfrak{c}_N(h^k) (\bar\nabla_{h_j}\bar\nabla_{h_k} \xi)
\\
&= - \bar\nabla_{h_k}\bar\nabla_{h_k} \xi + \sum_{j<k} \mathfrak{c}_N(h^j) \mathfrak{c}_N(h^k) \left(\bar\nabla_{h_j}\bar\nabla_{h_k} -\bar\nabla_{h_j}\bar\nabla_{h_k}\right) \xi
\\
&= \bar\nabla^{{\mathcal H}*}\bar\nabla^{\mathcal H} \xi  + \sum_{j<k} \mathfrak{c}_N(h^j) \mathfrak{c}_N(h^k) \left( \mathrm{Hess}(h_j, h_k) - \mathrm{Hess}(h_k, h_j)\right) \xi
\\
&= \bar\nabla^{{\mathcal H}*}\bar\nabla^{\mathcal H} \xi  + \sum_{j<k} \mathfrak{c}_N(h^j) \mathfrak{c}_N(h^k) \left(F^{\pi^*S}(h_j, h_k)\xi - \bar\nabla_{T(h_j, h_k)}\xi \right). 
\end{align*}
In the last equality we used \eqref{eq:basic_extension_bar_connection_properties5}.

For the proof of \eqref{eq:cross_terms3} we calculate over $N_p$, the symbols of the involved operators,   
\begin{align*}
 [\sigma_{\slashed D_0 + \bar D^Z} (v)\circ \bar {\mathcal A}+  \bar{\mathcal A}^* \circ\sigma_{\slashed  D_0 + \bar D^Z} (v)] \xi  &= [ (v_\alpha c^\alpha + v_j c^j) \circ \bar {\mathcal A} +  \bar{\mathcal A}^* \circ  (v_\alpha c^\alpha + v_j c^j)]\xi, 
\end{align*}
where the last expression vanishes by equation \eqref{eq:cond}. Hence the associated differential operator is of zeroth order.

To prove \eqref{eq:cross_terms2} we use Clifford relations and \eqref{eq:c_N_commutator_identities} over $N_p$ to calculate, 
\begin{align*}
(\slashed D_0^*  \bar D^Z + \bar D^{Z*} \slashed D_0)\xi  &= \mathfrak{c}_N(h^\alpha) \bar\nabla_{h_\alpha}(\mathfrak{c}_N(h^j) \bar\nabla_{h_j} \xi) +  \mathfrak{c}_N(h^j) \bar\nabla_{h_j}( \mathfrak{c}_N(h^\alpha) \bar\nabla_{h_\alpha} \xi)
\\
&=\mathfrak{c}_N(h^\alpha)\mathfrak{c}_N(h^j) \left(\bar\nabla_{h_\alpha} \bar\nabla_{h_j} - \bar\nabla_{h_j} \bar\nabla_{h_\alpha} \right) \xi
\\
&= \mathfrak{c}_N(h^\alpha)\mathfrak{c}_N(h^j) \left(\mathrm{Hess}(h_\alpha, h_j) -  \mathrm{Hess}(h_j, h_\alpha) \right) \xi.
\end{align*}
But the last term is the zero section, by using the symmetry of the Hessian in Proposition~\ref{prop:basic_extension_bar_connection_properties}  \eqref{prop:basic_extension_bar_connection_properties5}. 

To prove \eqref{eq:cross_terms25} we calculate over $N_p$,
\begin{align*}
[\slashed D_s^*, \bar\nabla_v] \xi &= \mathfrak{c}_N(h^\alpha) \left(\bar\nabla_{h_\alpha} \bar\nabla_v - \bar\nabla_v \bar\nabla_{h_\alpha} \right) \xi - s[\bar\nabla_v, r \bar A_r^*]  \xi
\\
&= \mathfrak{c}_N(h^\alpha) \left(\mathrm{Hess}(h_\alpha, v) -  \mathrm{Hess}(v, h_\alpha) \right) \xi - s[\bar\nabla_v, r \bar A_r^*]  \xi
\\
&= - s[\bar\nabla_v, r \bar A_r^*]  \xi, 
\end{align*}
where in the intermediate equality we used again the symmetry of the Hessian.

To prove \eqref{eq:cross_terms1}, we use the local expression in frames, $r \bar A_r = x_a \bar A_\alpha$ and we employ \eqref{eq:horizontal_lift} to obtain $[h_j , x_\alpha ]= - x_\beta\bar\omega_{j \beta}^\alpha $. We then calculate at any point of $N_U$,
\begin{align*}
(\bar D^Z)^* (r \bar A_r \xi) + r \bar A^*_r (\bar D^Z \xi) &= c^j (h_j + \phi_j^1)( x_\alpha \bar A_a\xi) + x_\alpha \bar A^*_\alpha ( c^j(h_j  + \phi_j^0) \xi )
\\
&=  c^j [h_j + \phi_j , x_\alpha\bar A_\alpha] \xi  + x_\alpha ( c^j \bar A_\alpha + \bar A_\alpha^* c^j) (h_j + \phi_j^0) \xi  
\\
&= x_\alpha( c^j [e_j + \phi_j , \bar A_\alpha] - \bar\omega_{j\alpha}^\beta \bar A_\beta ) \xi, 
\end{align*}
where in the last equality $c^j \bar A_\alpha + \bar A_\alpha^* c^j =0$  by equation \eqref{eq:cond}. It follows that the term is a 0-th order operator with coefficients of order $O(r)$, as $r \to 0$.

To prove \eqref{eq:cross_terms}, we write $\bar A_\alpha = c^\alpha M^0_\alpha$, so that $M_\alpha^0 \xi = \lambda_\ell \xi$. Hence continuing the preceding calculation in this case over $N_p$,

\begin{align*}
(\bar D^Z)^* (r \bar A_r \xi) + r \bar A^*_r (\bar D^Z \xi) &=  c^j [\bar\nabla_{h_j} , x_\alpha\bar A_\alpha]\xi
\\
&= c^j[\bar\nabla_{h_j} , x_\alpha c^\alpha M^0_\alpha] \xi 
\\
&=  x_\alpha   c^j c^\alpha [h_j , \lambda_\ell  ] \xi, \quad (\text{by \eqref{eq:c_N_commutator_identities}})
\\
&=- x_\alpha c^\alpha (e_j(\lambda_\ell) c^j)\xi 
\\
&= - x_\alpha \mathfrak{c}_N(h^\alpha) \mathfrak{c}_N( \pi^* d\lambda_\ell) \xi.
\end{align*}

Both the left and the right hand sides of the preceding equations are independent of coordinates and frames and therefore the equations hold at any reference frame and are true everywhere on the total space $N$.  The proof is now complete.
\end{proof}

Using Proposition~\ref{prop:properties_of_compatible_subspaces} we decompose $S\vert_Z$ into eigenbundles $S_\ell$ of $Q$ and further into eigenbundles $S_{\ell k}$ of $C$ in \eqref{eq:decompositions_of_S_ell} of eigenvalue $(n-m- 2k)\lambda_\ell$. Fixing $z\in Z$, we calculate explicitly the eigenvalues of $\slashed D^2_s\vert_z: L^2(N_z; (S_{\ell k})_z) \to L^2(N_z; (S_{\ell k})_z)$. By equation  \eqref{eq:slashed_D_s_Weitzenbock},
\[
( - \Delta  +  s^2 r^2 Q^2 - sC )\xi =  ( - \Delta +  s^2 r^2\lambda_\ell^2   -  s(n-m-2k)\lambda_\ell )\xi = \lambda \xi.
\]
Changing variables  $\{w_\alpha =(s\lambda_\ell)^{1/2} x_\alpha\}_\alpha$ and setting $\tilde \xi(w) = \xi(x)$, we obtain 
\[
-\Delta \tilde \xi + r^2 \tilde \xi = \left[(n-m - 2k) + \frac{\lambda}{s\lambda_\ell} \right] \tilde\xi.
\]
It follows by \cite{p}[Th. 2.2.2] that 
\[
(n-m - 2k) + \frac{\lambda}{s\lambda_\ell} \in \{ 2\lvert a\rvert + n-m:\ a \in \mathbb{Z}^{n-m}_{\geq 0}\},
\]
so that the spectrum of $\slashed D^2_s\vert_z$ as unbounded operator on $L^2(N_z; (S_{\ell})_z)$ is given by, 
\begin{equation}
\label{eq:spectrum_of_Delta_s}
\{2s\lambda_\ell (\lvert a\rvert + k):\ a \in \mathbb{Z}^{n-m}_{\geq 0},\  k\in \{0,\dots ,n-m\}\},
\end{equation}
for every $\ell$. In particular $\lambda=0$ is an eigenvalue if and only if  $a=0$ and $k=0$  in which case the kernel is 
\[
 \ker (\slashed {\mathcal D}_s\vert_z)= \bigoplus_\ell \varphi_{s\ell}\cdot (S^+_\ell)_z,
\]
where,
\[
\varphi_{s\ell}:N \to \mathbb{R}, \quad  \varphi_{s\ell}(v)=  (s\lambda_\ell)^{\tfrac{n-m}{4}} e^{- \tfrac{1}{2} s \lambda_\ell \lvert v\rvert^2} .
\]
Consider the map $N \to Z\times [0, +\infty),\ v \mapsto (\pi(v), \lvert v\rvert)$. We then see that $\varphi_{s\ell}$ is the pullback under the preceding map, of $(z,r) \mapsto  (s\lambda_\ell)^{\tfrac{n-m}{4}}  e^{- \tfrac{1}{2}s\lambda_\ell r^2} $.   

\medskip
   
  \vspace{1cm}

\section{The harmonic oscillator in normal directions} 
\label{sec:harmonic_oscillator}

In this section we obtain estimates for the derivatives in the normal directions of $\xi: N \to \pi^*S^0$. We begin with the following definition:
\begin{defn}
\label{defn:eighted_norms_spaces}
Given $\xi\in C_c^\infty(N)$ and $k, l\in \mathbb{N}_0$, we define the norms
\begin{align*}
\|\xi\|_{0,2,k,0}^2:&= \int_N (1+ r^2)^{k+1}\lvert\xi\rvert^2\, d\mathrm{vol}^N, 
\\
\|\xi\|_{1,2,k,l}^2 :&=\|\xi\|_{0,2,k,0}^2 + \int_{N} \left[(1+r^2)^k \lvert\bar\nabla^{\mathcal V} \xi\rvert^2 + l(1+r^2)^{l-1} \lvert\bar\nabla^{\mathcal H} \xi\rvert^2\right]\, d\mathrm{vol}^N, 
\end{align*}
and define the spaces $L^2_k(N)$ and $W^{1,2}_{k, l}(N)$ by completion of $C_c^\infty(N)$ in each of these  respective norms. When $k=-1$ we set $L^2_{-1}(N) := L^2(N)$.  
\end{defn}

 Note that with these definitions, the space $W^{1,2}_{k,0}(N)$ has sections admitting weak derivatives in the normal directions. By Theorem~\ref{thm:approximation_theorem_for_weighted_spaces},
\begin{align*}
L^2_k(N) &:= \{ \xi\in L^2(N): \| \xi\|_{0,2,k,0}<\infty\}, 
\\
W^{1,2}_{k,l}(N) &:= \{ \xi\in  W^{1,2}(N): \| \xi\|_{1,2,k,l}<\infty\}, \quad l \in \mathbb{N},
\\
W^{1,2}_{k,0}(N) &:= \{ \xi\in L^2(N) : \bar\nabla^{\mathcal V}\xi\in L^2(N), \ \text{and}\ \| \xi\|_{1,2,k,0}<\infty\},
\end{align*}
for every $k\in \mathbb{N}_0$ and therefore we can use approximations by test functions with compact support. We have embeddings $W^{1,2}_{k+1, l}(N) \subset W^{1,2}_{k, l}(N)$. Furthermore $L^2_k(N),\ W^{1,2}_{k,0}(N)$ and $W^{1,2}_{k,l}(N),\ l \leq k+2$ are $C^\infty(Z;\mathbb{R})$-modules and multiplication by a fixed function is a continuous map in these spaces. Finally, the operator $\slashed D_s: L^2(N) \to L^2(N)$ is densely defined with domain $W^{1,2}_{0,0}(N)$ and satisfies $\slashed D_s( W^{1,2}_{k+1,0}(N)) \subset L^2_k(N)$, for every $k\geq -1$. We summarize now some basic properties following from well known facts about the harmonic oscillator:

\begin{prop}
\label{prop:vertical_cross}
\begin{enumerate}
\item Let $k\geq 0$ and set $C_0 = \min_{\ell , Z} \lambda_\ell$ and $C_1 = \max_{Z, \ell} \lambda_\ell$. There exist constant $C = C( n-m, C_0, C_1)>0 $ so that,
\begin{equation} 
\label{eq:elliptic_estimate1_k>=0}
 s\|r^{k+1}\xi\|_{L^2(N)} +  \| r^k \bar\nabla^{\mathcal V} \xi \|_{L^2(N)} \leq
 C\left(\sum_{u=0}^k s^{-\tfrac{u}{2}} \| r^{k-u} \slashed D_s \xi\|_{L^2(N)} +  s^{\tfrac{1-k}{2}}\|\xi \|_{L^2(N)}\right),
\end{equation} 
for every $s>0$ and every $\xi \in W^{1,2}_{k,0}(N;\pi^*(S^0\vert_Z))$.

\item The operator,
\[
\slashed D_s  : L^2(N; \pi^*( S^0\vert_Z))\to L^2(N; \pi^* (S^1\vert_Z)), 
\]
is closed and $\ker \slashed D_s \subset L^2(N; \pi^* S^0)$ is a closed subspace.

\item 
There exist $s_0>0$ so that, when $s>s_0$, the kernel of $\slashed D_s$,  is given explicitly by,  
\begin{equation}
\label{eq:L_2_kernel_D_s_calculation}
\ker \slashed D_s = \bigoplus_\ell \varphi_{s\ell}\cdot L^2 (Z; S^{0+}_\ell),
\end{equation}
and
\begin{equation}
\label{eq:kernel_D_s_calculation}
    \ker \slashed D_s \cap W^{1,2}(N; \pi^*(S^0\vert_Z)) = \bigoplus_\ell \varphi_{s\ell} \cdot W^{1,2}(Z; S^{0+}_\ell).
\end{equation}

\item
For every $s> 0$ and every $\xi\in (\ker \slashed D_s)^{\perp_{L^2}}$, we have the spectral estimate
\begin{equation}
\label{eq:spectral_estimate}
2s C_0\|\xi\|_{L^2(N)}^2 \leq \|\slashed D_s\xi\|_{L^2(N)}^2. 
\end{equation}

\item
$\slashed D_s$ has closed range and we have the relations,
\begin{equation}
\label{eq:Fredholm_alternative}
(\ker \slashed D_s)^{\perp_{L^2}} = \mathrm{Im\,}\slashed D_s^* \quad \text{and} \quad  (\ker \slashed D_s^*)^{\perp_{L^2}} = \mathrm{Im\,}\slashed D_s. 
\end{equation}
\end{enumerate}
\end{prop}

\begin{proof}
We will use Hitchin's dot notation $v_\cdot$ throughout the proof instead of the $\mathfrak{c}_N(v)$. Recall the decomposition \eqref{eq:decompositions_of_S_ell} into eigenbundles of $Q^0$ and further into eigenbundles of $C^0$. We work with smooth sections $\xi: N \to \pi^*S_{\ell t}^0$ with compact support. There we have
\begin{align}
 \label{eq:N_Weitzenboock_2} 
\slashed D_s^* \slashed D_s\xi =  (- \Delta  +  s^2 r^2 \lambda_\ell^2 - s(n-m-2t)\lambda_\ell )\xi.
\end{align}
Multiplying \eqref{eq:N_Weitzenboock_2} by $r^{2k}$, taking inner products with $\xi$ and integrating by parts, we obtain,

\begin{align*}
  \| r^k   \slashed D_s\xi \|_2^2 + 2k\langle  \slashed D_s \xi,  r^{2k-1} dr_\cdot \xi\rangle_2 
  =& \sum_\alpha(\| r^k \partial_\alpha \xi\|_2^2 + 2k\langle\partial_\alpha \xi, x_\alpha r^{2k-2} \xi\rangle_2)
\\
&+ s^2\|r^{k+1} \lambda_\ell \xi \|_2^2 +  s(2t-n+m)\langle r^{2k} \lambda_\ell\xi,\xi\rangle_2.  
\end{align*}
By applying Peter-Paul inequality to the first couple of cross terms and absorbing alike terms we obtain an estimate,
\begin{equation}
\label{eq:basic_ineq_for_S}
  s^2 C_0^2\|r^{k+1} \xi\|_2^2 +  \|r^k \bar\nabla^{\mathcal V} \xi\|_2^2 \leq 
  6(C_1+1)(n-m)(\|r^k\slashed D_s \xi\|_2^2 + \|(k + \sqrt{s} r) r^{k-1} \xi\|_2^2),
\end{equation}
for every $k\in \mathbb{N}_0$, every $t\in \{0, \dots, n-m\}$ and every $s>0$.  Estimate \eqref{eq:elliptic_estimate1_k>=0} is then proven, by induction on $k$, using \eqref{eq:basic_ineq_for_S} in inductive step.

The fact that $\slashed D_s: L^2(N; \pi^*(S^0\vert_Z)) \to L^2(N; \pi^*(S^0\vert_Z))$ is closed, follows from estimate \eqref{eq:elliptic_estimate1_k>=0} for $k=0$. Also $\ker \slashed D_s$ is the kernel of a closed operator and therefore, a closed subspace of $L^2$.
We have already shown that,
\[
\ker \slashed D_s\cap C^\infty (N; \pi^* (S^0\vert_Z)) = \bigoplus_\ell \varphi_{s\ell} \cdot C^\infty(Z; S^{0+}_\ell).
\]
Given $\xi = \varphi_{s\ell}\cdot \eta$, with $\eta \in C^\infty(Z; S^{0+}_\ell)$ and using the substitution $t = \sqrt{s \lambda_\ell} r $, the $L^2$-norms are computed explicitly to give, 
\begin{equation*}
\|\xi\|_{L^2(N)}^2 = \left\lvert S^{n-m-1}\right\rvert\left(\int_0^\infty t^{n-m-1} e^{-t^2}\, dr\right)  \|\eta\|_{L^2(Z)}^2
= \frac{1}{2} \left\lvert S^{n-m-1} \right\rvert \Gamma \left(\frac{n-m}{2} \right) \|\eta\|_{L^2(Z)}^2, 
\end{equation*}
that is,
\begin{equation}
  \label{eq:Gaussian_evaluation}
\|\xi\|_{L^2(N)}^2 = \pi^{\tfrac{n-m }{2}} \|\eta\|_{L^2(Z)}^2.  
\end{equation}
Since $\bar\nabla \xi = d \varphi_{s\ell } \otimes \eta + \varphi_{s\ell} \bar\nabla \eta$, a computation using  estimate \eqref{eq:elliptic_estimate1_k>=0} with $k=0, 1$ shows that  there exist constants $C_2, C_3$ with
\[
C_2 \| \eta\|_{W^{1,2}(Z)} \leq \|\xi\|_{W^{1,2}(N)} \leq \sqrt{s} C_3 \|\eta\|_{W^{1,2}(Z)}.
\]
Since $\ker \slashed D_s$ is the $L^2$ closure of  $ \bigoplus_\ell\varphi_{s\ell} \cdot C^\infty(Z; S^{0+}_\ell)$, equations \eqref{eq:L_2_kernel_D_s_calculation} and \eqref{eq:kernel_D_s_calculation} follow. 

Recall that, by the calculation of the spectrum of $\slashed D^2_s\vert_z$ in \eqref{eq:spectrum_of_Delta_s}, the constant $2s\min_{\ell , Z} \lambda_\ell$ is a lower bound.  Given now $\xi \in C^\infty(N; \pi^*(S^0\vert_Z)) \cap(\ker\slashed D_s)^{\perp_{L^2}}$, Lemma~\ref{lemma:restriction_to_fibers_of_N} implies that the restriction on the fiber $N_z,\  \xi_z \in (\ker\slashed D_s\vert_z)^{\perp_{L^2}}$. Using the spectral decomposition of $\Delta_s\vert_z : L^2(N_z) \to L^2(N_z)$ in \cite{p}[Th. 2.2.2], we have, 
\[
2s\min_{\ell , Z} \lambda_\ell \|\xi\|_2^2(z) \leq \langle \Delta_s \xi, \xi\rangle_2(z) = \| \slashed D_s \xi\|_2^2(z), 
\]
for every $z\in Z$. Hence integrating over $Z$, 
\[
2s\min_{\ell , Z} \lambda_\ell \|\xi\|_2^2 \leq \|\slashed D_s\xi\|_2^2. 
\]
By completion, this inequality holds for every $\xi\in  W^{1,2}_{1,0}(N; \pi^*(S^0\vert_Z)) \cap(\ker\slashed D_s)^{\perp_{L^2}}$. This proves \eqref{eq:spectral_estimate}.

Finally, estimate \eqref{eq:spectral_estimate} shows that $\slashed D_s$ has closed range. Relations \eqref{eq:Fredholm_alternative} then follow as general results for closed operators admitting adjoints with closed range (see \cite{k}[Ch. IV, Th. 5.13, pp.234]).
\end{proof}

\begin{rem}
\label{rem:observations_on_W_1,2_k_0}
\begin{enumerate}
\item We have inclusions,
\label{rem:kernel_infinite_inclusions}
\[
\ker\slashed D_s \subset \bigcap_{\ell \in \mathbb{N}_0} W^{1,2}_{\ell,0}(N).
\]

\item 
\label{rem:decompositions}
Because of Remark \ref{rem:observations_on_W_1,2_k_0} \eqref{rem:kernel_infinite_inclusions}, we have decompositions
\begin{align*}
L^2_k(N) &= (L^2_k(N)\cap\ker \slashed D_s) \oplus [L^2_k(N)\cap(\ker \slashed D_s)^{\perp_{L^2}}], \quad k\geq -1
\\
W^{1,2}_{k,0}(N) &= (W^{1,2}_{k,0}(N)\cap\ker \slashed D_s) \oplus [W^{1,2}_{k,0}(N)\cap(\ker \slashed D_s)^{\perp_{L^2}}], \quad k \geq 0.
\end{align*}
Given any $\xi \in L^2(N)$ we decompose it into, 
\[
\xi = \xi^0 + \xi^1, \quad \xi^0\in \ker \slashed D_s  \quad \text{and} \quad \xi^1 \in (\ker \slashed D_s)^{\perp_{L^2}}.
\]
Henceforth we denote these projections with these notations. 

\item \label{rem:module_structures} The subspace $\ker \slashed D_s$ is a $C^\infty(Z; \mathbb{R})$-module. The same holds true for $W^{1,2}_{k,\ell}(N; \pi^*(S^0\vert_Z)) \cap(\ker\slashed D_s)^{\perp_{L^2}}$, for every $k, \ell\geq 0$.

\item Analogue estimates and decompositions hold for $\slashed D_s^*$.
\end{enumerate}
\end{rem}

We used the following lemma:
\begin{lemma}
\label{lemma:restriction_to_fibers_of_N}
Let $\xi \in C^\infty(N; \pi^*(S^0\vert_Z)) \cap(\ker\slashed D_s)^{\perp_{L^2}}$. Then the restriction on the fiber $\xi\vert_{N_p}=:  \xi_p$ is a section of $(\ker\slashed D_s\vert_p)^{\perp_{L^2}}$, for every $p\in Z$.
\end{lemma}

\begin{proof}
Let $\varrho_\varepsilon(z) = \varrho( d_Z(z,p)/ \varepsilon)$ where $d_Z(z,p)$ is the distance of $z$ from $p$ in $Z$ and $\varrho$ is a cutoff function with $\rho(0)=1$ and $\mathrm{supp\,} \varrho \subset [0,1]$. For every $\eta \in \ker\slashed D_s$ we have $\varrho_\varepsilon \cdot \eta \in\ker\slashed D_s$.  By Fubini's Theorem,
\bigskip 

\begin{align*}
0 &= \frac{1}{\mathrm{vol}(B(p, \varepsilon)\cap Z)} \int_N \langle \xi, \varrho_\varepsilon \cdot \eta\rangle\, d\mathrm{vol}^N 
\\
&= \frac{1}{\mathrm{vol}(B(p, \varepsilon)\cap Z)} \int_{B(p, \varepsilon)\cap Z}\varrho_\varepsilon(z)\cdot \int_{N_z} \langle \xi, \eta\rangle\, dv dz \to \int_{N_p} \langle \xi, \eta\rangle\, dv, 
\end{align*}
as $\varepsilon \to 0^+$. Hence for every $p\in Z$, we obtain $\xi\vert_{N_p} \in (\ker(\slashed D_s)\vert_p)^{\perp_{L^2}}$. 
\end{proof}

As a corollary we prove the following proposition with a bootstrap argument on $k\in \mathbb{N}_0$ for $\slashed D_s$ on $W^{1,2}_{k,0}$-spaces:

\begin{prop}
\label{prop:bootstrap_on_k}
Assume $ s>0$ and $k\geq 0$. Then 
\[
\slashed D_s : W^{1,2}_{k,0}(N)\cap (\ker\slashed D_s)^{\perp_{L^2}} \to L^2_{k-1}(N)\cap (\ker\slashed D_s^*)^{\perp_{L^2}},
\]
defines an isomorphism of $C^\infty(Z;\mathbb{R})$-modules,  for every $k \geq 0$. The inverse operator $G_s$ is defined and is bounded. More precisely, there exist constant $C = C(n-m, C_0, C_1, k)>0$ so that every $\xi\in W^{1,2}_{k,0}(N)\cap (\ker\slashed D_s)^{\perp_{L^2}} $ obeys an estimate,
\begin{equation}
\label{eq:bootstrap_estimate}
 s\|r^{k+1}\xi\|_2 +  \| r^k \bar\nabla^{\mathcal V} \xi \|_2 \leq C\sum_{u=0}^k s^{-\tfrac{u}{2}} \| r^{k-u} \slashed D_s \xi\|_2.
\end{equation}
\end{prop}

\begin{proof}
The operator $\slashed D_s : W^{1,2}_{k,0}(N)\cap (\ker\slashed D_s)^{\perp_{L^2}} \to L^2_{k-1}(N)\cap (\ker\slashed D_s^*)^{\perp_{L^2}}$ is injective and we have to prove that it is surjective. Let $\eta \in  L^2_k(N)\cap (\ker\slashed D_s^*)^{\perp_{L^2}}$. By Proposition~\ref{prop:vertical_cross} \eqref{eq:Fredholm_alternative}, there exist  unique $\xi \in W^{1,2}_{0,0}(N)\cap  (\ker \slashed D_s)^{\perp_{L^2}}$ with $\slashed D_s\xi = \eta$. We prove that $\xi$ is of class $W^{1,2}_{k,0}$.

Let $\rho : [0, +\infty) \to [0,1]$ with $\mathrm{supp\,} \rho \subset [0,2]$ and $ \rho^{-1}(\{1\}) = [0,1]$ and $\lvert d\rho\rvert \leq 1$. Define sequence $\rho_j(r) = \rho(r/ j)$ so that $\rho_j \to 1$ uniformly on compact subsets of $N$ and $\lvert d\rho_j\rvert \leq 1/j$. Then $\xi_j = \rho_j \cdot \xi \in \bigcap_\ell W^{1,2}_{\ell,0}(N)$ and by Dominated convergence, $\xi_j \to \xi$ and $\bar\nabla_\alpha \xi_j \to \bar\nabla_\alpha \xi$ in $ L^2(N)$ for every $\alpha$. Setting $\eta_j = \slashed D_s \xi_j$, by Dominated convergence, we also have that $\eta_j \to \eta$ in $L^2(N)$. 

\hfill

\textit{Claim:} $\{\xi_j\}_j$ is a Cauchy sequence in $W^{1,2}_{k,0}$-norm.

\proof[Proof of claim] 
By \eqref{eq:elliptic_estimate1_k>=0} with $k=0$,
\[
s\|r(\xi_j- \xi)\|_2 \leq C(\|\eta_j - \eta\|_2 +  \sqrt{s}\|\xi_j- \xi\|_2),
\] 
so that $\xi_j \to \xi$ in $W^{1,2}_{0,0}(N)$. We work using induction. Suppose that $\{\xi_j\}_j$ is Cauchy in $W^{1,2}_{\ell-1,0}(N)$ for some $1\leq\ell\leq k$. Then $\xi \in W^{1,2}_{\ell-1,0}(N)$ and $\{r^\ell \eta_j\}_j$ is Cauchy in $L^2(N)$, since
\[
\|r^\ell (\eta_j - \eta)\|_2 \leq \|r^\ell d \rho_{j\cdot} \xi\|_2 + \| (1- \rho_j) r^\ell \eta\|_2 \leq \frac{1}{j} \|r^\ell  \xi\|_2 +  \| (1- \rho_j) r^\ell \eta\|_2 \to 0.
\]
The convergence holds since $r^\ell \xi \in  L^2(N)$ by inductive assumption and $r^\ell \eta \in L^2(N)$ by our initial assumption, so that Dominated convergence applies to the last term. But then, by \eqref{eq:elliptic_estimate1_k>=0} with $k=\ell$,
\begin{equation*} 
s\|r^{\ell+1}(\xi_j- \xi_i)\|_2 + \|r^\ell \bar\nabla^{\mathcal V} (\xi_j- \xi_i)\|_2 \leq 
C\left(\sum_{u=0}^\ell s^{-\tfrac{u}{2}}\|r^{\ell - u}(\eta_j - \eta_i)\|_2 +  s^{\tfrac{1-\ell}{2}}\|\xi_j- \xi_i\|_2\right),
\end{equation*}
which proves that $\{\xi_j\}_j$ is Cauchy in $W^{1,2}_{\ell,0}(N)$ and therefore $\xi \in W^{1,2}_{\ell,0}(N)$. It follows that $\{\xi_j\}_j$ is Cauchy in $W^{1,2}_{k,0}(N)$ and therefore $\xi \in W^{1,2}_{k,0}(N)$. Estimate \eqref{eq:bootstrap_estimate} follows from an application of \eqref{eq:elliptic_estimate1_k>=0}, followed by an application of \eqref{eq:spectral_estimate}. The boundedness of the inverse operator follows from the Open Mapping Theorem.
\end{proof}

\medskip
   
  \vspace{1cm}
  
\setcounter{equation}{0}   
\section{The operator \texorpdfstring{$\bar D^Z$}{} and the horizontal derivatives}
\label{sec:The_operator_bar_D_Z_and_the_horizontal_derivatives}

Throughout the paragraph we will work with function spaces of sections of the bundle $\pi^*(S^0\vert_Z)\to N$, unless if we specify otherwise. The domain of $\bar D^Z : L^2(N) \to L^2(N)$ is the space $W^{1,2}_{1,1}(N)$ so that $\bar D^Z$ is densely defined. Also $\bar D^Z(W^{1,2}_{k, l}(N)) \subset L^2_{\min\{k-2, l-2\}}(N)$, for every $k\geq 1,\ l\geq 1$. Fix coordinates $ (\pi^{-1}(U), (x_j, x_\alpha)_{j,\alpha})$, orthonormal frames $\{e_j\}_j$ with their lifts $\{h_j\}_j$ and orthonormal frame $\{\sigma_k\}_k$ on $S^0\vert_U$. Recall operator $D^Z_+$ from Definition~\ref{defn:Dirac_operator_component}.

\begin{prop}
\begin{enumerate}
\item There exist a constant $C>0$ so that, for every $s>0$ and every $\xi\in W^{1,2}_{1,1}(N)$, 
\begin{equation}
\label{eq:elliptic_estimate2}
\|\bar \nabla^{\mathcal H} \xi\|_{L^2(N)}^2 \leq   \|\bar D^Z \xi\|_{L^2(N)}^2 
+  C( s^{-1}\|\slashed D_s \xi\|_{L^2(N)}^2  + \|r\slashed D_s \xi\|_{L^2(N)}^2 + \|\xi \|_{L^2(N)}).
\end{equation}

\item The operators $\bar D^Z$ and $D^Z_+$ are linked in the following way; for every $\xi = \sum_\ell \xi_\ell,\ \xi_\ell \in C^\infty(Z; S^{0+}_\ell)$, the operator $\xi \mapsto \xi^0 = \sum_\ell \varphi_{s \ell} \cdot \xi_\ell \in C^\infty(N; \pi^* S^{0+}_\ell)$ is defined in Remark~\ref{rem:observations_on_W_1,2_k_0} \eqref{rem:decompositions}. We have 
\begin{equation}
\label{eq:br_D_Z_versus_D_Z_+}
\bar D^Z \xi^0 = \sum_\ell \mathfrak{c}_N(\pi^* d \ln \varphi_{s\ell}) P_\ell^{0+} \xi^0 + (D^Z_+ \xi)^0,
\end{equation}
where $P^{0+}_\ell S^0 \to S^{0+}_\ell$ is the orthogonal projection.

\item 
For every $k\geq 0$
\begin{equation}
\label{eq:aux_estimate_2}
\| r^k\bar\nabla^{\mathcal H}(\xi^0) \|_{L^2(N)} \leq Cs^{-\tfrac{k}{2}} (\|[(D^Z_++{\mathcal B}_{0+}^Z)\xi]^0\|_{L^2(N)} + \|\xi^0\|_{L^2(N)})
\end{equation}

\end{enumerate}
\end{prop}

\begin{proof}
 Let $\xi\in C^\infty_c(\pi^{-1}(U);\pi^*(S^0\vert_Z))$. Applying $L^2$-norms with $\xi$ in \eqref{eq:bar_D_z_Weitzenbock} and integrating by parts,
\[
\|\bar\nabla^{\mathcal H}\xi\|_{L^2(N)}^2  \leq \|\bar D^Z\xi\|_{L^2(N)}^2  - \langle c_T \bar\nabla \xi, \xi\rangle_{L^2(N)} +  C\|\xi\|_{L^2(N)}^2.
\]
However by the expression of the torsion $T$ in Proposition~\ref{prop:basic_extension_bar_connection_properties}, we have that
\[
\lvert\langle c_T \bar\nabla \xi, \xi\rangle_{L^2(N)}\rvert \leq C\|r \bar\nabla^{\mathcal V} \xi\|_{L^2(N)} \|\xi\|_{L^2(N)}.
\]
Applying Cauchy-Schwartz followed by estimate \eqref{eq:elliptic_estimate1_k>=0}, we have
\[
\|\bar\nabla^{\mathcal H}\xi\|_{L^2(N)}^2 \leq   \|\bar D^Z \xi\|_{L^2(N)}^2+  C( s^{-1}\|\slashed D_s \xi\|_{L^2(N)}^2  + \|r\slashed D_s \xi\|_{L^2(N)}^2 + \|\xi \|_{L^2(N)}),
\]
as required. 

Proving \eqref{eq:br_D_Z_versus_D_Z_+} is a straightforward calculation:
\begin{align*}
\bar D^Z \xi^0 &= \mathfrak{c}_N(h^j)( h_j(\varphi_{s\ell}) \cdot \xi_\ell + \varphi_{s\ell} \cdot \bar\nabla_{h_j} \xi_\ell)
\\
&= e_j(\varphi_{s\ell}) \cdot c_j\xi_\ell +  \varphi_{s\ell} \cdot c_j\bar\nabla_{e_j} \xi_\ell, \qquad (\text{since $\pi_* h_j = e_j$})
\\
&= \mathfrak{c}_N(\pi^* d \ln \varphi_{s\ell}) P_\ell^{0+} \xi^0 +  (D^Z_+ \xi)^0.
\end{align*}
Also, for every $k\geq 0$,
\begin{align*}
\| r^k\bar\nabla^{\mathcal H}(\xi^0) \|_{L^2(N)}&\leq  C\sum_\ell(\|r^k \lvert d_Z\varphi_{s\ell}\rvert \xi_\ell\|_{L^2(N)} + \|r^k\varphi_{s\ell} \bar\nabla\xi_\ell\|_{L^2(N)})
\\
&\leq  C\|r^k(1+s r^2)\xi^0\|_{L^2(N)} + C\|r^k (\bar\nabla\xi)^0\|_{L^2(N)})
\\
&\leq Cs^{-\tfrac{k}{2}}( \|\xi^0\|_{L^2(N)} + \|(\bar\nabla \xi)^0 \|_{L^2(N)})\qquad (\text{by \eqref{eq:elliptic_estimate1_k>=0}})
\\
&= C\pi^{\tfrac{n-m}{4}}s^{-\tfrac{k}{2}} \|\xi\|_{W^{1,2}(Z)},\qquad (\text{by \eqref{eq:Gaussian_evaluation}})
\\
&\leq C\pi^{\tfrac{n-m}{4}}s^{-\tfrac{k}{2}}(\|(D^Z_+ + {\mathcal B}^Z_{0+})\xi\|_{L^2(Z)} +\|\xi\|_{L^2(Z)})
\\
&=  Cs^{-\tfrac{k}{2}}(\|[(D^Z_+ + {\mathcal B}^Z_{0+})\xi]^0\|_{L^2(N)} + \|\xi^0\|_{L^2(N)}),\qquad (\text{by \eqref{eq:Gaussian_evaluation}}) 
\end{align*}
 where in the fifth row we used an elliptic estimate for $D^Z_+ + {\mathcal B}^Z_{0+}$. The proof is complete.
\end{proof}

We now prove an analogue of Proposition~\ref{prop:bootstrap_on_k} but involving  both the vertical and horizontal derivatives: 
\begin{prop}
\label{prop:horizontal_regularity}
Let $s> 0$ and $\xi \in W^{1,2}_{0,0}(N) \cap (\ker \slashed D_s)^{\perp_{L^2}}$ so that $\slashed D_s \xi\in W^{1,2}_{k-1,k}(N)$ for some $k\geq 1$. Then $\xi \in  W^{1,2}_{k,k+1}(N)$ and there are estimates:
\begin{equation}
\label{eq:horizontal_regularity_1}
\|\bar\nabla^{\mathcal H} \xi \|_{L^2(N)} \leq Cs^{-1/2}\left(\|\bar\nabla^{\mathcal H} ( \slashed D_s \xi)\|_{L^2(N)}  + \|\slashed D_s\xi\|_{L^2(N)}\right).
\end{equation}
\end{prop}
\begin{proof}
By Remark~\ref{rem:observations_on_W_1,2_k_0} \eqref{rem:module_structures}, the spaces $W^{1,2}_{k-1,k}(N)\cap (\ker \slashed D_s)^{\perp_{L^2}},\ k\geq 1$ are $C^\infty(Z; \mathbb{R})$-modules. Hence by multiplying $\xi$ with members of a partition of unity of $Z$ subordinated in a trivializations of $N \to Z$, we may assume without loss of generality that $\xi$ is supported in a single chart $\pi^{-1}(U)$ of $N$ with coordinates $(x_j, x_\alpha)_{j, \alpha}$.  First we prove that $\xi$ posses weak derivatives whose weighted Sobolev class we compute. 

Let $\eta\in C^\infty(N)$ be a test function and decompose $\eta= \eta^0 + \eta^1$.  Set $\xi_1 = G_s \eta^1$, where $G_s$ is the Green's operator of $\slashed D^*_s$, defined in Proposition~\ref{prop:bootstrap_on_k}. Then $\xi_1,\ \eta^1$ are both of class $W^{1,2}_{\ell,0}(N)$, for every $\ell\geq 0$. For fixed $j$,
\[
\langle \bar\nabla_{h_j} \eta, \xi\rangle_{L^2(N)} = \langle \bar\nabla_{h_j} \eta^0, \xi\rangle_{L^2(N)} + \langle \bar\nabla_{h_j} \eta^1, \xi\rangle_{L^2(N)}, 
\]
and we proceed to evaluate the two terms on the right hand side. 

Write $\eta^0 = \sum_\ell \varphi_{s\ell} \eta_\ell$ for some $\eta_\ell \in C^\infty(Z; S^{0+}_\ell)$. Then
\[
\bar\nabla_{h_j} \eta^0 = \sum_\ell \left[ \left( \frac{n-m}{2} - s \lambda_\ell r^2\right) \frac{e_j (\lambda_\ell)}{2\lambda_\ell} \varphi_{s\ell} \eta_\ell + \varphi_{s\ell}\bar\nabla_{e_j} \eta_\ell \right],
\]
where the last term of the right hand side is again an element of $\ker\slashed D_s$. Since $\xi \in (\ker\slashed D_s)^{\perp_{L^2}}$, we obtain 
\[
\langle \bar\nabla_{h_j} \eta^0, \xi\rangle_{L^2(N)} = \langle \eta^0,  K_j\xi\rangle_{L^2(N)},
\]
where
\[
K_j\xi := \sum_\ell  \left[ \frac{n-m}{2} - s \lambda_\ell r^2\right] \frac{e_j(\lambda_\ell)}{2\lambda_\ell} P_\ell^{0+}\xi.
\]
On the other hand, using the identity  \eqref{eq:cross_terms25} that is
\[
\slashed D_s^* (\bar\nabla_{h_j} \xi_1) = \bar\nabla_{h_j} ( \slashed D_s^* \xi_1) - sr[\bar\nabla_{h_j},  \bar A_r^*]  \xi_1, 
\]
and integrating by parts, we obtain,

\begin{align*}
\langle \bar\nabla_{h_j} \eta^1, \xi\rangle_{L^2(N)} &= \langle \bar\nabla_{h_j} \slashed D_s^* \xi_1, \xi\rangle_{L^2(N)}
\\
&= \langle \slashed D_s^* \bar\nabla_{h_j} \xi_1, \xi\rangle_{L^2(N)}  - s \langle r [\bar\nabla_{h_j},  \bar A_r^*]  \xi_1, \xi \rangle_{L^2(N)}
\\
&= - \langle \eta^1,  G_s^* (\bar\nabla_{h_j} \slashed D_s \xi)^1\rangle_{L^2(N)} + \langle \eta^1,  G_s^* (O(sr) \xi)^1\rangle_{L^2(N)}.
\end{align*}
Therefore, the weak derivative in the direction $h_j$ is defined as a section $\pi^{-1}(U) \to \pi^*S^{0+}$, by 
\begin{equation}
\label{eq:weak_derivatives}
\bar\nabla_{h_j} \xi = (K_j\xi)^0 + G_s^* (\bar\nabla_{h_j} \slashed D_s \xi + O(sr)\xi)^1 
\end{equation}
and the global horizontal derivative is defined as $\bar\nabla^{\mathcal H} \xi = h^j\otimes \bar\nabla_{h_j} \xi \in C^\infty(N; {\mathcal H}\otimes \pi^*S^{0+})$ and is supported again in $\pi^{-1}(U)$. 

Combining assumption $\slashed D_s \xi \in W^{1,2}_{k-1, k}(N) \subset L^2_{k-1}(N)$ with Proposition~\ref{prop:bootstrap_on_k}, we have that $\xi \in W^{1,2}_{k,0}(N)$. By expression \eqref{eq:weak_derivatives} we have $\bar\nabla^{\mathcal H} \xi \in W^{1,2}_{k-1,0}(N) \subset L^2_{k-1}(N)$ so that $\xi\in W_{k,k+1}^{1,2}(N)$ as required. 

We next prove estimate \eqref{eq:horizontal_regularity_1}. We start by obtaining estimates on the first term of \eqref{eq:weak_derivatives}: let $\eta\in \ker\slashed D_s \cap C^\infty(N)$, with $\|\eta\|_{L^2(N)} \leq1$ and $\eta = \sum_\ell\varphi_{s\ell} \cdot \eta_\ell$ for some $\eta_\ell \in C^\infty (Z; S^+_\ell)$. Then
\begin{align*}
\lvert\langle  K_j\xi , \eta\rangle_{L^2(N)}\rvert &= \lvert\langle \xi , K_j\eta\rangle_{L^2(N)}\rvert
\\
&\leq C \|\xi\|_{L^2(N)}  \|(1+ sr^2)\eta\|_{L^2(N)}
\\
&\leq  C\|\xi\|_{L^2(N)} \|\eta\|_{L^2(N)},\qquad  (\text{by using \eqref{eq:elliptic_estimate1_k>=0} on $\eta$ with $k=1$})
\\
&\leq C s^{-1/2}\|\slashed D_s\xi\|_{L^2(N)}. \qquad (\text{by \eqref{eq:spectral_estimate} applied on $\xi$})
\end{align*}
Therefore the $L^2$-norm of the projection to $\ker \slashed D_s$, estimates,
\begin{equation}
\label{eq:projection1_estimate}
\|(K_j\xi)^0 \|_{L^2(N)} = \sup_{\eta\in \ker \slashed D_s, \|\eta\|_2 \leq 1} \lvert\langle  K_j\xi , \eta\rangle_{L^2(N)}\rvert 
\leq  Cs^{-1/2}\|\slashed D_s\xi\|_2.
\end{equation}
To estimate the $L^2$-norm of the second term of the right hand side of \eqref{eq:weak_derivatives}, we use \eqref{eq:spectral_estimate} so that,
\[
\| G_s^* (\bar\nabla_{h_j} \slashed D_s \xi + O( sr)\xi)\|_{L^2(N)} \leq  Cs^{-1/2}\|\bar\nabla_{h_j} ( \slashed D_s \xi) +  O( sr)\xi\|_{L^2(N)}.
\]
By using \eqref{eq:bootstrap_estimate} on the second term of the right hand side of the preceding inequality, we obtain
\[
 \| G_s^* (\bar\nabla_{h_j} \slashed D_s \xi + O( sr)\xi)\|_{L^2(N)} \leq  Cs^{-1/2}\left(\| \bar\nabla_{h_j} ( \slashed D_s\xi)\|_{L^2(N)} + \|\slashed D_s\xi\|_{L^2(N)}\right).
\]
Combining the preceding estimate with \eqref{eq:projection1_estimate} we obtain \eqref{eq:horizontal_regularity_1}. 
\end{proof}

\medskip
   
  \vspace{1cm}
  
\setcounter{equation}{0}   
\section{Separation of the spectrum}
\label{sec:Separation_of_the_spectrum}

Our main goal for this section is to prove the Spectrum Separation Theorem stated in the introduction. For that purpose we will use the bundles $S^+_\ell$ introduced in Definition~\ref{defn:IntroDefSp} and define a space of approximate solutions to the equation $D_s\xi = 0$. The space of approximate solutions is linearly isomorphic to a certain ``thickening'' of $\ker D_s$ by ``low'' eigenspaces of $D_s^*D_s$ for large $s$. The same result will apply to $\ker D_s^*$. The ``thickening'' will occur by a phenomenon of separation of the spectrum of $D_s^*D_s$ into low and high eigenvalues for large $s$. The following lemma will be enough for our purposes:

\begin{lemma}
\label{lemma:spectrum}
Let $L : H\rightarrow H'$ be a densely defined closed operator with between the Hilbert spaces $H, H'$ so that $L^*L$ has descrete spectrum. Denote by $E_\mu$ the $\mu$- eigenspace of $L^*L$. Suppose $V$ is a $k$-dimensional subspace of $H$ so that  
\[
\lvert Lv\rvert^2 \leq C_1 \lvert v\rvert^2  \qquad\mbox{and}\qquad \lvert Lw\rvert^2 \geq C_2 \lvert w\rvert^2
\]
for every $v\in V$ and every $w\in V^\perp$. Then there exist consecutive eigenvalues $\mu_1, \mu_2$ of $L^*L$ so that $\mu_1 \leq C_1, \,\mu_2 \geq C_2$. If in addition $4C_1<C_2$ then the \textbf{orthogonal projection}
\[
P : \bigoplus_{\mu\leq \mu_1} E_\mu \rightarrow V,
\]
is an isomorphism.
\end{lemma}
\begin{proof}
Let $\mu_1$ be the $k$-th eigenvalue of the self-adjoint operator $L^*L$, counted with  multiplicity, and $\mu_2$ be the next eigenvalue. Denote by $G_k(H)$ the set of $k$- dimensional subspaces of $H$ and set $W = \oplus_{\mu\leq \mu_1} E_\mu$, also $k$-dimensional. By the Rayleigh quotients we have
\[
\mu_2 = \max_{S\in G_k(H)}\left\{ \inf_{v\in S^\perp, \lvert v\rvert=1}\lvert Lv\rvert^2\right\} \geq  \inf_{v\in V^\perp, \lvert v\rvert=1}\lvert Lv\rvert^2 \geq C_2,
\]
and also
\[
\mu_1 = \max_{S\in G_{k-1}(H)}\left\{ \inf_{v\in S^\perp, \lvert v\rvert=1}\lvert Lv\rvert^2\right\}.
\]
But for any $k-1$-dimensional subspace $S\subset H$ there exist a vector $v_S\in S^\perp\cap V$ of  unit length so that
\[
\mu_1 \leq \max_{S\in G_{k-1}(H)}\left\{\lvert Lv_S\rvert^2\right\}\leq C_1,
\]
as required. Finally, given $w\in W$ write $w = v_0 + v_1$ with $v_0 = P(w)$ and $v_1\in V^\perp$. Then 
\begin{equation*} 
C_2\lvert w - P(w)\rvert^2 = C_2\lvert v_1\rvert^2 \leq 
\lvert Lv_1\rvert^2\leq 2(\lvert Lw\rvert^2 + \lvert Lv_0\rvert^2) \leq 2(\mu_1 + C_1)\lvert w\rvert^2 \leq 4C_1 \lvert w\rvert^2,
\end{equation*} 
and so $\lvert 1_W - P\rvert^2 \leq 4\tfrac{C_1}{C_2}$. If additionally $4C_1<C_2$ and $P(w)=0$ for some $w\neq 0$ then 
\[ 
\lvert w\rvert^2 = \lvert w-P(w)\rvert^2 \leq \lvert 1_W - P\rvert^2\lvert w\rvert^2 < \lvert w\rvert^2,
\]
a contradiction. Hence $P$ is injective and by dimension count an isomorphism.  
\end{proof}

We have to construct an appropriate space $V$ that will be viewed as the subspace of approximate solutions to the problem $L\xi = D_s \xi =0$. This is achieved by the following splicing construction on a fixed $m$ - dimensional component $Z=Z_\ell$ of the critical set $Z_{\mathcal A}$. Recall the subspaces $S^{+i} = \bigoplus_\ell S^{+i}_\ell$ with projections $P^{i+} = \sum_\ell P^{i+}_\ell$ and $P^{i-} :S^i \vert_Z \to  (S^{i+})^\perp \subset S^i,\ i=0,1$. Recall also the operators $D^Z_+$ and ${\mathcal B}^Z_+$ introduced in Definition~\ref{defn:Dirac_operator_component}. Fix a cutoff function $\rho_\varepsilon^Z: X \to [0,1]$,  supported in $B_Z(2\varepsilon) = \{ p\in X: r_Z(p) < 2\varepsilon\}$, where $r_Z$ is the distance function of a component $Z$ and taking the value $1$ in $B_Z(\varepsilon)$, so that $\lvert d\rho_\varepsilon\rvert \leq 1/\varepsilon$. Define also the bundle map ${\mathcal P}_s: \pi^*(S^{0+}\vert_Z) \to \pi^*(S^1\vert_Z)$, by
\begin{align*}
{\mathcal P}_s:=  \sum_\ell c_Z(d_Z (\ln\varphi_{s \ell})) P_\ell^{0+} + P^1\circ \left({\mathcal B}^0+ \frac{1}{2} s r^2 \bar A_{rr} \right) - {\mathcal B}^Z_{0+},
\end{align*}
where  ${\mathcal B}^Z_{0+}$ is introduced in Definition~\ref{defn:Dirac_operator_component}.

\begin{defn}
\label{def:space_of_approx_solutions}
Given $s>0,\ \ell$ and section $\xi\in \ker (D^Z_+ + {\mathcal B}^Z_{0+}),\ \xi = \sum_\ell \xi_\ell$, set 
\[
\xi^0 := \sum_\ell \varphi_{s\ell} \cdot\xi_\ell,
\]
and define $\xi^1 \in (\ker\slashed D_s)^{\perp_{L^2}}$ and $\xi^2\in C^\infty(N; \pi^*(S^0\vert_Z)^\perp)$, by solving
\begin{align}
\label{eq:balansing_condition_1}
\slashed D_s \xi^1  + {\mathcal P}_s\xi^0 &=0,
\\
\label{eq:balansing_condition_2}
s\bar{\mathcal A} \xi^2 + \left(1_{E^1\vert_Z} - P^1\right)\circ \left( {\mathcal B}^0 + \frac{1}{2}s r^2 \bar A_{rr}\right) \xi^0 &=0. 
\end{align}

We define an approximate low eigenvector of $\tilde D_s$, by 
\[
\xi_s := \xi^0 + \xi^1 + \xi^2 \in C^\infty(N; \pi^*(S^0\vert_Z)).
\]
Given $\varepsilon>0$, we define the spaces of approximate low eigenvectors of $D_s$ by,
\begin{align*}
V_{s,\varepsilon}^Z &:= \left\{ (\rho_\varepsilon^Z\cdot{\mathcal I}\circ\tau)(\xi_s \circ \exp^{-1}) : \xi \in \ker (D^Z_+ + {\mathcal B}^Z_+) \right \} 
\\
V_{s, \varepsilon} &:= \bigoplus_{ Z\in \mathrm{Comp}(Z_{\mathcal A})}  V_{s,\varepsilon}^Z,
\end{align*}
where  ${\mathcal I}$ is introduced in \eqref{eq:exp_diffeomorphism} and $\tau: C^\infty (\mathcal{N}_\varepsilon; \pi^*(E\vert_Z)) \to C^\infty( \mathcal{N}; \tilde E\vert_{\mathcal{N})_\varepsilon}$ is the parallel transport map with respect to the $\tilde\nabla^{\tilde E}$ introduced in \eqref{eq:parallel_transport_map}. We have analogue definitions of approximate low eigenvector for $D_s^*$ and we denote the subspace of approximate solutions by $W^Z_{s, \varepsilon}$. 
\end{defn}

Note that by expansion \eqref{eq:taylorexp} elements of $W^Z_{s, \varepsilon}$ will be associated to sections in the kernel 
of $D^{Z*}_+ + {\mathcal B}^Z_{1+}$. Elements of $V_{s, \varepsilon}$ are smooth sections of the bundle $E^0\to X$ that are compactly supported on the tubular neighborhood $B(Z, 2\varepsilon) \subset X$ of $Z$. Let 
\[
V_{s, \varepsilon}^\perp :=V_{s, \varepsilon}^{\perp_{L^2}}\cap W^{1,2}(X; E). 
\]

\begin{theorem}
\label{Th:hvalue}
If $D_s$ satisfies Assumptions~\ref{Assumption:transversality1}-\ref{Assumption:stable_degenerations} then there exist $\varepsilon_0>0$ and $s_0>0$ and constants $C_i = C_i(s_0)>0, \,i =1,2$  so that for every $0<\varepsilon<\varepsilon_0$ and every $s>s_0$, 
\begin{enumerate}[(a)]
\item for every $\eta \in V_{s, \varepsilon}$, 
\begin{align}
\label{eq:est1}
\|D_s\eta\|_{L^2(X)} \leq C_1 s^{-1/2}\|\eta\|_{L^2(X)},
\end{align}

\item or every $\eta \in V_{s, \varepsilon}^\perp$,
\begin{align}
\label{eq:est2}
\| D_s\eta\|_{L^2(X)} \geq  C_2\|\eta\|_{L^2(X)}.
\end{align}
\end{enumerate}
Since the $L^2$-adjoint operator $D_s^*$satisfies the same assumptions, the constants $s_0,\, C_1,\, C_2$ can be chosen to satisfy simultaneously the analogue estimates for $D_s^*$ in place of $D_s$ and the space $W^Z_{s, \varepsilon}$ in place of $V^Z_{s, \varepsilon}$.
\end{theorem}

\begin{rem}
\label{rem:comments_on_Th_hvalue}
\begin{enumerate}
\item It suffices to prove estimate \eqref{eq:est2} for an $L^2$-dense subspace of $V_{s, \varepsilon}^{\perp_{L^2}}$. Also notice that $V_{s, \varepsilon}^{Z_1}$ is $L^2$-perpendicular to $V_{s, \varepsilon}^{Z_2}$ for $Z_1 \neq Z_2$ since their corresponding sections have disjoint supports.

\item($L^2$-norms using different densities)
Estimate \eqref{eq:est1} refers to sections supported in tubular neighborhoods of the components of the critical set. Also in Lemma~\ref{lemma:local_implies_global}, we prove that we only have to show estimate \eqref{eq:est2} for sections supported in these tubular neighborhoods. Let $\xi\in C^\infty_c(\mathcal{N}; \pi^*(E\vert_Z))$ and $\eta = ({\mathcal I} \circ \tau) \xi \circ \exp^{-1}$.  We have that, 
\[
\int_X \lvert\eta\rvert^2 \, d\mathrm{vol}^X = \int_{\mathcal{N}_\varepsilon} \lvert\xi\rvert^2\, d\mathrm{vol}^{\mathcal{N}_\varepsilon} 
\]
and that, 
\[
\int_X \lvert D_s \eta\rvert^2 \, d\mathrm{vol}^X = \int_{\mathcal{N}_\varepsilon} \lvert\tilde D_s (\tau \xi)\rvert^2 \, d\mathrm{vol}^{\mathcal{N}_\varepsilon}.
\]
By estimate  \eqref{eq:density_comparison}, these norms are equivalent to the corresponding $L^2(N)$ norms in the total space $N$ with volume form $d\mathrm{vol}^N$, of the normal bundles of the components of the critical set $Z_{\mathcal A}$. The later norms are the ones used in the consequent analytic proofs.
\end{enumerate}
\end{rem}

The rest of this section is devoted to proving estimate \eqref{eq:est1}. The proof of estimate \eqref{eq:est2} will be given in the next section. The following lemma establishes existence and uniqueness of equations \eqref{eq:balansing_condition_1} and \eqref{eq:balansing_condition_2} and provide estimates of the solutions: 
\begin{lemma}
\label{lem:balancing_condition}
For every section $\xi\in \ker (D^Z_+ + {\mathcal B}^Z_{0+})$, there exist unique sections $\psi \in (\ker \slashed D_s^*)^{\perp_{L^2}}\cap \left(\bigcap_{k,l} W^{1,2}_{k,l}(N; \pi^*(S^0\vert_Z))\right)$ and $\zeta \in \bigcap_{k,l} W^{1,2}_{k,l}(N; \pi^*(S^0\vert_Z)^\perp)$ satisfying equations \eqref{eq:balansing_condition_1} and \eqref{eq:balansing_condition_2} respectively. Moreover $\psi$ and $\zeta$ obey the following estimates,

\begin{align}
\label{eq:Gaussian_estimate_1}
\|\psi\|_{L^2(N)} &\leq C s^{-1/2}\|\xi^0\|_{L^2(N)},
\\
\label{eq:Gaussian_estimate_2}
s\|r^{k+1}\psi\|_{L^2(N)} + \|r^k \bar\nabla^{\mathcal V} \psi\|_{L^2(N)}  &\leq C s^{-\tfrac{k}{2}}\|\xi^0\|_{L^2(N)}, \, k\geq 0,
\\
\label{eq:Gaussian_estimate_3}
\|\bar\nabla^{\mathcal H} \psi\|_{{L^2(N)}} &\leq  C s^{-1/2}\|\xi^0\|_{L^2(N)},
\\
\label{eq:Gaussian_estimate_4}
s\|r^{k+1}\zeta \|_{L^2(N)}+ (k+1)\|r^k \bar\nabla^{\mathcal V} \zeta\|_{L^2(N)}  &\leq s^{-\tfrac{k+1}{2}}\|\xi^0\|_{L^2(N)}, \, k\geq -1,
\\
\label{eq:Gaussian_estimate_5}
\|r^k \bar\nabla^{\mathcal H} \zeta\|_{L^2(N)} &\leq Cs^{-1 - \tfrac{k}{2}}\| \xi^0\|_{L^2(N)}, \, k\geq 0. 
\end{align} 
\end{lemma}

\begin{proof}
First we claim that
\[
{\mathcal P}_s\xi^0  \in (\ker \slashed D_s^*)^{\perp_{L^2}} \cap W^{1,2}_{k,0}(N),\ \forall k\in \mathbb{N}_0.
\]
By direct calculation, for every $\ell$,
\begin{equation}
\label{eq:H_derivatives_of_kernel_sections}
d_Z(\ln\varphi_{s\ell}) =   \left(\frac{n-m}{2} - s \lambda_\ell r^2\right)\frac{d_Z\lambda_\ell}{2\lambda_\ell}.
\end{equation}
This is a Hermite polynomial in $r$ that is $L^2$ perpendicular to the $\ker \slashed D_s$: an arbitrary section in $\ker \slashed D_s$ is a section of the form $\phi_{s\ell'}\cdot \theta$, with $\theta\in L^2(Z ; S_{\ell'}^+)$. By applying change of coordinates $\{y_\alpha = \sqrt{s \lambda_\ell} x_\alpha\}_\alpha$,
\begin{equation*}
\langle c_Z(d_Z(\ln \varphi_{s\ell}))P^{1+}_\ell \xi^0  , \varphi_{s\ell'}\cdot \theta\rangle_{L^2(N)} 
= \delta_{\ell , \ell'} \int_N  \left(\frac{n-m}{2} -   \lvert y\rvert^2\right)\frac{1}{2\lambda_\ell} e^{-\lvert y\rvert^2} \langle c_Z(d \lambda_\ell) \xi, \theta\rangle\, d\mathrm{vol}^N.
\end{equation*}
The integral is zero when $\ell = \ell'$, since its polar part is
\[
\int_0^\infty \left(\frac{n-m}{2} - r^2\right) r^{n-m-1} e^{-r^2} dr = \left(\frac{n-m}{2}\right) \Gamma\left(\frac{n-m}{2}\right) -  \Gamma\left(\frac{n-m}{2} +1\right) =0.
\]
This proves the claim for the first term in the expression of ${\mathcal P}_s\xi^0$. For the rest of the terms we calculate the the $L^2$-projections to the subspace $\ker \slashed D_s$ as
\begin{equation*}
\langle {\mathcal B}^0 \xi^0 , \varphi_{s \ell'} \cdot \theta \rangle_{L^2(N)} = s^{\tfrac{n-m}{2}}\sum_\ell \int_N (\lambda_\ell \lambda_{\ell'})^{\tfrac{n-m}{4}} \exp\left(-\frac{1}{2}s(\lambda_\ell + \lambda_{\ell'})\lvert x\rvert^2\right) \langle {\mathcal B}^0 \xi_\ell, \theta\rangle\, d\mathrm{vol}^N.
\end{equation*}
By applying change of coordinates $\{y_\alpha = [\tfrac{1}{2}s (\lambda_\ell+ \lambda_{\ell'})]^{1/2} x_\alpha\}_\alpha$ and then calculating the polar part of the resulting integral, we obtain the constant $C_{\ell, \ell'}$ of Definition~\ref{defn:Dirac_operator_component}. Similarly, using the expression  $r^2 \bar A_{rr} = x_\alpha x_\beta \bar A_{\alpha \beta}$, 
\begin{equation*}
s\langle r^2 \bar A_{rr} \xi^0 , \varphi_{s \ell'} \cdot \theta \rangle_{L^2(N)} =  s^{\tfrac{n-m}{2}+1}\sum_{\ell, \alpha, \beta} \int_N (\lambda_\ell \lambda_{\ell'})^{\tfrac{n-m}{4}} x_\alpha x_\beta \exp\left(-\frac{1}{2}s(\lambda_\ell + \lambda_{\ell'})\lvert x\rvert^2\right) \langle \bar A_{\alpha \beta} \xi_\ell, \theta\rangle\, d\mathrm{vol}^N.
\end{equation*}
The resulting integral is zero when $\alpha \neq \beta$ and when $\alpha = \beta$ we apply the same change of variables as with the preceding integral and we write the resulting integral as product of $n-m-1$  one dimensional Gaussian integrals, one for each normal coordinate different than the $x_\alpha$-coordinate and an integral that is $\frac{1}{2}\Gamma\left(\frac{3}{2}\right)$. The final result of the integration along the normal directions is $\frac{C_{\ell, \ell'}}{\lambda_\ell + \lambda_{\ell'}}$. Hence the $L^2$-projection of the term $\left({\mathcal B}^0 + \frac{1}{2} sr^2 \bar A_{rr}\right)\xi^0$ onto $\ker \slashed D_s$ is ${\mathcal B}^Z_{0+} \xi^0$. This finishes the proof of the claim.

The claim together with Proposition~\ref{prop:bootstrap_on_k} prove that equation \eqref{eq:balansing_condition_1} has a unique solution $\psi$. Since the right hand side, is a Hermite polynomial in $r$, it belongs to the space $\bigcap_{k,l} W^{1,2}_{k,l}(N)$. By Propositions~\ref{prop:horizontal_regularity}, we have that $\psi\in \bigcap_{k,l} W^{1,2}_{k,l}(N)$ and by \eqref{eq:spectral_estimate} and \eqref{eq:elliptic_estimate1_k>=0} with $k=1$ we obtain,  
\begin{align}
\sqrt{s}\|\psi\|_{L^2(N)}  &\leq C\|\slashed D_s \psi\|_{L^2(N)} \nonumber
\\
\|\slashed D_s \psi\|_{L^2(N)} &= C \| {\mathcal P}_s \xi^0\|_{L^2(N)} \leq  C\| (1+ sr^2)  \xi^0\|_{L^2(N)} \leq C\|\xi^0\|_{L^2(N)},\label{eq:aux_estimate_1}
\end{align}
so that estimate \eqref{eq:Gaussian_estimate_1} is proved. Estimate \eqref{eq:Gaussian_estimate_2} follows by combining \eqref{eq:bootstrap_estimate} with \eqref{eq:aux_estimate_1}.

Finally
\begin{align*}
\|\bar\nabla^{\mathcal H}\psi\|_{L^2(N)} &\leq  Cs^{-1/2}\left(\|\bar\nabla^{\mathcal H} ({\mathcal P}_s \xi^0)\|_{L^2(N)} + \| \slashed D_s \psi\|_{L^2(N)}\right), \quad \text{by \eqref{eq:horizontal_regularity_1},}
\\
&\leq  Cs^{-1/2}\left(\| (1+sr^2)(\lvert\xi^0\rvert+\lvert \bar\nabla^{\mathcal H} (\xi^0)\rvert)\|_{L^2(N)} + \| \slashed D_s \psi\|_{L^2(N)}\right)
\\
&\leq  Cs^{-1/2}\|\xi^0\|_{L^2(N)}, \qquad (\text{by \eqref{eq:elliptic_estimate1_k>=0},  \eqref{eq:aux_estimate_1}and \eqref{eq:aux_estimate_2}})
\end{align*}
so that estimate \eqref{eq:Gaussian_estimate_3} is proved.

Equation \eqref{eq:balansing_condition_2} is solvable in $\zeta$ since ${\mathcal A}$ is invertible in the subspaces where the equation is defined. Moreover we have poinwise estimates
\begin{align*}
sr^{k+1}\lvert\zeta\rvert &\leq (1+ sr^2) r^{k+1} \lvert\xi^0\rvert
\\
r^k \lvert\bar\nabla^{\mathcal V} \zeta\rvert &\leq Cr^{k+1} (1+sr^2) \lvert\xi^0\rvert  
\\
sr^k\lvert\bar\nabla^{\mathcal H} \zeta\rvert&\leq r^k(1 + sr^2)[\lvert\xi^0\rvert +  \lvert\bar\nabla^{\mathcal H} (\xi^0)\rvert]. 
\end{align*}
Applying $L^2$-norms and using \eqref{eq:elliptic_estimate1_k>=0} and \eqref{eq:aux_estimate_2}, we obtain estimates \eqref{eq:Gaussian_estimate_4} and \eqref{eq:Gaussian_estimate_5}.
\end{proof}

We now proceed to the
\begin{proof}[Proof of estimate \eqref{eq:est1} in Theorem \ref{Th:hvalue}]
Choose $\xi = \sum_\ell \xi_\ell\in \ker (D^Z_+ + {\mathcal B}^Z_{0+})$ and set $\eta = {\mathcal I} \tilde \eta$ where $\tilde\eta = \rho_\varepsilon^Z \cdot \tau \xi_s$ and $\xi_s = \xi^0 + \xi^1 + \xi^2$. Denote $\rho_\varepsilon^Z = \rho_\varepsilon$. The Taylor expansion from  Corollary~\ref{cor:taylorexp} gives
\begin{equation}
\label{eq:cor}
\begin{aligned}
(\tilde D + s \tilde {\mathcal A})\eta &= d\rho_{\varepsilon\cdot} \tau\xi_s  +  \rho_\varepsilon \cdot \tau O\left( r^2\partial^{\mathcal V}  + r\partial^{\mathcal H}  + r  + sr^3\right)\xi^0 
\\
&\quad + \rho_\varepsilon \cdot \tau O\left( r\partial^{\mathcal V} + \partial^{\mathcal H} + 1 + sr^2\right)\xi^1
\\
&\quad +  \rho_\varepsilon \cdot \tau O\left(\partial^{\mathcal V} + \partial^{\mathcal H} + 1 + sr\right)\xi^2.
\end{aligned}
\end{equation}

Here the lower order terms vanished because we used the simplifications coming from equations  
\begin{align*}
\slashed D_s \xi^0 &=0
\\ 
(\bar D^Z +{\mathcal B}^Z_{0+}) \xi^0 &= \left(\sum_\ell c_Z(d_Z(\ln \varphi_{s\ell})) P^{0+}_\ell\right) \xi^0, \qquad  (\text{by \eqref{eq:br_D_Z_versus_D_Z_+}}) 
\end{align*}
together with the equations \eqref{eq:balansing_condition_1} for $\xi^1$ and and \eqref{eq:balansing_condition_2} for $\xi^2$. By using \eqref{eq:elliptic_estimate1_k>=0} and   \eqref{eq:aux_estimate_2} for $\xi^0$ and \eqref{eq:Gaussian_estimate_1} up to  \eqref{eq:Gaussian_estimate_3} for $\xi^1$ and estimates \eqref{eq:Gaussian_estimate_4} and \eqref{eq:Gaussian_estimate_5} for $\xi^2$, we obtain that the $L^2(N)$ norm of the error terms of expansion \eqref{eq:cor} are bounded by $Cs^{-1/2}\|\xi^0\|_{L^2(N)}$.

Finally, because $d\rho$ has support outside the $\varepsilon$-neighborhood of $Z_{\mathcal A}$, the $L^2(N)$ norm of the first term on the right hand side is bounded as 
\begin{equation*}
\int_\mathcal{N} \left\lvert d\rho_{\varepsilon\cdot}\tau\xi^0\right\rvert^2\, d\mathrm{vol}^N 
\leq C\|\xi\|_{L^2(Z)}^2 \int_{\varepsilon\sqrt{s}}^\infty r^{n-m-1}e^{-r^2}\, dr \leq C \pi^{\tfrac{m-n}{2}} e^{- \tfrac{s\varepsilon^2}{2}}\|\xi^0\|_{L^2(N)}^2, \quad \text{(by \eqref{eq:Gaussian_evaluation}),}
\end{equation*}
and
\[
\|d\rho_{\varepsilon\cdot}\tau(\xi^1 + \xi^2)\|_{L^2(N)} \leq  Cs^{-1}\|\xi^0\|_{L^2(N)}^2.
\]
Putting all together we obtain, 
\begin{equation}
\label{eq:aux_estimate_3}
\int_\mathcal{N} \lvert(\tilde D + s \tilde {\mathcal A}) \tilde\eta\rvert^2 \, d\mathrm{vol}^N \leq C s^{-1}\|\xi^0\|_{L^2(N)}^2.
\end{equation}
Finally we show that
\begin{equation}
\label{eq:aux_estimate_4}
\|\xi^0\|_{L^2(N)} \leq 2\|\rho_\varepsilon \cdot\xi^0\|_{L^2(N)} 
\end{equation}
for every $s$ sufficiently large. Indeed, by using the change to $\{y_\alpha = \sqrt{s \lambda_\ell} x_\alpha\}_\alpha$ on each component of the section $\xi^0 = \sum_\ell (\xi_\ell)^0$, we estimate
\begin{equation*}
\|\rho_\varepsilon \cdot (\xi_\ell)^0\|^2_{L^2(N)} \geq \int_{B(Z, \varepsilon \sqrt{s})} e^{-\lvert y\rvert^2}\lvert\xi_\ell\rvert^2\, d\mathrm{vol}^N
=  \left\lvert S^{n-m-1}\right\rvert\|\xi_\ell\|^2_{L^2(Z)} \int^{\epsilon\sqrt{s}}_0r^{n-m-1}e^{- r^2}\, dr.
\end{equation*}
But there exist $s_0 = s_0(\varepsilon)>0$, so that 
\[
\int^{\epsilon\sqrt{s}}_0 r^{n-m-1}e^{- r^2}\, dr> \tfrac{1}{4} \int_0^\infty r^{n-m-1}e^{- r^2}\, dr,
\]
for every $s>s_0$. Estimate \eqref{eq:aux_estimate_4} follow.
Finally, using that $\tilde\eta = \rho_\varepsilon \tau(\xi^0+ \xi^1+ \xi^2)$ and \eqref{eq:aux_estimate_4},
\[
\|\tilde\eta\|_{L^2(N)}^2 \geq \frac{1}{4}\|\xi^0\|_{L^2(N)}^2  + 2 \langle\rho_\varepsilon \xi^0, \rho_\varepsilon (\xi^1+ \xi^2)\rangle_{L^2(N)}
\]
where the cross terms estimate as, 
\begin{align*}
\lvert 2 \langle\rho_\varepsilon \xi^0, \rho_\varepsilon (\xi^1+ \xi^2)\rangle_{L^2(N)}\rvert &\leq  \frac{1}{16}\|\xi^0\|_{L^2(N)}^2 + 16\| \xi^1 + \xi^2\|_{L^2(N)}^2
\\
&\leq   (\frac{1}{16} + Cs^{-1})\|\xi^0\|_{L^2(N)}^2, \qquad (\text{by \eqref{eq:Gaussian_estimate_1} and  \eqref{eq:Gaussian_estimate_4}})
\\
&\leq \frac{1}{8}\|\xi^0\|_{L^2(N)}^2,
\end{align*}
for $s>0$ sufficiently large. Therefore, we obtain
\[
\|\tilde\eta\|_{L^2(N)}^2 \geq \frac{1}{8} \|\xi^0\|_{L^2(N)}^2. 
\]
Combining this last inequality with \eqref{eq:aux_estimate_3}, we obtain 
\[
\int_\mathcal{N}\lvert\tilde D_s \tilde \eta\rvert^2\, d\mathrm{vol}^N \leq C s^{-1}\int_\mathcal{N} \lvert\tilde \eta\rvert^2\, d\mathrm{vol}^N.
\]
By \eqref{eq:density_comparison}, the volume densities $d\mathrm{vol}^\mathcal{N}$ and $d\mathrm{vol}^N$ are equivalent. Therefore, the preceding inequality holds for the density $d\mathrm{vol}^\mathcal{N}$. Since $\| D_s \eta\|_{L^2(X)} = \|\tilde D_s \tilde\eta\|_{L^2(\mathcal{N})}$ and $\|\eta\|_{L^2(X)} = \|\tilde \eta\|_{L^2(\mathcal{N})}$, estimate \eqref{eq:est1} follows.
\end{proof}

\vspace{2mm}

Applying Lemma \ref{lemma:spectrum} we get a proof of Spectrum Separation Theorem stated in the introduction:
\begin{proof}[Proof of Spectral Separation Theorem.] 
By Theorem~\ref{Th:hvalue}, we may choose $s_0>0$ so that the constants satisfy $4\tfrac{C_1}{s}< C_2$ for every $s> s_0$. We apply Lemma~\ref{lemma:spectrum} with $L = D_s$ and $H = L^2(X, E^0),\, H' = L^2(X, E^1)$ and $V_{s, \varepsilon}$ constructed in Definition~\ref{def:space_of_approx_solutions}. The analogue versions of Theorem~\ref{Th:hvalue} for the adjoint operator $D_s^*$ are also applied. As a result, we obtain that the for the first eigenvalue $\lambda_0$ of $D_s^* D_s$ and $D_s D^*_s$ satisfying $\lambda_0 \leq C_1 s^{-1/2}$, the orthogonal projections,
\begin{equation*}
\Pi^0 :\mathrm{span}^0(s, \lambda_0) \simeq V_{s, \varepsilon} 
\simeq \bigoplus_{Z \in \mathrm{Comp}(Z_{\mathcal A})} \ker\{ D^Z_+ + {\mathcal B}^Z_{0+} : C^\infty(Z ; S^{0+}\vert_Z) \rightarrow C^\infty(Z; S^{1+}\vert_Z)\},
\end{equation*}
and 
\begin{equation*}
\Pi^1:\mathrm{span}^1(s, \lambda_0) \simeq W_{s, \varepsilon} 
\simeq \bigoplus_{Z \in \mathrm{Comp}(Z_{\mathcal A})} \ker\{ D^{Z*}_+ + {\mathcal B}^Z_{1+} : C^\infty(Z ; S^{1+}\vert_Z) \rightarrow C^\infty(Z; S^{0+}\vert_Z)\},
\end{equation*}
are both linear isomorphisms, for every $s>s_0$. It also follows that $N^i(s,\lambda_0) = N^i(s, C_1 s^{-1/2})$,\ i =0,1. This completes the proof when $s> s_0$. By replacing ${\mathcal A}$ with $- {\mathcal A}$, the preceding considerations prove an analogue theorem for $D_s$ with $s$ being large and negative. The bundle where approximate sections are constructed, is then changing from $S^{i+}$ to $S^{i-},\, i =0,1$.   
\end{proof}

\begin{rem} Combining  Theorem \ref{Th:hvalue} with the proof of Lemma \ref{lemma:spectrum} we obtain a bound on the error of the approximate eigensections that is if $\|\xi\|_{L^2(X)} =1$ and $D_s\xi = 0$ we have, 
\[
\| \xi - \Pi^i(\xi)\|_{L^2(X)}^2 \leq \frac{4C_1}{sC_2} \rightarrow 0 \quad \mbox{as}\quad s\rightarrow \infty,\quad i =0,1.
\]

\end{rem}

\medskip
   
  \vspace{1cm}
  
\setcounter{equation}{0}   
\section{A Poincar\'{e} type inequality}

\label{sec:A_Poincare_type_inequality}

This section is entirely devoted to the proof of estimate \eqref{eq:est2} of Theorem \ref{Th:hvalue}. We start by reducing the proof of the estimate to a local estimate for sections supported in the tubular neighborhood $B_{Z_{\mathcal A}}(4\varepsilon)$:  

\begin{lemma}
\label{lemma:local_implies_global}
If estimate (\ref{eq:est2}) is true for $\eta\in V_{s, \varepsilon}^\perp$ supported in $B_{Z_{\mathcal A}}(4\varepsilon)$ then it is true for every $\eta\in V_{s, \varepsilon}^\perp$.
\end{lemma}

\begin{proof}
Let $\eta\in V_{s, \varepsilon}^\perp$ and recall the cutoff $\rho_\varepsilon$ used in Definition~\ref{def:space_of_approx_solutions} and define $\rho' = \rho_{2\varepsilon} :X\rightarrow [0,1]$, a bump function supported in $B_{Z_{\mathcal A}}(4\varepsilon)$ with $\rho' \equiv 1$ in $B_{Z_{\mathcal A}}(2\varepsilon)$.  Write $\eta = \rho'\eta + (1-\rho')\eta = \eta_1 + \eta_2$ with supp $\eta_1 \subset B_{Z_{\mathcal A}}(4\varepsilon)$ and supp $\eta_2\subset X\backslash B_{Z_{\mathcal A}}(4\varepsilon) = \Omega(4\varepsilon)$. Then
\begin{equation}
\label{eq:interpol}
\|D_s\eta\|^2_{L^2(X)} = \|D_s\eta_1\|^2_{L^2(X)} + \|D_s\eta_2\|^2_{L^2(X)} +2\langle D_s\eta_1, D_s\eta_2\rangle_{L^2(X)}.
\end{equation}
 Since $\rho'\cdot \rho_\varepsilon = \rho_\varepsilon$ we have $\eta_1\in V_{s, \varepsilon}^\perp$ and by assumption there exist $C_1=C_1(\varepsilon)>0$ and $s_0 = s_0(\varepsilon)>0$ so that 
\[
\|D_s\eta_1\|^2_{L^2(X)}\geq C_1\|\eta_1\|^2_{L^2(X)}
\]
for every $s> s_0$.  Also, by a concentration estimate
\[ 
\|D_s\eta_2\|^2_{L^2(X)} \geq s^2\|{\mathcal A}\eta_2\|^2_{L^2(X)} - s\lvert\langle \eta_2, B_{\mathcal A}\eta_2\rangle_{L^2(X)}\rvert \geq (s^2 \kappa^2_{2\varepsilon} - s C_0)\|\eta_2\|^2_{L^2(X)}.
\]
with constants $\kappa_{2\varepsilon}$ and $C_0$ as in \eqref{eq:useful_constants}. To estimate the cross terms we calculate  
\[
D_s \eta_1\, =\, \rho' D_s\eta + (d\rho')_\cdot \eta,\qquad D_s\eta_2\, =\, (1-\rho') D_s\eta - (d\rho')_\cdot \eta
\]
and hence
\begin{align*}
\langle D_s\eta_1, D_s\eta_2\rangle_{L^2(X)} =& \int_X\rho'(1-\rho') \left\lvert D_s\eta\right\rvert^2\, d\mathrm{vol}^X + \int_X(1-2\rho')\langle (d\rho')_\cdot \eta, D_s\eta\rangle\, d\mathrm{vol}^X - \int_X \left\lvert(d\rho')_\cdot \eta\right\rvert^2\, d\mathrm{vol}^X
\\
\geq&  -\frac{1}{2}\|D_s\eta\|^2_{L^2(X)} -\frac{3}{2} \int _X \left\lvert(d\rho')_\cdot\eta\right\rvert^2\, d\mathrm{vol}^X.
\end{align*}
 We used that $\lvert ab\rvert\leq \tfrac{1}{2}(a^2 + b^2)$ and that $(1-2\rho')^2\leq 1$. But $(d\rho')_\cdot \eta $ is supported in $\Omega(2\varepsilon)$ hence by a concentration estimate applied again
\begin{align*}
\int _X \left\lvert(d\rho')_\cdot\eta\right\rvert^2\, d\mathrm{vol}^X &\leq C_\varepsilon \int_{\Omega(2\varepsilon)}\left\lvert\eta\right\rvert^2\, d\mathrm{vol}^X \leq \frac{C_\varepsilon}{s^2 \kappa^2_{2\varepsilon}} \|D_s \eta\|^2_{L^2(X)} + \frac{C_\varepsilon C_0}{s \kappa^2_{2\varepsilon}}\|\eta\|_{L^2(X)}^2 
\\
&\leq \frac{1}{3} \|D_s\eta\|^2_{L^2(X)} + \frac{C_\varepsilon}{s}\|\eta\|_{L^2(X)}^2
\end{align*}
for $s$ large enough. Hence
\[
\langle D_s\eta_1, D_s\eta_2\rangle_{L^2(X)} \geq - \|D_s\eta\|^2_{L^2(X)} - \frac{C_\varepsilon}{s}\|\eta\|_{L^2(X)}^2.
\]
Substituting to \eqref{eq:interpol} and absorbing the first term in the left hand side there is an $s_1 = s_1(\varepsilon)$ with
\begin{align*}
3\|D_s\eta\|^2_{L^2(X)}&\geq  \|D_s\eta_1\|^2_{L^2(X)} + (s^2 \kappa^2_{2\varepsilon} - s C_0)\|\eta_2\|^2_{L^2(X)} - \frac{C_\varepsilon}{s}\|\eta\|_{L^2(X)}^2
\\ 
&\geq C_1(\|\eta_1\|^2_{L^2(X)} + \|\eta_2\|^2_{L^2(X)}) - \frac{C_\varepsilon}{s}\|\eta\|_{L^2(X)}^2
\\
& \geq C_1 \|\eta\|^2_{L^2(X)},
\end{align*}
for every $s\geq s_1$. 
\end{proof}

Since $L^2$-norms are additive on sections with disjoint supports, we can work with $\eta \in V_s^\perp$ so that the support of $\eta$ lies in a tubular neighborhood $B(Z, 4\varepsilon)$ of some individual singular component $Z$ of $Z_{\mathcal A}$. There the distance function $r$ from the set $Z$ is defined and we have the following  generalization of  estimate \eqref{eq:est2} of Theorem \ref{Th:hvalue}:

\begin{lemma}
\label{lemma:aux_estimate}
There exist $\varepsilon_0>0$ and $C= C(\varepsilon_0)$ so that for every $\varepsilon \in (0, \varepsilon_0)$ there exist $s_0(\varepsilon)>0$ with the following property: for every $s > s_0$ and every $\eta \in V_{s, \varepsilon}^\perp \cap W^{1,2}_0(B_Z(4\varepsilon) ; E^0)$ and every $k=0,1,2$,
\begin{align}
\label{eq:est2'}
\| D_s\eta\|_{L^2(X)} \geq   s^{k/2}C\| r^k\eta\|_{L^2(X)}.
\end{align}
\end{lemma}

As mentioned in Remark~\ref{rem:comments_on_Th_hvalue}, the exponential map identifies diffeomorphically this neighborhood to a neighborhood $\mathcal{N}_{2\varepsilon}$ of the zero section on the total space of the normal bundle $N$ of $Z$ and by using the maps introduced in \eqref{eq:exp_diffeomorphism}, we can prove  estimate \eqref{eq:est2'} for the diffeomorphic copies $\tilde D_s$ of $D_s$ and $\tilde V_{s, \varepsilon}^\perp\cap W^{1,2}_0(\mathcal{N}_{2\varepsilon} ; \tilde E^0\vert_{\mathcal{N}_{2\varepsilon}}) $ of $V_{s, \varepsilon}^\perp \cap W^{1,2}_0( B_Z(4\varepsilon) ; E^0\vert_{B_Z(4\varepsilon)})$. The tubular neighborhood admits two different volume elements namely the pullback volume $ d\mathrm{vol}^\mathcal{N} = \exp^* d\mathrm{vol}^X$ and the volume form $d\mathrm{vol}^N$ introduced in Appendix~\ref{subApp:The_expansion_of_the_volume_from_along_Z}. The corresponding densities are equivalent per Appendix~\eqref{eq:density_comparison}. We prove  estimate \eqref{eq:est2'} for the $L^2(N)$-norms and function spaces induced by $d\mathrm{vol}^N$.  

In the following lemmas until the end of the paragraph, we use the following conventions: given $\eta =\tau \xi$ and $\xi\in C^\infty_c(\mathcal{N}_{2\varepsilon}; \pi^*(E^0\vert_Z))$, we decompose $\eta = \eta_1 + \eta_2 = \tau (\xi_1 + \xi_2)$ where $\xi_1 = P^0 \xi $ and $\xi_2 = (1_{E^0\vert_Z} - P^0)\xi$ are sections of the bundles $\pi^*(S^0\vert_Z)$ and $\pi^*(S^0\vert_Z)^\perp$ respectively. It follows that $\xi_1$ and $\xi_2$ belong in different $\text{Cl}_n$-modules. We further decompose $\xi_1 = \xi^0_1  +\xi^1_1$ where $\xi_1^0 \in \ker \slashed D_s$ and $\xi^1_1 \in (\ker \slashed D_s)^{\perp_{L^2}} \cap \left(\bigcap_{k,l} W^{1,2}_{k,l}(N; \pi^*(S^0\vert_Z)) \right)$. We have the following basic estimate:

\begin{lemma}
There exist $s_0, \varepsilon_0>0$ and a constant $C=C(s_0, \varepsilon_0) >0$ so that for every $\varepsilon \in (0, \varepsilon_0)$, every $s>s_0$ and every $\eta =\tau \xi \in C^\infty_c(\mathcal{N}_{2\varepsilon}; \pi^*(E^0\vert_Z))$, we have an estimate,
\begin{multline}
\label{eq:Taylor_estimate}
\| \slashed D_s \xi_1^1\|_{L^2(N)} + \| \slashed D_0 \xi_2\|_{L^2(N)} + s\|\bar{\mathcal A} \xi_2\|_{L^2(N)}
\\
+ \| \bar \nabla^{\mathcal H} \xi_1^0\|_{L^2(N)} +\| \bar \nabla^{\mathcal H} \xi_1^1\|_{L^2(N)} +  \| \bar \nabla^{\mathcal H} \xi_2\|_{L^2(N)}  
\\
\leq C(\|\tilde D_s \eta\|_{L^2(N)} + \|\eta\|_{L^2(N)}).
\end{multline}

\end{lemma}
\begin{proof}
It is enough to prove the estimate for a section $\xi$, supported in a bundle chart $(\pi^{-1}(U), (x_j, x_\alpha)_{j,\alpha})$.  We first prove the auxiliary estimates, 
\begin{align}
\label{eq:auxiliary_estimate1}
\|\slashed D_s \xi_1\|_{L^2(N)}^2 + \|\bar D^Z\xi_1\|_{L^2(N)}^2 &\leq C(\|(\slashed D_s + \bar D^Z)\xi_1\|^2_{L^2(N)} + \|\xi_1\|_{L^2(N)}^2),
\end{align} 
then
\begin{equation}
    \label{eq:auxiliary_estimate1.5}
\|\bar\nabla^{\mathcal H} \xi_1^0\|^2_{L^2(N)} + \|\bar\nabla^{\mathcal H} \xi_1^1\|^2_{L^2(N)}
\leq C (\|\slashed D_s \xi_1^1\|^2_{L^2(N)} + \|\bar\nabla^{\mathcal H} \xi_1\|^2_{L^2(N)} + \|\xi_1\|^2_{L^2(N)}), 
\end{equation} 
and
\begin{equation}
\label{eq:auxiliary_estimate2}
\| \slashed D_0 \xi_2\|_{L^2(N)}^2 +  \|\bar D^Z\xi_2\|_{L^2(N)}^2 + s^2 \|\xi_2\|_{L^2(N)}^2 
\leq  C\|(\slashed D_0 + \bar D^Z + s\bar {\mathcal A})\xi_2\|^2_{L^2(N)}.
\end{equation}

In proving \eqref{eq:auxiliary_estimate1} we expand the right hand side and we are led in estimating the cross term,  
\begin{equation}
\label{eq:auxiliary_cross_terms}
\begin{aligned}  
2\langle \slashed D_s \xi_1, \bar D^Z\xi_1\rangle_{L^2(N)} =&\, \langle (\slashed D_s ^*\bar D^Z+ \bar D^{Z*} \slashed D_s) \xi_1,\xi_1\rangle_{L^2(N)} 
\\
=&\, \langle (\slashed D_s^* \bar D^Z+ \bar D^{Z*} \slashed D_s) \xi^0_1,\xi^0_1\rangle_{L^2(N)}  
\\
&+ 2 \langle (\slashed D_s^* \bar D^{Z*}+ D^{Z*} \slashed D_s) \xi^0_1,\xi^1_1\rangle_{L^2(N)}  
\\
&+ \langle (\slashed D_s^* \bar D^Z+ \bar D^{Z*} \slashed D_s) \xi^1_1,\xi^1_1\rangle_{L^2(N)}. 
\end{aligned} 
\end{equation}
We further decompose $\xi^0_1 = \sum_\ell \xi^0_{1 \ell}$ where $\xi^0_{1 \ell} \in\bigcap_{k,l} W^{1,2}_{k,l}(N; \pi^*S_\ell^{0+})$. Then by using \eqref{eq:cross_terms}, the pointwise inner product is
\[
 \langle (\slashed D_s^* \bar D^Z+ D^{Z*} \slashed D_s) \xi^0_1,\xi^0_1\rangle(v) = - s \sum_{\alpha, \ell} x_\alpha \langle \mathfrak{c}_N (h^\alpha) \mathfrak{c}_N( \pi^* d\lambda_\ell)\xi^0_{1\ell}, \xi^0_{1\ell} \rangle(v)=0, 
\]
because $\mathfrak{c}_N(h^\alpha) \mathfrak{c}_N (\pi^*d\lambda_\ell)\xi^0_{1\ell}$ belong to the eigenspace of $C^0$ with eigenvalue $(n-m-1)\lambda_\ell$ and $\xi^0_{1\ell}$ belong to the eigenspace with eigenvalue $(n-m)\lambda_\ell$. By Proposition~\ref{prop:Weitzenbock_identities_and_cross_terms} \eqref{eq:cross_terms1} the operator $\slashed D_s^* \bar D^Z+ \bar D^{Z*} \slashed D_s$ is a bundle map with coefficients vanishing up to $sO(r)$ as $r \to 0^+$. Therefore the remaining terms in \eqref{eq:auxiliary_cross_terms} are estimated above by,
\begin{align*}
Cs(\| \xi^0_1\|_{L^2(N)} + \|\xi^1_1\|_{L^2(N)}) \|r\xi^1_1\|_{L^2(N)} &\leq C\|\xi_1\|_{L^2(N)}\|\slashed D_s \xi^1_1\|_{L^2(N)} 
\\
&\leq \delta \|\slashed D_s \xi_1\|_{L^2(N)}^2 + C\delta^{-1}\|\xi_1\|_{L^2(N)}^2,
\end{align*}
where we used \eqref{eq:bootstrap_estimate} with $k=0$. The cross term \eqref{eq:auxiliary_cross_terms} therefore estimates as
\[
2\lvert\langle \slashed D_s \xi_1, \bar D^Z \xi_1\rangle_{L^2(N)}\rvert \leq  \delta\|\slashed D_s \xi_1\|^2_{L^2(N)} + C\delta^{-1} \|\xi_1\|^2_{L^2(N)},  
\]
for an updated constant $C$. Combining the preceding estimates and choosing $\delta>0$ small enough as suggested by the preceding constants, the term $\|\slashed D_s \xi_1\|^2_{L^2(N)}$ is absorbed to the left hand side of the expansion thus arriving at inequality \eqref{eq:auxiliary_estimate1}.

To prove inequality \eqref{eq:auxiliary_estimate1.5} we again expand the right hand side and this time, we estimate the cross term,  
\[
2\vert\langle \bar \nabla^{\mathcal H} \xi^0_1, \bar\nabla^{\mathcal H} \xi_1^1\rangle_{L^2(N)}\vert  =  2\vert\langle \bar\nabla^{{\mathcal H}*}\bar \nabla^{\mathcal H} \xi^0_1, \xi_1^1\rangle_{L^2(N)}\vert.
\]
Using again the decomposition $\xi_1^0 = \sum_\ell \xi_{1\ell}^0$, with $\xi_{1\ell}^0 = \varphi_{s \ell} \cdot \zeta_\ell$, we calculate explicitly,
\begin{align*}
\bar\nabla^{{\mathcal H}*}\bar \nabla^{\mathcal H} \xi^0_1 &= - \sum_{i,\ell} \bar \nabla_{h_i} \bar\nabla_{h_i} \xi_{1\ell}^0
\\
&=- \sum_{i,\ell}\bar \nabla_{h_i}(M_{i\ell}\cdot \xi^0_{s\ell} + \varphi_{s\ell} \cdot \bar\nabla_{e_i} \zeta_\ell)
\\
&= -\sum_{i,\ell}[(e_i(M_{s\ell}) - M_{i\ell}^2) \cdot \xi^0_{s\ell} +2M_{i\ell} \cdot\bar\nabla_{h_i}\xi_{s\ell}^0] + (\bar\nabla^*\bar\nabla \zeta)^0,    
\end{align*}
where $\zeta = \sum_\ell \zeta_{\ell}$ and $M_{i\ell} :=\left(\frac{n-m}{2} - sr^2 \lambda_\ell\right) \frac{e_i(\lambda_\ell)}{2\lambda_\ell}$.
We then estimate,
\begin{align*}
\lvert\langle  \bar\nabla^{{\mathcal H}*}\bar \nabla^{\mathcal H} \xi^0_1 , \xi_1^1\rangle_{L^2(N)}\rvert \leq &\, C\|(1+ sr^2 + s^2 r^4)\xi^0_1\|_{L^2(N)}\| \xi_1^1\|_{L^2(N)}  +C\|\bar\nabla^{\mathcal H}\xi_1^0\|_{L^2(N)} \|(1+sr^2)\xi_1^1\|_{L^2(N)}
\\
 \leq &\, C\|\xi_1^0\|_{L^2(N)} \|\xi_1^1\|_{L^2(N)}  + C\|\bar\nabla^{\mathcal H}\xi_1^0\|_{L^2(N)} [\|\xi_1^1\|_{L^2(N)}+ (\varepsilon + s^{-1/2})\|\slashed D_s\xi_1^1\|_{L^2(N)}],
\end{align*}
and by applying Cauchy-Schwartz,
\begin{equation}
    \label{eq:auxiliary_cross_terms_3}
2\lvert\langle  \bar\nabla^{{\mathcal H}*}\bar \nabla^{\mathcal H} \xi^0_1 , \xi_1^1\rangle_{L^2(N)}\rvert \leq  \frac{1}{2} \|\bar\nabla^{\mathcal H} \xi_1^0\|^2_{L^2(N)} + C(\|\slashed D_s \xi_1^1\|^2_{L^2(N)} + \|\xi_1\|^2_{L^2(N)}),
\end{equation}
where in the third line of the preceding estimate, we applied \eqref{eq:elliptic_estimate1_k>=0} on $\xi_1^0$ with $k=1$ and $k=3$ and we applied \eqref{eq:bootstrap_estimate} with $k=1$ on $\xi_1^1$. Absorbing the first term of \eqref{eq:auxiliary_cross_terms_3} on the left hand side of the expansion, we obtain \eqref{eq:auxiliary_estimate1.5}.

By  Proposition~\ref{prop:Weitzenbock_identities_and_cross_terms} \eqref{eq:cross_terms3}, the operator $ \bar{\mathcal A}^*\circ(\slashed D_0 + \bar D^Z) +  (\slashed D_0 + \bar D^Z)^* \circ \bar{\mathcal A}$ is a bundle map and we estimate
\begin{align*}
2 \lvert\langle (\slashed D_0 + \bar D^Z) \xi_2, \bar{\mathcal A} \xi_2 \rangle_{L^2(N)}\rvert \leq C \| \xi_2\|_{L^2(N)}^2 \leq C \| \bar{\mathcal A} \eta_2\|_{L^2(N)}^2,
\end{align*}
and by Proposition~\ref{prop:Weitzenbock_identities_and_cross_terms} \eqref{eq:cross_terms2} $\slashed D_0^* \bar D^Z + \bar D^{Z*} \slashed D_0 \equiv 0$, so that
\begin{multline*}
\|\slashed D_0 \xi_2\|_{L^2(N)}^2+ \|\bar D^Z\xi_2\|_{L^2(N)}^2 + s^2 \|\bar{\mathcal A} \xi_2\|_{L^2(N)}^2
\\
\leq \|(\slashed D_0 + \bar D^Z + s \bar{\mathcal A}) \xi_2\|_{L^2(N)}^2 + 2s \lvert\langle (\slashed D_0 + \bar D^Z) \xi_2, \bar{\mathcal A} \xi_2 \rangle_{L^2(N)}\rvert
\\
\leq \|(\slashed D_0 + \bar D^Z + s \bar{\mathcal A}) \xi_2\|_{L^2(N)}^2 +C s\| \bar{\mathcal A} \eta_2\|_{L^2(N)}^2.
\end{multline*}
Choosing $s>0$ large enough we absorb the term $s\| \bar{\mathcal A} \eta_2\|_{L^2(N)}^2$ to the left hand side of the preceding inequality thus, obtaining \eqref{eq:auxiliary_estimate2}.

Finally we prove \eqref{eq:Taylor_estimate}:  by combining expansions \eqref{eq:taylorexp} and  \eqref{eq:taylorexp1} and rearranging terms, we obtain
\begin{equation}
\label{eq:Taylorexp1}
\tau[(\slashed D_s + \bar D^Z)\xi_1 + (\slashed D_0 + \bar D^Z + s\bar{\mathcal A})\xi_2] =  D_s \eta +    O(r^2 \partial^{\mathcal V} + r\partial^{\mathcal H} + 1) \xi  + O(sr^2)\xi_1 +  O(sr) \xi_2.
\end{equation}
Taking $L^2$-norms on and applying \eqref{eq:auxiliary_estimate1} and \eqref{eq:auxiliary_estimate2}
\begin{multline}
\label{eq:long_inequality}
\| \slashed D_s \xi_1\|_{L^2(N)} + \| \bar D^Z \xi_1\|_{L^2(N)} + \| \slashed D_0 \xi_2\|_{L^2(N)} + \| \bar D^Z \xi_2\|_{L^2(N)} + s\|\bar{\mathcal A} \xi_2\|_{L^2(N)} 
\\
\leq C\| D_s \eta\|_{L^2(N)} + \|  O( r^2\partial^{\mathcal V} + r\partial^{\mathcal H} + 1 + sr^2) \tau\xi_1\|_{L^2(N)} 
\\
+ \|O( r^2\partial^{\mathcal V} + r\partial^{\mathcal H} +1+sr)\tau\xi_2\|_{L^2(N)} + C\|\xi_1\|_{L^2(N)}.
\end{multline}
The $L^2$ norm  over the tubular region $\mathcal{N}_\varepsilon$ of the error terms for $\xi_1$ in the right hand side of \eqref{eq:Taylorexp1} are  estimated as,
\begin{equation*}
\|O(r^2) \partial^{\mathcal V} \xi_1\|_{L^2(N)} \leq C[(\varepsilon^2+ \varepsilon s^{-1/2} + s^{-1} )\|\slashed D_s \xi_1\|_{L^2(N)} + s^{-1/2}\|\xi_1\|_{L^2(N)}],
\end{equation*}
by applying \eqref{eq:elliptic_estimate1_k>=0} with $k=2$, then 
\begin{equation*}
\|r \partial^{\mathcal H} \xi_1\|_{L^2(N)} \leq C\varepsilon[(\varepsilon+ s^{-1/2})\|\slashed D_s\xi_1 \|_{L^2(N)} +\|\bar D^Z \xi_1\|_{L^2(N)} + \|\xi_1\|_{L^2(N)})],
\end{equation*}
by using \eqref{eq:elliptic_estimate2}, and
\begin{equation*}
s\|O(r^2) \xi_1\|_{L^2(N)} \leq  C[(s^{-1/2} + \varepsilon)\|\slashed D_s \xi_1\|_{L^2(N)} + \|\xi_1\|_{L^2(N)}],
\end{equation*}
by using \eqref{eq:elliptic_estimate1_k>=0} with $k=1$. For the corresponding error terms for $\xi_2$,
\[
\| O(r^2 \partial^{\mathcal V}+ r \partial^{\mathcal H} + 1+ sr) \xi_2\|_2 \leq C\varepsilon(\|\slashed D_0 \eta_2\|_2 +  \|\bar \nabla^{\mathcal H} \xi_2\|_2 + s\|\bar{\mathcal A} \xi_2\|_2 ) + C\|\xi_2\|_2. 
\]
By combining the estimates of the error terms and the preceding estimates with \eqref{eq:long_inequality}  and absorbing terms on the left hand side and choosing at first $\varepsilon$ small enough and then $s>0$ large enough, we obtain
\[
\| \slashed D_s \xi_1\|_2 + \| \bar \nabla^{\mathcal H} \xi_1\|_2 + \| \slashed D_0 \xi_2\|_2 + \| \bar \nabla^{\mathcal H} \xi_2\|_2 + s\|\bar{\mathcal A} \xi_2\|_2 
\leq C(\| D_s \eta\|_2 + \|\eta\|_2).
\]
Finally, by combining the preceding estimate with \eqref{eq:auxiliary_estimate1.5}, we obtain estimate \eqref{eq:Taylor_estimate}.
\end{proof}

In the proofs of the following lemmas, we use the re-scaling $\{y_\alpha = \sqrt{s} x_\alpha\}_\alpha$. This is independent of the Fermi coordinates defining a global diffeomorphism of the tubular neighborhoods $\mathcal{N}^\varepsilon \to \mathcal{N}^{\sqrt{s}\varepsilon}$. The volume element re-scales accordingly as  $d\mathrm{vol}^N_x = s^{\frac{m-n}{2}} d\mathrm{vol}^N_y$. Recall the orthogonal projections $P^i: E^i\vert_Z \to S^i\vert_Z,\, i =0,1$ introduced in Section~\ref{sec:Concentration_Principle_for_Dirac_Operators}.

\begin{lemma}
\label{lemma:perps_of_app_solutions}
Suppose there exists sequence $\{s_j\}_j$ of positive numbers with no accumulation point and a sequence $\{\eta_j\} \subset W^{1,2}_0(\mathcal{N}_{2\varepsilon} , \tilde E^0\vert_{\mathcal{N}_{2\varepsilon}})$  satisfying $\sup_j \|\eta_j\|_{L^2(N)} < \infty$ and $\|\tilde D_{s_j}\eta_j\|_{L^2(N)}^2 \rightarrow 0$  as $j\rightarrow \infty$. Then $P^1 \eta_j \to 0$ in $L^2(N)$ and, after re-scaling, there exist a subsequence of $\{\xi_j\}$, of $P^0 \eta_j$ that converges $L^2_{\text{loc}}$-strongly and $W^{1,2}$-weakly on $N$, to a section $\sum_\ell\phi_{1\ell}\xi_\ell$ with $\xi_\ell\in W^{1,2}(Z, S^{0+}_\ell)$. Furthermore the section $\bar\xi = \sum_\ell \xi_\ell :Z \to S^{0+}$ satisfies, 
\begin{equation}
\label{eq:limiting_conditions0}
 (D^Z_+ + {\mathcal B}^Z_{0+})\bar\xi = 0. 
\end{equation}
\end{lemma}

\begin{proof}
We decompose $\eta_j = \eta_{j1} + \eta_{j2}$ into sections of $S^0$ and $(S^0)^\perp$ correspondingly. We re-scale the sequence $\{\eta_j\}_j$ around the critical set $Z$: recall the Fermi coordinates $(\mathcal{N}_U ,\,(x_k,\, x_\alpha)_{k,\alpha})$ and the parallel transport map $\tau$ from \eqref{eq:parallel_transport_map}. We define the re-scaled sections of $\pi^*S^0 \to\mathcal{N}^{2\sqrt{s_j}\varepsilon}_U$, by 
\begin{equation}
\label{eq:modified_section}
\tau\xi_{jl}(x_k,y_\alpha) = s_j^{\tfrac{m-n}{4}} \eta_{jl}\left(x_k,\frac{y_\alpha}{\sqrt{s_j}}\right),\ l=1,2 \quad \text{and}\quad \xi_j = \xi_{j1} + \xi_{j2},
\end{equation}
The re-scaling is defined independently of the choice of the Fermi coordinates allowing the components $\xi_j$ over the various charts $\{\mathcal{N}_U: U\subset Z\}$ to patch together, defining sections over the tubular neighborhood $\mathcal{N}_{\sqrt{s_j}\varepsilon}$. We decompose further $\xi_{j1} = \xi_{j1}^0 + \xi_{j1}^1$ with $\xi_{j1}^0\in \ker \slashed D_1$ and $\xi_{j1}^1 \perp_{L^2} \overline{\ker\slashed D_1}^{L^2}$. The operator $\bar D^Z$ and the derivatives $\bar\nabla^{\mathcal H}$ remain invariant under the change of variables $\{y_\alpha = x_\alpha\sqrt{s_j}\}_\alpha$ while the operator $\slashed D_{s_j}$ changes to $\sqrt{s_j}\slashed D_1$. Changing variables on the left hand side of  \eqref{eq:Taylor_estimate}, we obtain
\begin{multline}
\label{eq:Taylor_sequence}
\sqrt{s_j}\|\slashed D_1 \xi_{j1}\|_{L^2(N)}  + \sqrt{s_j}\| \slashed D_0 \xi_{j2}\|_{L^2(N)} + s_j\| \xi_{j2}\|_{L^2(N)} 
\\
+ \|\bar \nabla^{\mathcal H}\xi_{j1}^0\|_{L^2(N)} + \|\bar \nabla^{\mathcal H}\xi_{j1}^1\|_{L^2(N)} + \|\bar\nabla^{\mathcal H}\xi_{j2}\|_{L^2(N)}
\\
\leq C(\|\tilde D_{s_j} \eta_j\|_{L^2(N)} + \|\eta_j\|_{L^2(N)}) \leq C, 
\end{multline}
for every $j$. We now deal with each individual component of  \eqref{eq:Taylor_sequence}:

\bigskip

\begin{case}[The sequence $\{\xi_{j1}^0\}_j$]\hfill

Changing variables on integrals of the left hand side of estimates \eqref{eq:elliptic_estimate1_k>=0} with $k=0$ and using \eqref{eq:Taylor_sequence}, we obtain a uniform bound,
\begin{equation*}
 \|\bar\nabla \xi_{j1}^0 \|_{L^2(N)} \leq  \|\bar\nabla^{\mathcal V} \xi_{j1}^0 \|_{L^2(N)} +  \|\bar\nabla^{\mathcal H} \xi_{j1}^0 \|_{L^2(N)} \leq  C\|\xi^0_{j 1}\|_{L^2(N)} +  \|\bar\nabla^{\mathcal H} \xi_{j1}^0 \|_{L^2(N)} \leq C, 
\end{equation*}
for every $j$. Hence $\sup_j\|\xi_{j1}^0\|_{W^{1,2}(N)} < \infty$. By the weak compactness of the unit ball in $W^{1,2}(N)$, there exist a subsequence denoted again as $\{\xi_{j1}^0\}_j$ converging weakly in $W^{1,2}(N)$ to a section $\xi\in W^{1,2}(N)$. By Rellich Theorem, for every $T>0$, there exist a subsequence, denoted again as $\{\xi_{j1}^0\}_j$, converging in $L^2(Z,T)$ to, $\xi\vert_{B(Z,T)}$ . Notice that we can choose a subsequence $\{\xi_{1j}^0\}_j$ so that $\xi_{1j}^0 \to \xi$ in $L^2_{loc}(N)$ topology. Indeed this follows by obtaining the subsequence $\{\xi_{j1}^0\vert_{B(Z,T)}\}_j$ extracted by Rellich Theorem applied to $B(Z, T+1)$ and extracting a subsequence by applying again Rellich's Theorem on the larger neighborhood $B(Z, T+1),\ T\in \mathbb{N}$ and then applying a diagonal argument on $T$. It follows that $\{\xi_{j1}^0\}_j$ converges on $L^2_{loc}(N)$ and weakly on $W^{1,2}(N)$ to $\xi$. We denote the $L^2$ norm of $L^2(B(Z,T))$ by $\|\cdot\|_{2,T}$. By weak lower semicontinuity $\|\slashed D_1\xi\|_{2,T}=0$ for every $T>0$ so that $\slashed D_1\xi=0$.
\end{case}

\begin{case}[The sequence $\{\xi_{j1}^1\}_j$]\hfill

We have estimates,
\begin{align*}
\|\bar\nabla^{\mathcal V} \xi_{j1}^1 \|_{L^2(N)} &\leq C\|\slashed D_1 \xi_{j1}\|_{L^2(N)} \leq C s_j^{-1/2}, \qquad (\text{by \eqref{eq:Taylor_sequence}})  
\\ 
\|\bar\nabla^{\mathcal H} \xi_{j1}^1\|_{L^2(N)} &\leq C,\qquad (\text{by \eqref{eq:Taylor_sequence}})
\\
\sup_j\sqrt{s_j}\|\xi_{j1}^1\|_{L^2(N)} &\leq C  \sup_j \sqrt{s_j}\|\slashed D_1 \xi_{j1}^1\|_{L^2(N)}< C, \qquad (\text{by \eqref{eq:spectral_estimate}.}) 
\end{align*}

Hence $\{\xi_{j1}^1\}_j$ converges strongly on $L^2(N)$ to zero and weakly on $W^{1,2}(N)$ to the zero section. Also, by weak compactness in $L^2(N)$, there exist weak limit $\psi_1\in L^2(N)$ with
\[
\sqrt{s_j} \xi_{j1}^1  \rightharpoonup \psi_1,  
\]
on $L^2(N)$. By weak lower semicontinuity of the $L^2(N)$-norm, for every $T'<T$, every $0<h<\tfrac{1}{2}(T-T')$ and every $\alpha$, the difference quotients in the fiber directions at $(p,v)\in N$, 
\[
\partial_\alpha^h \psi_1(p,v) := \frac{1}{h}[\psi(p,v + h e_\alpha) - \psi(p,v)]\in E_p ,
\]
satisfy
\begin{align*}
\|\partial_\alpha^h \psi_1\|_{2,T'} &\leq \liminf_j \sqrt{s_j}\|\partial^h_\alpha\xi_{j1}^1\|_{2, T'} 
\\
&\leq \limsup_j C_T \sqrt{s_j}\|\partial_\alpha\xi_{j1}^1\|_{2, T}
\\
&\leq \limsup_j C_T \sqrt{s_j}\|\slashed D_1 \xi_{j1}^1\|_{L^2(N)}< C_T,   
\end{align*}
where in the last couple of lines of the preceding estimate, we used  \eqref{eq:bootstrap_estimate} with $k=0$ and \eqref{eq:Taylor_sequence}. Hence $\psi_1$ has uniform $L^2(N)$-bounds on the difference quotients of the normal directions and therefore has weak derivatives in the normal directions that are bounded in $L^2(N)$. Since $\sup_j \|\sqrt{s_j} \slashed D_1 \xi_{j1}\|_{L^2(N)}<\infty$, it follows by Lemma ~\ref{lemma:weak_convergence} applied to the sequence $\{\sqrt{s_j} \xi_{j1}- \psi_1\}_j$, that $\sqrt{s_j} \slashed D_1 \xi_{j1}\rightharpoonup \slashed D_1\psi_1$ in $L^2(N)$-weakly. 
\end{case}

\begin{case}[The sequence $\{\xi_{j2}\}_j$]\hfill   

By \eqref{eq:Taylor_sequence}, we have estimates
\begin{align*}
\|\bar\nabla^{\mathcal V} \xi_{j2} \|_{L^2(N)} &=  \| \slashed D_0 \xi_{j2}\|_{L^2(N)} \leq C s_j^{-1/2}
\\
\|\bar\nabla^{\mathcal H} \xi_{j2} \|_{L^2(N)} &\leq C
\\
 s_j\| \xi_{j2}\|_{L^2(N)} &\leq C.
\end{align*}
Hence $\{\xi_{j2}\}_j$ converge strongly on $L^2(N)$ and weakly on $W^{1,2}(N)$ to the zero section. 
Similarly the sequence $\{\sqrt{s_j} \xi_{j2}\}_j$ converge strongly on $L^2(N)$ to zero and the sequence $\{\sqrt{s_j} \slashed D_0\xi_{j2}\}_j$ is bounded on $L^2(N)$. Using Lemma ~\ref{lemma:weak_convergence} the later sequence converge weakly in $L^2$ to zero. Finally by weak compactness in $L^2(N)$, there  exist $\psi_2\in L^2(N)$ so that
\[
 s_j \xi_{j2}  \rightharpoonup \psi_2.
\]
\end{case}

To summarize, we have the $L^2(N)$-weak limits,
\begin{align*}
 (\sqrt{s_j} \slashed D_1 + \bar D^Z)\xi_{j1} &\rightharpoonup \slashed D_1 \psi_1 + \bar D^Z\xi,
\\
\sqrt{s_j} \slashed D_0 \xi_{j2},\ \bar D^Z\xi_{j2} & \rightharpoonup 0,
\\
s_j \xi_{j2}&\rightharpoonup \psi_2.
\end{align*}

Using the expansions \eqref{eq:taylorexp} and  \eqref{eq:taylorexp1},
\begin{align*}
\tau^{-1}\tilde D_{s_j} \eta_j =& \left(\sqrt{s_j}\slashed D_1 + \bar D^Z + {\mathcal B}^0 + \frac{1}{2}r^2 \bar A_{rr}\right)\xi_{j1}  +(\sqrt{s_j}\slashed D_0 + \bar D^Z +   s_j\bar{\mathcal A})\xi_{j2} 
\\
&+s_j^{-1/2} O(r^2  \partial^{\mathcal N} +  r\partial^{\mathcal H})\xi_j + s^{-1/2}_jO(r +  r^3)\xi_{j1} +\tau O(1 + s_j^{1/2} r)\xi_{j2} .
\end{align*}
The $L^2(\mathcal{N}^T)$-norm of the error term estimates no more than $s_j^{-1/2} C_{T, \varepsilon} \| \xi_j\|_{1,2, T}$ for large $j$. By our assumption that $\tilde D_{s_j} \eta_j \to 0 $ in $L^2(N)$, we obtain that 
\[
\left\|\left(\sqrt{s_j}\slashed D_1 + \bar D^Z + {\mathcal B}^0 + \frac{1}{2}r^2 \bar A_{rr}\right)\xi_{j1} + (\sqrt{s_j}\slashed D_0 + \bar D^Z +   s_j\bar{\mathcal A})\xi_{j2}\right\|_{2,T}  \to 0 
\]
as $j\to \infty$ for every $T>0$. By weak lower semicontinuity, we conclude that $\xi, \psi_1, \psi_2$ satisfy the system of mutually $L^2(N)$-orthogonal components,
\begin{align*}
\slashed D_1 \xi&=0,
\\
\slashed D_1 \psi_1 + \left(\bar D^Z + P^0\circ \left({\mathcal B}^0 + \frac{1}{2}r^2 \bar A_{rr}\right) \right) \xi&=0, 
\\
 \bar{\mathcal A} \psi_2 +  (1- P^0)\circ \left( {\mathcal B}^0 + \frac{1}{2}r^2 \bar A_{rr}\right)\xi &=0.
\end{align*}
By Proposition~\ref{prop:horizontal_regularity}, we obtain that $\psi_1 \in W^{1,2}_{k,l}(N)$, for every $k,l>0$. We further brake the second equation into components. We start by using the decompositions $S^0 = \bigoplus_\ell S^0_\ell$ so that
\[
\xi = \sum_\ell \varphi_{1\ell}\cdot \xi_\ell\qquad \text{and} \qquad \bar\xi:= \sum_\ell \xi_\ell, 
\]
where $\varphi_{1\ell} = \lambda_\ell^{\tfrac{n-m}{4}} \exp\left(-\tfrac{1}{2} \lambda_\ell r^2\right)$  and $\xi_\ell \in W^{1,2}(Z; S^{0+}_\ell)$.
By \eqref{eq:br_D_Z_versus_D_Z_+},
\[
\bar D^Z \xi  =  \sum_\ell (\mathfrak{c}_N(\pi^* d \ln\varphi_\ell) \xi_\ell +  \varphi_\ell\cdot D^Z_+\xi_\ell).
\]
In the proof of the claim in Lemma~\ref{lem:balancing_condition} we calculated the components of the second equation of the system belonging to $\ker\slashed D_1$. These give the equation $D^Z_+\bar\xi  + {\mathcal B}^Z_{0+}\bar \xi =0$. The remaining terms of the aforementioned  equation, are $L^2$-perpendicular to  $\ker\slashed D_1$. They give $\slashed D_1 \psi_1 + {\mathcal P}_1 \xi =0$. This completes the proof of the existence of $\xi$ with the asserted properties. 
\end{proof}

We used the following lemma:

\begin{lemma}
\label{lemma:weak_convergence}
Let $\{\xi_j\}_j $ a sequence converging weakly in $L^2(N)$ to zero and possessing directional weak derivatives in $L^2$ in the necessary directions to guarantee the existence of the sequence $\{L\xi_j\}_j$ for a given differential operator $L$ of 1st order. Assume $\sup_j\|L\xi_j\|_{L^2(N)} < \infty$. Then the sequence $\{L\xi_j\}_j$ converges weakly to zero in $L^2$-norm. 
\end{lemma}

\begin{proof}
Let $\sup_j \|L\xi_j\|_{L^2(N)} = M$ and choose $\psi \in L^2(N)$ and $\varepsilon>0$. Choose smooth compactly supported section $\chi\in W^{1,2}(N)$ with $\| \chi - \psi\|_{L^2(N)} < \varepsilon/M$. Then 
\begin{equation*}
\lvert\langle L\xi_j, \psi\rangle_{L^2(N)}\rvert \leq \lvert\langle L\xi_j, \chi\rangle_{L^2(N)}\rvert + \|L\xi_j\|_{L^2(N)}\| \psi - \chi\|_{L^2(N)} 
\leq \lvert\langle \xi_j, L^*\chi\rangle_{L^2(N)}\rvert +  \varepsilon  
\end{equation*}
so that $\limsup_j \lvert\langle L\xi_j, \psi\rangle_{L^2(N)}\rvert< \varepsilon$. Since this is an arbitrarily chosen $\varepsilon$, we have that $\lim_j \langle L\xi_j, \psi\rangle_{L^2(N)} =0$. This concludes the proof.
\end{proof}

\begin{proof}[Proof of estimate (\ref{eq:est2'}) in Lemma \ref{lemma:aux_estimate}]
This is a Poincar\'{e} type inequality and we prove it by contradiction. Fix $k\in \{0,1,2\}$. Negating the statement of  Lemma \ref{lemma:aux_estimate} for $\varepsilon_0 = 1/j$ and $C=j$, there exist $0< \varepsilon_j < 1/j$  with the following significance: there is an unbounded sequence of $\{s_u(\varepsilon_j)\}_u$ and $\eta_u(\varepsilon_j) \in\tilde V^\perp_{\varepsilon_j, s_u}\cap  C^\infty_c(\mathcal{N}_{2\varepsilon_j} ; \tilde E^0)$ so that
\[
\int_N\lvert\eta_u\rvert^2\, d\mathrm{vol}^\mathcal{N}=1 \quad \text{and}\quad  j\| \tilde D_{s_u} \eta_u\|_{L^2(N)} \leq s_u^{\tfrac{k}{2}}\| r^k \eta_u\|_{L^2(N)}, \quad \text{for every $u\in \mathbb{N}$.} 
\]
In particular, we set $s_j$ to be the first term of the unbounded sequence $\{s_u\}_u$ that is bigger than $\frac{j^2}{\varepsilon_j^2}$. We also set $\eta_j$ to be the associated compactly supported section to $s_j$, and denote $\mathcal{N}_j: = \mathcal{N}_{2\varepsilon_j}$ and $\tilde  V_j^\perp:= \tilde V^{\perp_{L^2}}_{\varepsilon_j, s_j}\cap C^\infty_c(\mathcal{N}_j ; \tilde E^0\vert_{\mathcal{N}_j})$.

By using induction in $j\in \mathbb{N}$, we obtain sequences  $\{\varepsilon_j\}_j \subset (0, 1/j),\ \{s_j\}_j \subset (j, \infty)$ and $\{\eta_j\}_j \subset \tilde V^\perp_j$  so that, 
\[
\int_N \lvert\eta_j\rvert^2\, d\mathrm{vol}^\mathcal{N} =1 \quad \text{and}\quad j\|\tilde D_{s_j}\eta_j\|_{L^2(N)}\leq s_j^{\tfrac{k}{2}}\|r^k\eta_j\|_{L^2(N)}, \quad \text{for every $j\in \mathbb{N}$.}
\]
When $k=0$ this implies that $\|\tilde D_{s_j}\eta_j\|_{L^2(N)} \to 0$. When $k=1$ or $2$, then set $ \tau\bar \eta_j = \eta_j$ and estimate
\begin{equation*}
s_j^{\tfrac{k}{2}}\|r^k\bar\eta_j\|_{L^2(N)} \leq C \left( (s_j^{-1/2} + \varepsilon_j) \|\slashed D_{s_j} \bar \eta_j\|_{L^2(N)} +  \|\eta_j\|_{L^2(N)}\right) \leq C( \|\tilde D_{s_j} \eta_j\|_{L^2(N)} + 1), 
\end{equation*}
where in the first inequality we used \eqref{eq:elliptic_estimate1_k>=0}, with $k=1$ or $k=2$ and in the second one we used \eqref{eq:Taylor_estimate} and the fact that $\|\eta_j\|_{L^2(N)} \leq 2$. It follows that
\[
j\|\tilde D_{s_j}\eta_j\|_{L^2(N)}\leq C( \|\tilde D_{s_j} \eta_j\|_{L^2(N)} + 1),
\]
for every $j$ in which case we obtain again that $\|\tilde D_{s_j}\eta_j\|_{L^2(N)} \to 0$.

We recall the re-scaled sequence $\{\xi_j\}_j\subset W^{1,2}(N; \pi^*(E^0\vert_Z))$ of $\{\eta_j\}_j$ introduced in \eqref{eq:modified_section}. By Lemma~\ref{lemma:perps_of_app_solutions} there exist a subsequence denoted again as $\{\xi_j\}_j$ that converges $L^2_{loc}$-strongly and $W^{1,2}$-weakly to a section $\sum_\ell \varphi_\ell \xi_\ell$ where $\xi_\ell \in W^{1,2}(Z; S^{0+}_\ell)$ satisfies $(D^Z_+ + {\mathcal B}^Z_{0+})\xi_\ell = 0$ for every $\ell$ and $\varphi_\ell := \lambda_\ell^{\tfrac{n-m}{4}} \exp\left( -\frac{1}{2}\lambda_\ell r^2 \right)$.

\medskip

\textit{Claim:} $\xi_\ell \equiv 0$ for every $\ell$. 

\proof[Proof of claim]
By assumption $\eta_j \perp_{L^2} \tilde V_{s_j, \varepsilon_j}$. For every $j$ we construct $\xi_{s_j} = \xi_j^0 + \xi_j^1 + \xi_j^2$ where  $\xi_j^0 = \sum_\ell \varphi_{s_j \ell} \cdot \xi_\ell$ and $\xi_j^1,\ \xi_j^2$ are constructed by equations \eqref{eq:balansing_condition_1} and \eqref{eq:balansing_condition_2} respectively. Using Definition~\ref{def:space_of_approx_solutions}, we have  $\rho_{\varepsilon_j} \cdot \tau\xi_{s_j}\perp_{L^2} \eta_j$, for every $j$. 
We denote by $d\mathrm{vol}_j$, the density with density function the pullback of $d\mathrm{vol}^\mathcal{N}/ d \mathrm{vol}^N$ under the rescalling $\{y_\alpha = \sqrt{s_j} x_\alpha\}_\alpha$.  The orthogonality condition writes
\begin{equation}
\label{eq:orthogonality_in_L_2}
0=\int_{\mathcal{N}_j} \langle \eta_j, \rho_{\varepsilon_j} \cdot \tau\xi_{s_j}\rangle\, d\mathrm{vol}^\mathcal{N} = \sum_\ell\int_N \langle \xi_j, \rho_{\varepsilon_j\sqrt{s_j}} \cdot \varphi_\ell\cdot\xi_\ell\rangle\, d\mathrm{vol}_j + \int_{\mathcal{N}_j} \langle \eta_j, \rho_{\varepsilon_j} \cdot \tau (\xi^1_j + \xi^2_j)\rangle\, d\mathrm{vol}^\mathcal{N}.
\end{equation}
The second integral of the right hand side obeys the bound,
\begin{equation*}
\left\lvert\int_{\mathcal{N}_j} \langle \eta_j, \rho_{\varepsilon_j} \cdot (\xi^1_j + \xi^2_j)\rangle\, d\mathrm{vol}^\mathcal{N}\right\rvert \leq C\|\eta_j\|_{L^2(N)} (\|\xi^1_j \|_{L^2(N)} +  \|\xi^2_j \|_{L^2(N)}) \leq Cs_j^{-1/2}\sum_\ell\|\xi_\ell\|_{L^2(N)},
\end{equation*}
where in the last line we used estimates of Lemma~\ref{lem:balancing_condition}). It follows that the second integral of the right hand side vanishes as $j\to \infty$. Here we also used that the density $d\mathrm{vol}^\mathcal{N}$ is equivalent to $d\mathrm{vol}^N$. On the other hand, by construction $\varepsilon_j \sqrt{s_j} >j$ for every $j$ and therefore $\rho_{\varepsilon_j \sqrt{s_j} } \to 1$ uniformly on compact subsets on $N$.  Also by expansion \eqref{eq:volume-expansion} we have $\lim_j \lvert d\mathrm{vol}_j - d\mathrm{vol}^N\rvert =0$. Hence for every $T>0$,
\begin{align*}
\sum_\ell\int_{B(Z,T)} \varphi_\ell^2 \cdot \lvert\xi_\ell\rvert^2\, d\mathrm{vol}^N  & =  \lim_j \sum_\ell\int_{B(Z,T)} \langle \rho_{\varepsilon_j\sqrt{s_j}} \cdot\xi_j,  \varphi_\ell\cdot\xi_\ell\rangle\, d\mathrm{vol}_j
\\
&= - \lim_j\sum_\ell\int_{B(Z,T)^c} \langle \xi_j, \rho_{\varepsilon_j\sqrt{s_j}} \cdot \varphi_\ell\cdot\xi_\ell\rangle\, d\mathrm{vol}_j 
\\ 
&\leq \lim\sup_j \sum_\ell \int_{B(Z,T)^c}\lvert\langle\xi_j, \varphi_\ell \xi_\ell\rangle\rvert\, d\mathrm{vol}_j
\\
&\leq \sum_\ell\left(\int_{B(Z,T)^c}\varphi_\ell^2\lvert\xi_\ell\rvert^2\, d\mathrm{vol}^N\right)^{1/2},
\end{align*}
where in the second line we used \eqref{eq:orthogonality_in_L_2} the inequality in the last line follows by Cauchy-Schwarz and the fact that $\lim_j \int_N \lvert\xi_j\rvert^2 \, d\mathrm{vol}_j =1$. Letting $T\rightarrow \infty$ and using the fact that $\phi_\ell \cdot\xi_\ell\in L^2(N)$, we see that $\xi_\ell \equiv 0$, finishing the proof of the claim. 

Finally we fix first $T>0$ and $j$ large enough and we decompose the region $\mathcal{N}_j$ where $\eta_j$ is supported as
\[
\mathcal{N}_j = B(Z,T/\sqrt{s_j}) \cup (B(Z,T/\sqrt{s_j})^c,
\]
It follows that for every $T>0$, as $j\rightarrow \infty$,
\[
\lim_j\int_{B(Z,T/\sqrt{s_j})}\lvert\eta_j\rvert^2\, d\mathrm{vol}^\mathcal{N}\, =\, \lim_j \int_{B(Z,T)} \lvert\xi_j\rvert^2\, d\mathrm{vol}_j\, =\, \int_{B(Z, T)}\varphi_\ell^2\cdot\lvert\xi\rvert^2\, d\mathrm{vol}^N\, =\, 0,
\]
And therefore
\[
\lim_j  \int_{B(Z, T/\sqrt{s_j})^c}\lvert\eta_j\rvert^2 d\mathrm{vol}^\mathcal{N} = 1 -\lim_j \int_{B(Z,T/\sqrt{s_j})}\lvert\eta_j\rvert^2\, d\mathrm{vol}^\mathcal{N} = 1.
\]
We now obtain a contradiction from the concentration estimate. Since $\eta_j$ is compactly supported in $\mathcal{N}_j$, by Lemma \ref{lemma:normal_rates}, it satisfies the pointwise estimate,
\[
\lvert\tilde {\mathcal A} \eta_j\rvert^2 \geq Cr^2 \lvert\eta_j\rvert^2,
\]
But then, by a concentration estimate
\begin{equation*}
\int_{\mathcal{N}_j}\lvert\tilde D_{s_j}\eta_j\rvert^2 \, d\mathrm{vol}^\mathcal{N} \geq s_j^2\int_{B(Z, T/\sqrt{s_j})^c}\lvert \tilde{\mathcal A}(\eta_j)\rvert^2\, d\mathrm{vol}^\mathcal{N} - C_1 s_j \geq s_j\left(C T^2 \int_{B(Z,T/(\sqrt{s_j}))^c}\lvert\eta_j\rvert^2\, d\mathrm{vol}^\mathcal{N} -  C_1\right).
\end{equation*}
But then, for $T> 2\sqrt{C_1/C} $, the preceding estimate contradicts the estimate,
\[
\lim _j \int_\mathcal{N} \lvert\tilde D_{s_j}\eta_j\rvert^2\, d \mathrm{vol}^\mathcal{N} \leq 2\lim _j \int_\mathcal{N} \lvert\tilde D_{s_j}\eta_j\rvert^2\, d \mathrm{vol}^N =0, 
\]
where we used the volume inequality \eqref{eq:density_comparison}. This contradiction proves estimate \eqref{eq:est2'}.
\end{proof}

Finally we have the 
\begin{proof}[Proof of estimate \eqref{eq:est2} of Theorem \ref{Th:hvalue}]
Estimate  \eqref{eq:est2} follows for compactly supported sections $\eta \in V_{s, \varepsilon}^\perp \cap W^{1,2}_0(B_Z(4\varepsilon) ; E^0)$, from estimate \eqref{eq:est2'} by setting $k=0$. By Lemma~\ref{lemma:local_implies_global}, the same estimate is true for general sections $\eta \in V_{s, \varepsilon}^\perp$. This completes the proof. 
\end{proof}

\medskip
   
\vspace{1cm}

\section{Morse-Bott example} 
\label{sec:Morse_Bott_example}

 On a closed Riemmanian manifold $(X^n,g)$ the bundle $E\oplus F = {\Lambda}^\mathrm{ev} T^*X \oplus {\Lambda}^{odd}T^*X$ is a Clifford algebra bundle in two ways:  
\begin{align}
\label{eq:hatcl}
c(v) = v\wedge - \iota_{v^\#}\qquad \mbox{and}\qquad \hat{c}(w) = w\wedge + \iota_{w^\#}
\end{align}
for $v,w\in T^*X$. One checks that these anti-commute: 
\begin{align}
\label{eq:twocliff}
c(v)\hat{c}(w)\ =\   - \hat{c}(w) c(v).
\end{align} 
Note that $D=d+d^*$ is a  first-order operator whose symbol is $c$. Fix a Morse-Bott function $f$ with critical $m_\ell$-dimensional submanifolds $Z_\ell$ and normal bundles $N_\ell$ so that the Hessian $\text{Hess}(f)_\ell:N_\ell \to N_\ell$ is symmetric nondegenerate of Morse-index $q_\ell$.  Then Theorem~\ref{Th:mainT} shows that the low eigenvectors of
\[
D_s = D + s{\mathcal A}_f = (d + d^*) + s\hat{c}(df) : \Omega^{ev}(X)\rightarrow \Omega^{odd}(X)
\]
concentrate around the critical submanifolds $Z_\ell$; fix a critical set $Z=Z_\ell$ of dimension $m$ with normal bundle $N$ and Morse index $q$. The splitting $T^*X\vert_Z= T^*Z\oplus N^*$ gives decompositions
\begin{align*}
\Lambda^{ev}X\vert_Z &= \Lambda^{ev}Z \otimes \Lambda^{ev}N \oplus  \Lambda^{odd}Z\otimes \Lambda^{odd}N
\\
\Lambda^{odd}X\vert_Z &= \Lambda^{ev}Z \otimes \Lambda^{odd}N \oplus  \Lambda^{odd}Z \otimes \Lambda^{ev}N
\end{align*}
The normal bundle $N$ is orientable and further decompose to the orientable positive $N^+\to Z$ and negative $N^-\to Z$ eigenbundles of the Hessian, of dimension $q$. In Morse-Bott coordinates near $p\in Z$ the function takes the form 
\[
f(x_i ,x_\alpha)= \sum_{\alpha= 1}^{n-m} \eta_\alpha x_\alpha^2 ,\quad \text{where} \quad \eta_\alpha = \begin{cases} -1&\quad \text{when $\alpha \leq q$}
\\
1&\quad \text{when $\alpha \geq q+1$}
\end{cases},
\]
and the nondegenerate hessian  has the form $\text{Hess}(f)\vert_Z = \text{diag}(\eta_\alpha)$. Then  
\[
M_\alpha^0 = - \eta_\alpha c(dx^\alpha) \hat{c}(dx^\alpha): {\Lambda}^\mathrm{ev} T_p^*X\rightarrow {\Lambda}^\mathrm{ev} T_p^*X, \ \text{for every $\alpha$}
\]
are invertible self-adjoint matrices with symmetric spectrum of eigenvalues $\pm 1$ that commute with each other. 

\begin{lemma} 
\begin{itemize}
\item If $Z$ has index $q$  then the real line bundle $ \Lambda^qN^-  \to Z$ is trivial. 

\item If $\phi$ belongs in the +1- eigenspace of $M_\alpha^0$ then
\[
S^{0+}_\alpha =\begin{cases} \{\xi\in \Lambda^{ev}X: \xi\wedge dx^\alpha=0\},&\quad \text{when $\alpha\leq q$}
\\ 
\{\xi\in \Lambda^{ev}X: \iota_{\partial_\alpha}\xi =0\} ,&\quad \text{when $\alpha\geq q+1$} \end{cases}.
\]
An analogue description holds for $M_\alpha^1$ but with $\Lambda^{odd}X$ in place of $\Lambda^{ev}X$.
\item If $q$ is even, then $S^{0+} \cong \Lambda^{ev}Z \otimes \Lambda^qN^-$ and $S^{1+} =\Lambda^{odd}Z \otimes \Lambda^qN^- $, and 
\item If $q$ is odd, then  $S^{0+} = \Lambda^{odd}Z \otimes \Lambda^qN^-$ and $S^{1+}\cong  \Lambda^{ev}Z \otimes \Lambda^qN^-$.

\item The Clifford map $c:T^*Z \otimes \Lambda^*X \to \Lambda^*X $ restricts to the map $\bar c:T^*Z\otimes \Lambda^{ev}Z \to \Lambda^{odd}Z$ and $c_Z:T^*Z \otimes \Lambda^{ev}Z \otimes \Lambda^q N^- \to \Lambda^{odd}Z \otimes \Lambda^q N^-$ is given by $c_Z = \bar c\otimes 1_{\Lambda^q N^-}$ when $q$ is even and by $c_Z = \bar c^*\otimes 1_{\Lambda^q N^-}$, when $q$ is odd.

\end{itemize}
\end{lemma}
\begin{proof}
The first bullet follows because $N^-\to Z$ is an orientable bundle.

For the second bullet we use an orthonormal coframe $\{e^j\}_j$ at $p\in Z$ and decompose $\phi = \sum_I \lambda_I e^I$. Then, using \eqref{eq:hatcl}, $\phi\in S^{0+}_\alpha$ if and only if 
\begin{align*}
\phi &= M_\alpha^0 \phi = - \eta_\alpha c(dx^\alpha) \hat{c}(dx^\alpha)\phi 
= -\eta_\alpha( dx^\alpha\wedge(\iota_{\partial_\alpha}\phi)\, -\, \iota_{\partial_\alpha}(dx^\alpha\wedge \phi))
\\ &= \eta_\alpha( \phi\, - \,2dx^\alpha\wedge (\iota_{\partial_\alpha} \phi)) 
\\
&= \eta_\alpha\left(\sum_{\{I: \lvert I\rvert\, \text{even},\ \alpha\notin I\}} \lambda_I e^I -   \sum_{\{I: \lvert I\rvert\, \text{even},\ \alpha \in I\}} \lambda_I e^I \right)
\\
&= \begin{cases} \sum_{\{I: \lvert I\rvert\, \text{even},\ \alpha\notin I\}} \lambda_I e^I -   \sum_{\{I: \lvert I\rvert\, \text{even},\ \alpha \in I\}} \lambda_I e^I , \quad \text{if $\alpha>q$}
\\
-\sum_{\{I: \lvert I\rvert\, \text{even},\ \alpha\notin I\}} \lambda_I e^I +   \sum_{\{I: \lvert I\rvert\, \text{even},\ \alpha \in I\}} \lambda_I e^I , \quad \text{if $\alpha \leq q$},
\end{cases}
\end{align*}
where in the fourth equality we used the Cartan's identity. It follows that when $\alpha>q$ then $\lambda_I=0$ if and only if $\alpha\in I$ and when $\alpha\leq q$ then $\lambda_I=0 $ if and only if $\alpha\notin I$. Therefore we obtain the descriptions in the second bullet.  

Continuing the preceding argument, $\phi \in \bigcap_\alpha S^{0+}_\alpha$ then $\lambda_I=0$ if and only if $\{1, \dots, q\} \subset I$ and $\{q+1, \dots, n-m\} \cap I = \emptyset$. Hence if $q$ is even, the third bullet holds. If $q$ is odd then the fourth bullet holds. The last bullet is an easy consequence of the preceding bullets.
\end{proof}

The Levi-Civita connection of $X$ restricts to $Z$ and together with the restriction from the Clifford action induce the operator $d_Z+ d_Z^*: \Lambda^{ev}Z \to \Lambda^{odd}Z$. The Localization Theorem with the Poincare-Hopf Theorem then give,
\begin{align*}
\chi(X) &= \mathrm{index\,}\left(d+d^*: C^\infty(X;\Lambda^{ev}X) \to C^\infty(X; \Lambda^{odd}X) \right) 
\\
&= \sum_\ell (-1)^{q_\ell} \mathrm{index\,}\left(d_\ell+d^*_\ell: C^\infty(Z_\ell;\Lambda^{ev}Z_\ell) \to C^\infty(Z_\ell; \Lambda^{odd}Z_\ell) \right) 
\\
&= \sum_\ell (-1)^{q_\ell} \chi(Z_\ell),
\end{align*}
a well know identity emerging from Morse-Bott homology. This is the localization in  E.~Witten's well-known paper on Morse Theory \cite{w1} but in the Morse-Bott case.

\appendix

\section*{Appendix}

\setcounter{equation}{0}
\renewcommand\theequation{A.\arabic{equation}}
\section{Fermi - coordinates setup near the singular set}
\label{App:Fermi_coordinates_setup_near_the_singular_set}

Fix $Z$ an $m$-dimensional submanifold of the $n$-dimensional Riemannian manifold $(X,\, g_X)$ with normal bundle $\pi:N \rightarrow Z$. Following the exposition in the reference \cite{g}[Ch.2], the distance function $r$ from the core set $Z$ as well as the exponential map of the normal bundle $N$ are well defined on a sufficiently small neighborhood of $Z$ in $X$. For small $\varepsilon>0$ they identify an open tubular neighborhood $B_Z(2\varepsilon) := \{p\in X : r(p)<2\varepsilon\}$ of $Z$ in $X$ with the neighborhood $\mathcal{N}_\varepsilon= \{(z,v) \in N: \lvert v\rvert_z<2\varepsilon\}$ of the zero section in $N$. In particular, we get a principal frame bundle isomorphism $I$ and its induced bundle isomorphism ${\mathcal I}$ introduced in \eqref{eq:exp_diffeomorphism} both covering the exponential diffeomorphism of the base. The metric tensor $g_X$, the volume form $d\mathrm{vol}^X$, the Levi-Civita connection $\nabla^{TX}$, the Clifford multiplication $c$ and the Clifford compatible connection $\nabla^E$, pull back and define a metric $g = \exp^* g_X$, a volume form $d\mathrm{vol}^\mathcal{N}$, connections  $\tilde\nabla^{T\mathcal{N}}$ and $\tilde \nabla^{\tilde E}$ and a Clifford action $\tilde c$ , respectively. Henceforth, the analysis and the expansions are carried out on $\mathcal{N}_\varepsilon\subset N$. Note that, when restricted to $Z$, the bundle maps $I$ and ${\mathcal I}$ are the identity bundle isomorphisms. 

The orthogonal decomposition $T\mathcal{N}\vert_Z = TZ \oplus N$ descents to a decomposition of the Riemannian metric $g\vert_Z = g_Z \oplus g_N$ and a decomposition of sections of $T\mathcal{N}\vert_Z$ into vector fields on $Z$ and sections of the bundle $N$. The restriction of the Levi-Civita connection $\tilde\nabla^{T\mathcal{N}}$ on sections of $T\mathcal{N}\vert_Z$ decomposes accordingly to three parts; 1) the Levi-Civita connection $\nabla^{TZ}$ of $(Z, g_Z)$, 2) a metric compatible connection $\nabla^N$ acting on sections of the normal bundle $(N, g_N)$ and 3) the 2nd fundamental form of the embedding $(Z, g_Z) \hookrightarrow (X, g_X)$.

Fix normal coordinates $(z_j)_j$ on a neighborhood $U\subset Z$ centered at $p$ and choose orthonormal frames $\{e_\alpha\}_\alpha$ that are $\nabla^N$-parallel at $p$, trivializing the restriction bundle $N\vert_U = \pi^{-1}(U)$. The frame $\{e_\alpha\}_\alpha$ at $z\in U$ identifies $N_z$ with $\mathbb{R}^{n-m}$ and $\mathcal{N}_z\subset N_z$ with an open subset of $B(0, 2\varepsilon)\subset\mathbb{R}^{n-m}$ with coordinates $(t_\alpha)_\alpha$. We define the so called Fermi coordinates as the bundle coordinates, $(x_j, x_\alpha)_{A=j, \alpha}$ on a neighborhood $\pi^{-1}(U)$ by requiring,
\begin{align*}
x_i\left(\sum_{\alpha=m+1}^n t_\alpha e_\alpha(z) \right) &= z_j ,\quad j= 1,\dots, m,\\
x_\alpha \left(\sum_{\alpha=m+1}^n t_\alpha e_\alpha(z)\right)&= t_\alpha, \quad \alpha = m+1,\dots ,n,
\end{align*}
when restricted on subsets 
\[
\mathcal{N}_U^\varepsilon:= \mathcal{N}_\varepsilon\cap \pi^{-1}(U).
\]
Frequently we will omit $\varepsilon$ from the notation and write $\mathcal{N}_U$ instead of $\mathcal{N}_U^\varepsilon$ and $\mathcal{N}$ instead of $\mathcal{N}_\varepsilon$. 

We denote by $\{\partial_j, \partial_\alpha\}$ and $\{dx^j, dx^\alpha\}$, the tangent and cotangent frames respectively so that $\partial_a\vert_U = e_\alpha$ and $\partial_j\vert_U = \partial_{z_j}$. More generally, a local vector field at $\mathcal{N}_U$ is called tangential Fermi field provided it has the form $ \phi = \sum_j d_j\partial_j$, for some constants $\{d_j\}_j$ and is called a normal Fermi field if it has the form $\psi = \sum_\alpha d_\alpha\partial_\alpha$, for some constants $\{d_\alpha\}_\alpha$. In these coordinates, the distance function is described as $r^2 = x_{m+1}^2 + \dots+ x_n^2$ and the radial outward vector field with its dual, are given by
\[
\partial_r = \sum_{\alpha = m+1}^n \frac{x_\alpha}{r} \partial_\alpha\qquad \text{and} \qquad dr = \sum_{\alpha = m+1}^n \frac{x_\alpha}{r} dx^\alpha.
\]

The notation $O(r^k\partial^H)$ and $O(r^k\partial^N)$ will denote expressions of tangential and normal Fermi fields correspondingly with coefficient components vanishing up to order $r^k$ when $r\to 0^+$. Using the transformation rules of the bundle coordinates,
\[
y_j= y_j(x_1, \dots, x_m), \quad y_\alpha = O_{\alpha\beta}(x_1,\dots, x_m) x_\beta
\]
for some orthogonal matrix $[O_{\alpha\beta}(x_1,\dots, x_m)]_{\alpha,\beta}$, the orders in $r$ of expressions of the form $O(r^k \partial^N)$ and $O(r^{k+1} \partial^N + r^k \partial^H)$ do not change under different coordinate frames.

Fermi coordinates generalize Gauss coordinates and the components $\{g_{AB}\}_{A,B= j,\alpha, r}$ of the Riemannian metric $g$ in $\mathcal{N}_U$, satisfy   
\begin{equation}
\label{eq:Fermi_Riemannian_metric}
 g_{rr}=1, \quad  g_{jr}=0, \quad g_{\alpha r} = \frac{x_\alpha}{r},\quad g_{j\alpha} = O(r), \quad g_{a\beta} = \delta_{a\beta}+ O(r^2),
\end{equation}
as $r\to 0^+$, for every $j,\alpha$. The Christofel symbols of $\tilde\nabla^{T\mathcal{N}}$ in Fermi coordinates are
\[
\tilde\nabla_j^{T\mathcal{N}} = \partial_j + \Gamma_{j A}^B,\qquad \tilde\nabla_\alpha^{T\mathcal{N}} = \partial_\alpha + \Gamma_{\alpha A}^B \qquad \text{and} \qquad \nabla_r = \partial_r + \Gamma_{r A}^B, 
\]
where $A, B = j, \beta$. We make use of the bar notation to denote restriction of the quantity to $U\subset \mathcal{N}_U \subset \mathcal{N}$ that is $ \bar\Gamma_{j \alpha}^k =\Gamma_{j \alpha}^k\vert_U$. The Christofel symbols obey relations
\begin{equation}
\label{eq:Christoffel_relations_0}
\Gamma_{AB}^\Delta = \Gamma_{BA}^\Delta,\quad \bar\Gamma_{\alpha \beta}\equiv 0,  \quad  \bar\Gamma_{j \alpha}^\beta = - \bar \Gamma_{j \beta}^\alpha, \quad \bar\Gamma_{ij}^\alpha = - \bar\Gamma_{i\alpha}^k \bar g_{kj},
\end{equation}
for every $A,B,\Delta$ and every $i,j,\alpha, \beta$ and the contraction $-\bar\Gamma_{ij}^\alpha \bar g^{ij} = \sum_k\bar\Gamma_{k\alpha}^k$ is the component $H_\alpha:= H_{e_\alpha}$ of the mean curvature in the direction of $e_\alpha$. 

The Dirac operator $D$ is expressed using orthonormal frames instead of Fermi frames. To that purpose, we introduce a $\nabla^{TZ}$-parallel at $p\in U$, orthonormal frame $\{e_j\}_j$ trivializing $TZ\vert_U$, so that $e_j\vert_p = \partial_j\vert_p$. The frames are compared with $e_j= d^j_k \partial_k\vert_U$, where  $d^i_k:U \to \mathbb{R}$ satisfies $d^i_k(p) = \delta^i_k$. We extend the frames $\{e_j, e_\alpha\}_{j,\alpha}$ by radial $\tilde \nabla^{T{\mathcal N}}$-parallel transport to frames $\{\tau_j, \tau_\alpha\}_{j,\alpha}$ over $\mathcal{N}_U$. The connection components $\tilde\nabla^{T\mathcal{N}}_{\tau_A} \tau_B = \omega_{AB}^\Delta \tau_\Delta$, satisfy 
\begin{equation}
\label{eq:connection_comp_rates}
\omega_{A B}^\Delta + \omega_{A \Delta}^B =0, \quad \bar\omega_{\alpha A}^B = 0 , \quad \text{and} \quad \omega_{j \alpha}^\beta(p) = \omega_{j k }^l(p)=0, 
\end{equation}
for every $A,B= j,\alpha$ and every $j,\alpha$.
\bigskip

\setcounter{equation}{0}
\renewcommand\theequation{B.\arabic{equation}}

\section{Taylor Expansions in Fermi coordinates and 1-jets}
\label{App:Taylor_Expansions_in_Fermi_coordinates}

In this section of the appendix, we calculate the Taylor expansions of some of the quantities involved in the proof of Theorem~\ref{Th:mainT}. The expansions are calculated up to the orders $O(r^3)$ for ${\mathcal A}$ and $O(r)$ for $\nabla{\mathcal A}$. The error terms are then, becoming negligible after the rescaling $\{w_\alpha=\sqrt{s} x_\alpha\}_\alpha$ as $s\to\infty$.

We work on a Fermi chart $(\mathcal{N}_U, (x_j, x_\alpha)_{j, \alpha})$ defined by normal coordinates $(U, (x_j)_j)$ on $Z$ and an orthonormal frame $\{e_\alpha\}_\alpha$ of $N\vert_U$. Let $\{\sigma_\ell\}_\ell,\ \{f_k\}_k$ orthonormal frames trivializing $E^0\vert_U$ and $E^1\vert_U$. Using Assumptions~\ref{Assumption:transversality1} and \ref{Assumption:transversality2}, we choose the frames so that the first $d$ vectors trivialize $S^i\vert_U,\ i=0,1$. We extend these frames radially by $\tilde\nabla^{\tilde E^i}$-parallel transport to obtain trivializations $\{\sigma_\ell\}_\ell,\ \{f_k\}_k$ over $\mathcal{N}_U$. 

The connection 1-form of $\tilde \nabla^{\tilde E^i}$ in these frames with respect to the frames $\{\tau_A\}_{A = j, \alpha}$ is given by
\begin{equation}
\label{eq:Clifford_connection_one_form_local_representations_in_on_frames}
\begin{aligned}
\tilde\nabla^{\tilde E^i}_{\tau_A} &= \tau_A + \theta^i_A, \quad A= j,\alpha,
\\
\bar\theta^i_a &= 0,  \quad  \partial_\alpha \bar{\theta}_\beta^i + \partial_\beta \bar{\theta}_\alpha^i =0
\end{aligned}
\end{equation}
for every $\alpha, \beta$ and every $i=0,1$. 

Finally we introduce the $\tilde\nabla^{\tilde E}$-parallel transport map along the radial geodesics of $\mathcal{N}_z$
\begin{align}
\label{eq:parallel_transport_map}
\begin{array}{cccl}
\tau : &C^\infty(\mathcal{N}_z; E_z^i) &\rightarrow &C^\infty(\mathcal{N}_z ; \tilde E^i\vert_{\mathcal{N}_z})
\\
&f(z,v)\sigma_k(z,0)&\mapsto &f(z,v)\sigma_k(z,v),
\end{array}  
\end{align}
for every $v\in N_z$ with $\lvert v\rvert_z< 2\varepsilon$ and every $i=0,1$. Hence $\tau$ operates from sections of $\pi^*E\vert_Z$ to give sections of $\tilde E\vert_\mathcal{N}$.

Subsections \ref{subApp:The_expansion_of_A_A*A_nabla_A_along_Z} and \ref{subApp:The_expansion_of_the_Spin_connection_along_Z} of the Appendix and Section~\ref{Sec:structure_of_A_near_the_singular_set} deal with the proof of Proposition~\ref{prop:properties_of_compatible_subspaces},  supporting the existence of decompositions \eqref{eq:eigenspaces_of_Q_i} and \eqref{eq:decompositions_of_S_ell} and the proof of Proposition~\ref{prop:basic_restriction_connection_properties} assering the properties of the adapted connection $\bar\nabla$ introduced in Definition~\ref{eq:connection_bar_nabla} and the term $B^i$ introduced in \eqref{eq:remainder_term}. Aside from these sections, the properties of the frames $\{\sigma_\ell\}_\ell,\ \{f_k\}_k$ trivializing the bundles $E^0\vert_U$ and $E^1\vert_U$ are updated so that: 1) they respect decompositions \eqref{eq:eigenspaces_of_Q_i} and \eqref{eq:decompositions_of_S_ell} and the decomposition of each $S^i_{\ell k}$ into simultaneous eigenspaces of $\{M_\alpha\}_\alpha$ as in \eqref{eq:decomposition_of_S_0_given_frame_e_a},  2) they are $\bar\nabla$-parallel at $p\in U$. In particular there exist a sub-frame that trivializes $S^{i+}_\ell\vert_U$. The components of the connections $\tilde\nabla^{\tilde E^i},\ \bar\nabla$ and the term $B^i$, are given in these frames by, 
\begin{align*}
\tilde\nabla^{\tilde E^0}_{e_A}\sigma_\ell = \bar\theta^0_A\sigma_\ell, &\qquad \tilde\nabla^{\tilde E^1}_{e_A}f_\ell = \bar\theta^1_Af_\ell,\quad A= j,\alpha,
\\
\bar\nabla_{e_j}^{E^0\vert_Z}\sigma_\ell  = \phi_j^0\sigma_\ell, &\qquad
\bar\nabla_{e_j}^{E^1\vert_Z}f_\ell  = \phi_j^1f_\ell,
\\
B^0_{e_j}\sigma_\ell = B^0_j\sigma_\ell, &\qquad
B^1_{e_j}f_\ell = B^1_jf_\ell, 
\end{align*}
so that, 
\begin{equation}
\label{eq:comparing_tilde_nabla_orthonormal_to_bar_nabla_orthonormal}
\bar\theta_j^i = \phi_j^i + B_j^i, \quad \phi^i_j(p)=0, \quad \phi_j^i(S^{i+}_\ell ) \subset S^{i+}_\ell,
\end{equation}
for every $i=0,1$.

\medskip

\subsection{The expansions of \texorpdfstring{${\mathcal A},\ {\mathcal A}^*{\mathcal A},\ \nabla{\mathcal A}$}{} along \texorpdfstring{$Z$}{}, as \texorpdfstring{$r\to 0^+$}{}.} \hfill
\label{subApp:The_expansion_of_A_A*A_nabla_A_along_Z}

We next  proceed to calculating terms up to second order in the Taylor expansion along the normal directions of $Z$ of the bundle maps $\tilde {\mathcal A},\  \tilde{\mathcal A}^*\tilde{\mathcal A}$ and their covariant derivatives.  We work in a chart $(\mathcal{N}_U, (x_j,x_\alpha)_{j,\alpha})$, choosing frames $\{\sigma_\ell\}_\ell$ of $\tilde E^0\vert_{\mathcal{N}_U}$ centered at $p\in U$ with coframes $\{\sigma^\ell\}_\ell$ and $\{f_k\}_k$ of $\tilde E^1\vert_{\mathcal{N}_U}$ so that they respect the decompositions $\tilde E^i= S^i \oplus (S^i)^\perp,\,  i =0,1$ introduced in Definition~\ref{defn:kernel_bundles}. Recall the orthogonal projections $P^i: E^i\vert_Z \to S^i\vert_Z,\ i=0,1$.  We denote the following notations for covariant derivatives, 
\[
A_\alpha:=  (\tilde\nabla_{\tau_\alpha}\tilde{\mathcal A})\vert_{\mathcal{N}_U}, \qquad A_{\beta \gamma} := (\tilde\nabla_{\tau_\beta} \tilde\nabla_{\tau_\gamma} \tilde{\mathcal A})\vert_{\mathcal{N}_U}.
\]
and
\begin{equation}
\label{eq:1st_jet_2nd_jet_of_A}
\bar A_ r := \frac{x_a}{r} \bar A_\alpha \qquad \bar A_{rr} := \frac{x_\alpha x_\beta}{r^2}\bar A_{\alpha\beta},
\end{equation}
where we use the bar notation when restricting components in $U$ that is $\bar{\mathcal A} := \tilde{\mathcal A}\vert_U,\ \partial_A \bar A_\ell^k := (\partial_A A_\ell^k)\vert_U,\ \bar A_\alpha= \tilde\nabla_{e_\alpha}\tilde{\mathcal A}$ etc. The expressions $A_B, \, A_{B\Gamma}$ depend on the choice of Fermi coordinates and frames while the expressions  $\bar A_r,\, \bar A_{rr}$ do not. 

\begin{proposition}
\label{prop:properties_of_perturbation_term_A}
\begin{enumerate}
\item 
\label{eq:transversality_equivalence}
If ${\mathcal A}$ satisfies transversality Assumption~\ref{Assumption:transversality1} then ${\mathcal A}^*$ does too.

\item
Assumption~\ref{Assumption:transversality1} imply that,
\begin{align}
\label{eq:transversality_assumption_with_respect_to_connection}
\nabla_u{\mathcal A}(S^0)\subseteq S^1 \qquad \mbox{and}\qquad \nabla_u{\mathcal A}^*(S^1)\subseteq S^0,
\end{align}
for every $u \in N$, where the later relation is obtained by the former one using \eqref{eq:dcond}. 

\item
\label{prop:Taylor_expansions_ of perturbation_term}
Recall the map $\tau$ from \eqref{eq:parallel_transport_map}. Then, for every $\eta\in C^\infty(\mathcal{N}_U; S^0)$ with $\eta = \tau \xi$,
\begin{align}
\tilde {\mathcal A} \eta &= \tau \left(r\bar A_r\xi + \frac{1}{2} r^2 \bar A_{rr}\xi + O(r^3) \xi \right)  \label{eq:jet0}
\\
\tilde {\mathcal A}^* \tilde {\mathcal A} \eta &=  r^2\tau\left( \bar A_r^* \bar A_r \xi+  \frac{1}{2} \bar{\mathcal A}^* \bar A_{rr} \xi + O(r) \xi \right)\label{eq:jet1},
\end{align}

\item  
\label{prop:expressions_for_map_C_i} 
Recall the matrices $C^i$ introduced in \eqref{eq:IntroDefCp}. We have expressions, 
\begin{equation}
C^i = \begin{cases} -\sum_\alpha c(e^\alpha) \left. \bar A_\alpha\right\vert_{S^0}, \quad \text{if $i=0$,}
\\
 -\sum_\alpha c(e^\alpha) \left.\bar A^*_\alpha\right\vert_{S^1}, \quad \text{if $i=1$}
\end{cases}, \qquad 
\end{equation}

\end{enumerate}
\end{proposition}

\begin{proof}
The bundle map $\tilde{\mathcal A}$ writes in local frames as,
\begin{equation}
\label{eq:perturbation_term_frame}
\tilde {\mathcal A}\vert_{\mathcal{N}_U} = A_\ell^k \sigma^\ell \otimes f_k\qquad \text{where} \qquad \bar A_\ell^k = 0 \qquad \text{when $\ell \leq d$ or $k \leq d$}, 
\end{equation}
and we obtain expressions, 
\begin{align}
A_\alpha =  (\tau_\alpha A_\ell^k + A_l^k \theta_{\alpha l}^{0\ell} + A_\ell^l\theta_{\alpha l}^{1k}) \sigma^\ell\otimes f_k, \label{eq:perturb_vertical_taylor},
\end{align}
and
\begin{align*}
\bar A_{\alpha \beta} &= \partial^2_{\alpha \beta} \bar A_\ell^k  \sigma^\ell \otimes f_k + \bar A_\ell^k( \partial_\alpha \bar{\theta}_{\beta \ell}^{0 t} \sigma^t \otimes f_k + \partial_\alpha \bar{\theta}^{1 t}_{\beta k}\sigma^\ell \otimes f_t),  
\end{align*}
where some of the terms of the preceding expression are vanishing because the connection 1-forms $\theta^i,\ i=0,1$ satisfy relations \eqref{eq:Clifford_connection_one_form_local_representations_in_on_frames}. Similarly,
\begin{equation}
\label{eq:nabla_perturb}
(\bar A_\alpha)_\ell^k = \partial_\alpha \bar A_\ell^k\qquad \text{and} \qquad  \sum_{\alpha, \beta} x_\alpha x_\beta (\bar A_{\alpha \beta})_\ell^k = \sum_{\alpha, \beta}x_\alpha x_\beta \partial_{\alpha \beta}^2 \bar A_\ell^k,
\end{equation}
for every $k,\ell$. 

To prove \eqref{eq:transversality_assumption_with_respect_to_connection} we recall the transversality Assumption~\eqref{Assumption:transversality1} in local coordinates, 
\begin{equation}
\label{eq:coordinate_transversality}
\partial_\alpha \bar A_\ell^k \equiv 0, \quad \text{when $\ell\geq d+1,\ k\leq d$ or $\ell \leq d,\ k\geq d+1$.}  
\end{equation}
It follows that \eqref{eq:transversality_equivalence} and inclusions \eqref{eq:transversality_assumption_with_respect_to_connection} is a direct consequence of \eqref{eq:coordinate_transversality} and \eqref{eq:nabla_perturb}.

By Taylor expanding the coefficients $A_\ell^k$ along the normal directions, we obtain 
\begin{align*}
\tilde {\mathcal A} =& \left( \bar A_\ell^k + x_\alpha (\partial_\alpha \bar A_\ell^k) + \frac{1}{2} x_\alpha x_\beta (\partial_{\alpha\beta}^2 \bar A_\ell^k) + O(\lvert x\rvert^3)\right) \sigma^\ell \otimes f_k,
\\
\tilde {\mathcal A}^* \tilde{\mathcal A} =&  \left(\bar A^k_l \bar A^k_\ell +  x_\alpha [ \bar A^k_l (\partial_\alpha \bar A_\ell^k) + (\partial_\alpha \bar A^k_l) \bar A_\ell^k]\right) \sigma^\ell\otimes \sigma_l 
\\
&+ x_\alpha x_\beta \left((\partial_\alpha \bar A^k_l)(\partial_\beta \bar A_\ell^k) + \frac{1}{2}\left((\partial_{\alpha\beta}^2 \bar A^k_l) \bar A_\ell^k + \bar A^k_l ( \partial_{\alpha\beta}^2 \bar A_\ell^k) \right) \right)\sigma^\ell\otimes \sigma_l + O(\lvert x\rvert^3) \sigma^\ell\otimes \sigma_l.
\end{align*}

Therefore, when we restrict to the subbundle $S^0$ spanned by the frame $\{\sigma_\ell\}_{\ell \leq d}$, we use \eqref{eq:perturbation_term_frame} and \eqref{eq:coordinate_transversality} to obtain,
\begin{equation}
\label{eq:perturb_exp}
\left.\tilde {\mathcal A} \right\vert_{S^0} =  \sum_{\ell, k \leq d}x_\alpha (\partial_\alpha \bar A_\ell^k) \sigma^\ell \otimes f_k + \sum_{\ell \leq d}\left(\frac{1}{2}x_\alpha x_\beta  \partial_{\alpha \beta}^2 \bar A_\ell^k + O(\lvert x\rvert^3) \right) \sigma^\ell \otimes f_k, 
\end{equation}
and
\begin{equation*}
\left.\tilde {\mathcal A}^* \tilde{\mathcal A} \right\vert_{S^0} =  x_\alpha x_\beta\sum_{k,l, \ell \leq d}(\partial_\alpha \bar A^k_l)(\partial_\beta \bar A_\ell^k)  \sigma^\ell\otimes \sigma_l  +  \frac{1}{2} x_\alpha x_\beta \sum_{\genfrac{}{}{0pt}{2}{ k,l \geq d+1}{\ell\leq d}} \bar A^k_l ( \partial_{\alpha\beta}^2 \bar A_\ell^k) \sigma^\ell\otimes \sigma_l  + O(\lvert x\rvert^3)\sigma^\ell\otimes \sigma_l. 
\end{equation*}
From these relations and equations \eqref{eq:nabla_perturb}, expansions \eqref{eq:jet0} and \eqref{eq:jet1} follow.

Finally for \eqref{prop:expressions_for_map_C_i}, we prove the case for $i=0$ and the other is proven similarly. We have that 
\[
\tilde\nabla \tilde{\mathcal A}\vert_{S^0} = e^j\otimes \bar A_j\vert_{S^0} + e^\alpha\otimes \bar A_\alpha\vert_{S^0}\quad\text{and}\quad c= e_j\otimes c(e^j)+ e_\alpha\otimes c(e^\alpha),  
\] 
and by using \ref{eq:transversality_assumption_with_respect_to_connection},
\[
\iota_\mathcal{N}^*\otimes P^1\circ\nabla{\mathcal A}\vert_{S^0} = e^\alpha\otimes A_\alpha\vert_{S^0}.
\]
It follows that
\[
C^0= - c(e^\alpha)\circ A_\alpha\vert_{S^0},
\]
finishing the proof.
\end{proof}

\begin{lemma}
\label{lemma:normal_rates}
Assume ${\mathcal A}$ satisfies Assumption~\ref{Assumption:normal_rates}. Then there exist $\varepsilon_0>0$ and $C=C(\varepsilon_0)>0$ so that whenever $\eta\in C^\infty\left(B_{Z_{\mathcal A}}(2\varepsilon_0); E^0\vert_{B_{Z_{\mathcal A}}(2\varepsilon_0)}\right)$, then 
\[
\lvert{\mathcal A} \eta\rvert^2 \geq C r^2\lvert\eta\rvert^2,
\]
where $r$ is the distance function from $Z_{\mathcal A}$
\end{lemma}

\begin{proof}
Without loss of generality assume that $Z_{\mathcal A}$ is a single component $Z$ and we can replace ${\mathcal A}$ by $\tilde{\mathcal A}$. We decompose $\eta = \eta_1 + \eta_2$ according to the decomposition $E^0\vert_{\mathcal{N}}= S^0\oplus (S^0)^\perp$. By using Assumption \ref{Assumption:normal_rates},  we estimate
\begin{align*}
\lvert \tilde{\mathcal A} \eta\rvert^2 &=  \langle \tilde{\mathcal A}^*\tilde{\mathcal A} \eta_1, \eta_1\rangle + 2  \langle \tilde{\mathcal A}^* \tilde{\mathcal A} \eta_1, \eta_2\rangle + \lvert \tilde{\mathcal A} \eta_2\rvert^2
\\
&\geq  r^2\langle \eta_1, Q^0 \eta_1\rangle + \langle O(r^3) \eta_1, \eta_1\rangle + r^2\langle A_{rr} \eta_1, \bar{\mathcal A} \eta_2\rangle +O(r^3)\lvert\eta_1\rvert\lvert\eta_2\rvert + C_1\lvert\eta_2\rvert^2.
\end{align*}
For the cross terms we estimate
\[
\lvert r^2\langle A_{rr} \eta_1, \bar{\mathcal A} \eta_2\rangle\rvert \leq \delta r^2\lvert\eta_1\rvert^2 + \frac{r^2C_3}{2\delta}\lvert\eta_2\rvert^2,
\]
so that, 
\begin{equation*}
\lvert\tilde {\mathcal A} \eta\rvert^2 \geq (C_0-\delta)r^2\lvert\eta_1\rvert^2  + (C_1-\frac{r^2C_3}{2\delta})\lvert\eta_2\rvert^2 + O(r^3)(\lvert\eta_1\rvert^2+\lvert\eta_2\rvert^2) \geq C_4r^2 (\lvert\eta_1\rvert^2 + \lvert\eta_2\rvert^2) = C_4 r^2 \lvert\eta\rvert^2,
\end{equation*}
where $\delta = C_0/2$ and $0<r<\varepsilon_0$ for some possibly even smaller $\varepsilon_0>0$ and $C_4=C_4(\varepsilon_0)$. 
\end{proof}

\begin{lemma}
\label{lemma:coordinate_independency}
The terms 
\[
\sum_\alpha x_a  \bar A_\alpha = x_\alpha c^\alpha M_\alpha^0, \quad \sum_\alpha \bar A_{\alpha\alpha}\quad  \text{and}\quad \sum_j c^j B_j^0, 
\]
do not depend on the choice of bundle coordinates $(N\vert_U, (x_j, x_\alpha)_{j,\alpha})$ and frames, where $c^j=c(dx^j)$ and $c^\alpha := c(e^\alpha)$ are the Clifford matrices in local frames. 
\end{lemma}

\begin{proof}
Let $V\subset Z$ with $V\cap U \neq \emptyset$ and new coordinates $(N\vert_V, (y_j, y_\alpha)_{j ,\alpha})$ so that $ y_\alpha = \Delta_\beta^ \alpha x_\beta$ and frames $\{\tilde\sigma_l\}_l$ of the bundle $E\vert_{U\cap V}$ associated to the $y$-coordinates so that $\tilde\sigma_l = O_l^k \sigma_k$  for some orthogonal transformations $\Delta : U\cap V \to SO(n-m)$ and $O : U\cap V \to SO(\dim E)$. If $\xi$ and $\tilde\xi$ denote the coordinates of a section of the bundle $\pi^*(E\vert_Z)$ with respect to the corresponding local frames, then $\tilde\xi = O \xi$.  Also the coordinate vectors and covectors are transforming as,
\begin{align*}
dy^j &= \frac{\partial y_j}{\partial x_k} dx^k, \quad  \partial_{y_j} = \frac{\partial x_k}{\partial y_j} \partial_{x_k} + \frac{\partial \Delta^\beta_\alpha}{\partial y_j} \Delta^\beta_\gamma x_\gamma \partial_{x_\alpha} 
\\
dy^\alpha &=  \frac{\partial \Delta^\alpha_\beta}{\partial x_k} x_\beta dx^k +  \Delta^\alpha_\beta dx^\beta, \quad \partial_{y_\alpha} = \Delta_\beta^\alpha \partial_{x_\beta}.
\end{align*}
If $e^\alpha = dx^\alpha\vert_U$, it follows that 
\[
c_Z(\tilde e^\alpha)O = \Delta_\beta^\alpha O c_Z(e^\beta) \quad \text{and}\quad  c_Z(dy^j) O = \frac{\partial y_j}{\partial x_k} O c_Z(dx^k),
\]
and the connection transform as
\[
\bar\nabla_{\partial_{y_\alpha}} \tilde \xi =  \Delta_\beta^\alpha O  \bar\nabla_{\partial_{x_\beta}}\xi = \Delta_\beta^\alpha  O\bar\nabla_{\partial_{x_\beta}}\xi \quad \text{and} \quad\bar\nabla_{{\mathcal H}(\partial_{y_j}\vert_{U\cap V})} \tilde\xi =   \frac{\partial x_k}{\partial y_j}  O\bar\nabla_{{\mathcal H}(\partial_{x_k}\vert_{U\cap V})}\xi. 
\]
For the first couple of terms ,
\begin{align*}
 y_\alpha(\nabla_{\partial_{y_\alpha}}{\mathcal A}) \tilde \xi &= x_\beta \Delta^\alpha_\beta \Delta^\alpha_\gamma O (\nabla_{\partial_{x_\alpha}}{\mathcal A})  \xi = x_\beta O (\nabla_{\partial_{x_\beta}}{\mathcal A}) \xi,
\\
(\nabla_{\partial_{y_\alpha}}\nabla_{\partial_{y_\alpha}}{\mathcal A}) \tilde\xi &=  \Delta^\alpha_\beta \Delta^\alpha_\gamma O (\nabla_{\partial_{x_\beta}} \nabla_{\partial_{x_\gamma}}{\mathcal A})  \xi =  O(\nabla_{\partial_{x_\beta}}\nabla_{\partial_{x_\beta}}{\mathcal A}) \tilde\xi.  
\end{align*}
For the third term,
\[
c_Z(dy^j)B^0_{\partial_{y_j}\vert_{U\cap V}} \tilde\xi=   \frac{\partial y_j}{\partial x_k}\frac{\partial x_l}{\partial y_j}   O c_Z(dx^k) B^0_{\partial_{x_l}\vert_{U\cap V}} \xi =  O c_Z(dx^j)B^0_{\partial_{x_j}\vert_{U\cap V}}\xi.
\]
\end{proof}

In the following  lemma we show how we can perturb the bundle map ${\mathcal A}$ to a new map with second order jet terms vanishing: 
\begin{lemma}
\label{lemma:whew}
Under Assumption \ref{Assumption:normal_rates}, we can choose our perturbation ${\mathcal A}$ so that $\nabla^2_{u,v}{\mathcal A}\vert_Z \equiv 0$ for every $u,v\in N$. 
\end{lemma}
\begin{proof}
Set $T_x = x_\alpha A_\alpha$. Then according to Assumption \ref{Assumption:normal_rates} and relation \eqref{eq:transversality_assumption_with_respect_to_connection}, $T_x(S^0\vert_Z) = S^1\vert_Z$. Also by Lemma~\ref{lem:basic_properties_of_M_v_w} and it's proof, for every $\xi_1\in S^0\vert_Z$, 
\begin{equation*}
T_x^*T_x\xi_1 = \sum_{\alpha\neq \beta}x_\alpha x_\beta (A^*_\alpha A_\beta + A^*_\beta A_\alpha)\xi_1 + \sum_\alpha x_\alpha^2 A_\alpha^*A_\alpha\xi_1 =  \sum_{\alpha\neq \beta}x_\alpha x_\beta M_{\alpha, \beta}\xi_1 + \lvert x\rvert^2 Q_0\xi_1 = \lvert x\rvert^2 Q_0\xi_1,
\end{equation*}
so that 
\[
\lvert T_x \xi_1\rvert^2 = \lvert x\rvert^2 \langle Q_0 \xi_1, \xi_1\rangle \geq C_0\lvert x\rvert^2 \lvert\xi_1\rvert^2.
\]
When  $\xi\in E^0\vert_Z = (S^0\oplus (S^0)^\perp)\vert_Z$, we decompose $\xi = \xi_1 + \xi_2$ and estimate
\begin{equation*}
\lvert(\bar {\mathcal A} + T_x)\xi\rvert^2 = \lvert\bar A \xi_2 + T_x\xi_1 + T_x \xi_2\rvert^2 = \lvert\bar{\mathcal A} \xi_2 + T_x\xi_2\rvert^2 +\lvert T_x \xi_1\rvert^2 + 2 \langle T_x \xi_1 , T_x \xi_2\rangle.
\end{equation*}
Each of the preceding terms estimates,
\begin{align*}
\lvert\bar {\mathcal A} \xi_2 + T_x\xi_2\rvert &\geq \lvert\bar {\mathcal A} \xi_2\rvert  - \lvert T_x\xi_2\rvert \geq (C_0- C_1\lvert x\rvert)\lvert\xi_2\rvert
\\
2 \lvert\langle T_x \xi_1 , T_x \xi_2\rangle\rvert &\leq 2 C_1^2\lvert x\rvert^2\lvert\xi_1\rvert\lvert\xi_2\rvert \leq  \delta \lvert x\rvert^2\lvert\xi_1\rvert^2 + \frac{\lvert x\rvert^2C_1^2}{\delta}\lvert\xi_2\rvert^2
\\
\lvert T_x \xi_1\rvert^2 &\geq   C_0\lvert x\rvert^2 \lvert\xi_1\rvert^2.
\end{align*}
Putting everything together and choosing $\epsilon_1>0$ small enough, we obtain
\begin{align*}
\lvert(\bar {\mathcal A} + T_x)\xi\rvert^2 &\geq \left[(C_0- C_1\lvert x\rvert)^2-\frac{\lvert x\rvert^2C_1^2}{\delta}\right] \lvert\xi_2\rvert^2 + (C_0 - \delta)\lvert x\rvert^2 \lvert\xi_1\rvert^2 \geq  C_3 \lvert x\rvert^2 \lvert\xi\rvert^2,
\end{align*}
for $\delta = C_0/2$ and every $0<\lvert x\rvert<\epsilon_1$. Hence
\[
\bar {\mathcal A} + T_x : E^0\vert_Z\rightarrow E^1\vert_Z
\]
is invertible for every $0<\lvert x\rvert<\epsilon_1$ and 
\[
\lvert(\bar {\mathcal A} + T_x)^{-1}\rvert \leq \frac{C}{\lvert x\rvert}.
\]

Introduce now a cut off function supported on $\mathcal{N}$, a tubular neighborhood around $Z$ of radius $\epsilon$ to be chosen later. Let $\hat{\rho} : [0,\infty) \rightarrow [0, 1]$ smooth cut off with $\hat{\rho}^{-1}(\{0\}) = [1, \infty),\, \hat{\rho}^{-1}(\{1\}) = [0,1/2]$ and strictly decreasing in $[1/2, 1]$, define $\rho(q) =\hat{\rho}(\tfrac{r(q)}{\epsilon})$ on $\mathcal{N}$ and extend as $0$ on $X-\mathcal{N}$. Recall the map $\tau$ from \eqref{eq:parallel_transport_map}. We define the bundle map,
\[
{\mathcal B} : E^0\rightarrow E^1, \, \quad {\mathcal B}(q) = \frac{\rho(q)}{2} \tau\circ\nabla^2_{v,v}({\mathcal A}\circ\tau^{-1} , \quad q = \exp_z(v), 
\]
for every $(z,v)\in N$. Derivating  relation \eqref{eq:dcond} we get that
\begin{equation}
\label{eq:second_order_cond}
u_\cdot \nabla_\alpha\nabla_\beta{\mathcal A} = -\nabla_\alpha\nabla_\beta{\mathcal A}^* u_\cdot 
\end{equation}
hence
\[
u_\cdot {\mathcal B} = - {\mathcal B}^* u_\cdot
\]
for every $u\in T^*X\vert_Z$. Hence ${\mathcal A} - {\mathcal B}$ satisfies (\ref{eq:dcond}) and using the expansion of ${\mathcal A}$ on $\mathcal{N}_U$ 
\[
{\mathcal A} - {\mathcal B} \,=\, \bar{\mathcal A}+ T_x +\,\frac{1 -  \rho(x)}{2} x_\alpha x_\beta\bar A_{\alpha \beta}\,+\, O(\lvert x\rvert^3).
\]
Choose $0<\epsilon < \epsilon_1$ so that for every $0<\lvert x\rvert<\epsilon$
\[
\left\lvert\frac{1 -  \rho(x)}{2} x_\alpha x_\beta \bar A_{\alpha\beta}\,+\, O(\lvert x\rvert^3)\right\rvert < \frac{\lvert x\rvert}{2C} < \frac{1}{2}\left\lvert \left(\bar{\mathcal A} + T_x\right)^{-1}\right\rvert^{-1}. 
\]
If 
\[
\Delta:= \frac{1 -  \rho(x)}{2} x_\alpha x_\beta \bar A_{\alpha\beta}\,+\, O(\lvert x\rvert^3),
\]
then if $({\mathcal A} - {\mathcal B})\xi=0$ for some $\xi\in E\vert_{\mathcal{N} - Z}$, then $\xi$ solves the equation
\[
-(\bar{\mathcal A}+ T_x)^{-1} \Delta \xi = \xi. 
\]
Therefore
\[
\lvert\xi\rvert= \lvert(\bar{\mathcal A}+ T_x)^{-1} \Delta \xi\rvert \leq \lvert(\bar{\mathcal A}+ T_x)^{-1}\rvert \lvert\Delta\rvert \lvert\xi\rvert \leq \frac{1}{2}\lvert\xi\rvert,
\]
so that $\xi=0$. Hence ${\mathcal A} - {\mathcal B}$ is invertible on $\mathcal{N} - Z$ and agrees with ${\mathcal A}$ outside $\mathcal{N}$ and therefore $Z_{{\mathcal A} - {\mathcal B}} = Z_{\mathcal A}$. Also by construction
\[
\nabla^2_{u,v}({\mathcal A} - {\mathcal B})\vert_Z \equiv 0 : E^0\vert_Z\rightarrow E^1\vert_Z 
\]
for every $u,v \in N$. Finally we have only changed the 2-jet of ${\mathcal A}$ around $Z$ to produce ${\mathcal A} - {\mathcal B}$. Assumption \ref{Assumption:normal_rates} still holds for ${\mathcal A} - {\mathcal B}$ since it relates only to the 1-jet of ${\mathcal A}$ on $Z$. Replacing ${\mathcal A}$ with ${\mathcal A}-{\mathcal B}$ on every component $Z$ of $Z_{\mathcal A}$ we are done.
\end{proof}

\medskip

\subsection{Proof of Proposition \ref{prop:basic_restriction_connection_properties}} 
\label{subApp:The_expansion_of_the_Spin_connection_along_Z}

Recall that since $(S^0 \oplus S^1)\vert_Z$ and its complement in $(E^0 \oplus E^1)\vert_Z$ are both $\mathbb{Z}_2$-graded $\text{Cl}(T^*\mathcal{N}\vert_Z)$-submodules and the projections satisfy 
\begin{equation}
\label{eq:projections_vs_Clifford_multiplication}
c(u)\circ P^{1-i} = P^i \circ c(u)   \qquad \text{and} \qquad c(v)\circ P^{(1-i)+}_\ell = P^{i+}_\ell \circ c(v),
\end{equation}
for every $u\in T^*\mathcal{N}\vert_Z$ and every $v\in T^*Z$. We use the dot notation $w_\cdot$ for $c(w)$ throughout the proof. The connection $\bar\nabla$ is the push-forward of the connection on the isometric bundle $\Lambda^*N\otimes S^{i+}$ through the isomorphisms \eqref{eq:decompositions}. The later is metric compatible and so is the former.  

Recall the parallelism condition of the compatibility with the Clifford multiplication,
\begin{equation}
\label{eq:compatibility}
[\tilde\nabla^{\tilde E}, w_\cdot ]\xi = (\tilde\nabla^{T^*\mathcal{N}}w)_\cdot \xi,
\end{equation}
for every $\xi \in C^\infty(Z; E\vert_Z)$.

When $\xi \in C^\infty(Z; (S^i\vert_Z)^\perp)$ we apply to both sides projection $1_{E^i\vert_Z} - P^i$ and use \eqref{eq:projections_vs_Clifford_multiplication} to obtain \eqref{prop:basic_restriction_connection_properties1}.  

In proceeding to show \eqref{prop:basic_restriction_connection_properties2}, we divide the proof into two cases: 

\setcounter{case}{0}

\begin{case} 
Assume that $\xi\in C^\infty(Z; S^+_\ell)$. When $w\in C^\infty(Z; T^*Z)$, then $w_\cdot \xi \in  C^\infty(Z; S^+_\ell)$ and
\begin{align*}
[\bar\nabla, w_\cdot ]\xi &= P^+_\ell [\tilde\nabla^{E\vert_Z}, w_\cdot ]\xi , \qquad (\text{by Definition~\ref{eq:connection_bar_nabla}}) 
\\
&=   P^+_\ell  ( \tilde\nabla^{T^*\mathcal{N}}w_\cdot \xi), \qquad (\text{ by \eqref{eq:compatibility}})
\\
&= P^+_\ell (  \nabla^{T^*Z}w_\cdot \xi +  p_{N^*}\nabla^{T^*\mathcal{N}}w_\cdot \xi) 
\\
&= \nabla^{T^*Z}w_\cdot \xi.
\end{align*}
Here $p_{N^*}: T^*\mathcal{N}\vert_Z \to N^*$ is the projection and the last equality follows since the term  $p_{N^*}\nabla^{T^*\mathcal{N}}w_\cdot \xi$ is a section pointwise orthogonal to $S^+_\ell$. When  $w\in C^\infty(Z; N^*)$ the identity follows by the definition of $\bar\nabla$. This finishes Case 1.
\end{case}

\begin{case}
Assume now that $\xi = \theta_\cdot \xi^+$ for some section $\theta\in C^\infty(Z; \Lambda^k N^*)$ and some $\xi^+\in C^\infty(Z; S^+_\ell)$. When $w\in C^\infty(Z; T^*Z)$, then $w_\cdot \xi^+ \in C^\infty(Z; S^+_\ell)$ and
 \begin{align*}
[\bar\nabla, w_\cdot ] (\theta_\cdot\xi^+) &= \bar\nabla ((-1)^k\theta_\cdot w_\cdot \xi^+) - w_\cdot \bar\nabla(\theta_\cdot \xi^+)
\\
&= (-1)^k \left[ (\nabla^{N^*}\theta)_\cdot w_\cdot \xi^+ + \theta_\cdot \bar\nabla(w_\cdot \xi^+)\right]  - w_\cdot (\nabla^{N^*}\theta )_\cdot \xi^+ - w_\cdot \theta_\cdot \bar\nabla\xi^+
\\
&=   w_\cdot (\nabla^{N^*}\theta)_\cdot \xi^+ + (-1)^k\theta_\cdot [\bar\nabla,  w_\cdot] \xi^+ - w_\cdot(\nabla^{N^*} \theta)_\cdot \xi^+
\\
&=  (-1)^k \theta_\cdot (\nabla^{T^*Z} w)_\cdot \xi^+, \qquad (\text{by Case 1})
\\
&= \nabla^{T^*Z}w_\cdot (\theta_\cdot\xi^+),
\end{align*}
where in the second line we used Definition~\ref{eq:connection_bar_nabla}.
When $w\in C^\infty(Z;N^*)$, we use the identity $w_\cdot \theta_\cdot \xi^+ = ( \hat c(w)\theta)_\cdot \xi^+$ proved in \cite{m}[Lemma 4.1]. We calculate,
\begin{align*}
[\bar\nabla, w_\cdot ] (\theta_\cdot\xi^+) &= \bar\nabla ( w_\cdot\theta_\cdot \xi^+) - w_\cdot (\nabla^{N^*}\theta)_\cdot \xi^+ - w_\cdot \theta_\cdot \bar\nabla \xi^+, \qquad (\text{by Definition~\ref{eq:connection_bar_nabla}}) 
\\
&= \bar\nabla (  (\hat c(w)\theta)_\cdot \xi^+)  - (\hat c(w)\theta)_\cdot \bar\nabla \xi^+ - (\hat c(w) \nabla^{N^*} \theta)_\cdot \xi^+
\\
&= (\nabla^{N^*}(\hat c(w)\theta)_\cdot \xi^+  - (\hat c (w)\nabla^{N^*} \theta)_\cdot \xi^+ , \qquad (\text{by Definition~\ref{eq:connection_bar_nabla}})
\\
&= ([\nabla^{N^*} , \hat c(w)]\theta)_\cdot \xi^+.
\end{align*}
 Here $\hat c(w) \theta = w \wedge \theta - \iota_w^\sharp \theta \in \Lambda^*N^*$. As a consequence of $\nabla^{N^*}$ being  compatible with the Riemannian metric on $\Lambda^*N^*$, we have that $[\nabla^{N^*} , \hat c(w)]\theta = \hat c(\nabla^{N^*}w)\theta$ and therefore, 
\[
([\nabla^{N^*}, \hat c(w)]\theta)_\cdot \xi^+ = ( \hat c(\nabla^{N^*}w)\theta)_\cdot \xi^+ =  \nabla^{N^*} w_\cdot( \theta_\cdot \xi^+),
\]
as required. This finishes Case 2.
\end{case}
The proof of the proposition is complete.\hfill $\square$

\medskip

\subsection{The expansion of the \texorpdfstring{$\{\tau_j,\tau_\alpha\}_{j,\alpha}$}{} frames and the Clifford action.}
\label{subApp:tau_j_tau_a_frames}

In expanding the Dirac operator $\tilde D = c(\tau^A) \tilde\nabla_{\tau_{\mathcal A}}$ along the fibers of the normal bundle in Section~\ref{sec:structure_of_D_sA_along_the_normal_fibers}, Lemma~\ref {lem:Dtaylorexp}, one has to expand also the frames $\{\tau_A\}_A,\ A=j,\alpha$. Recall the frames $\{e_j, e_\alpha\}_{j,\alpha}$ and the connection components $\tilde\nabla^{T\mathcal{N}}_{\tau_A} \tau_B = \omega_{AB}^\Delta$, satisfying \eqref{eq:connection_comp_rates}. The connection $\nabla^N$ of the bundle $N\to Z$ induces horizontal lifts of $TZ$ in the tangent space $TN$ and a bundle isomoprhism $H:\pi^*(TZ \oplus N) \to TN$, given in local coordinates by  
\begin{equation}
\label{eq:horizontal_lift}
 H((x_j ,x_\alpha), w) = \begin{cases} ((x_j, x_\alpha), d_j^k\partial_k -  x_\gamma \bar\omega_{j \gamma}^\beta \partial_\beta),& \quad \text{if $w=(e_j,0)$,}
\\ 
((x_j, x_\alpha), \partial_\alpha),& \quad \text{if $w=(0,e_\alpha)$.}\end{cases}
\end{equation}
We introduce the vertical and horizontal distributions by ${\mathcal V}= H(\pi^*N)$ and ${\mathcal H} = H(\pi^*TZ)$ respectively so that $TN = {\mathcal V} \oplus {\mathcal H}$. We also denote  
\[
h_A = H e_A \quad \text{and the algebraic dual by}\quad h^A = (H^*)^{-1} e^A, 
\]
for every $A=j,\alpha$ where $H^*$ denotes the adjoint map 
\begin{equation}
\label{eq:dual_horizontal_lift}
H^*: T^*N\to \pi^*(T^*Z \oplus N^*).
\end{equation}
The dual frames are described alternatively, by 
\begin{equation}
\label{eq:dual_horizontal_lift_frames}
h^\alpha = dx^\alpha + x_\beta \omega_{k \beta}^\alpha d^k_\ell dx^\ell \quad \text{and}\quad h^j = d^j_\ell dx^\ell.
\end{equation}
The horizontal lifts appear in the expansion of the  $\{\tau_j,\tau_\alpha\}_{j,\alpha}$ frames because of the curvature of the normal bundle $(N, \nabla^N)\to Z$:

\begin{prop}
\label{prop:expansions_of_orthonormal_frames}
The frames $\{\tau_i\}_i$ and $\{ \tau_\alpha\}_\alpha$ expand in the radial directions as
\begin{align}
\label{eq:on_frames_expansions}
\tau_\alpha = h_\alpha + O( r^2 \partial^{\mathcal V} + r \partial^{\mathcal H})\quad \text{and}\quad  \tau_j = h_j + O( r^2 \partial^{\mathcal V} + r \partial^{\mathcal H}),
\end{align}
where $O(r^k\partial^{\mathcal H})$ and $O(r^k\partial^{\mathcal V})$ will denote expressions in horizontal and vertical lifts correspondingly with coefficient components vanishing up to order $r^k$ when $r\to 0^+$.
\end{prop}
\begin{proof}
The frames $\{\tau\}_i$ and $\{ \tau_\alpha\}_\alpha$ expand with respect to tagential and normal Fermi fields as 
\begin{align*}
\tau_\alpha = \partial_\alpha + O( r^2 \partial^N + r \partial^H)\quad \text{and}\quad  \tau_j = d_j^k \partial_k - x_\beta \bar\omega_{j \beta}^\alpha \partial_\alpha + O( r^2 \partial^N + r \partial^H).
\end{align*}
Using the coordinate formulas of the bundle map $H:\pi^*(TZ\oplus N)\to TN$, we obtain the expansions \eqref{eq:on_frames_expansions}. 
\end{proof}

\begin{rem}
\begin{enumerate}
\item 
The expansions of the frames $ \{\tau_j\}_j$ are being developed up to order $O( r^2)\partial^N$ in the spherical Fermi fields because  terms of order $O( r)\partial^N$ become significant after the rescaling $\{w_\alpha = \sqrt{s} x_\alpha\}_\alpha$ is applied. In particular the horizontal lifts $h_j$ are precisely the terms of significant order in the expansion of $\tilde D$.   

\item If the bundle $(N, \nabla^N)\to Z$ is flat, then the tangential and spherical Fermi-fields suffice to describe the expansion of $\tilde D$.
\end{enumerate}
\end{rem}

There is also a Clifford action of $T^*N$ to the fibers:
\begin{defn}
\label{defn:definition_of_mathfrac_Clifford}
We use the bundle isomorphism $H^*$ from \eqref{eq:dual_horizontal_lift} to define $\mathfrak{c}_N$ as the composition
\[
\xymatrix{
\mathfrak{c}_N: T^*N \otimes \pi^*(E^i\vert_Z) \ar[r]^-{H^*\otimes 1}&  \pi^*( (T^*Z\oplus N^*)\otimes E^i\vert_Z) \ar[r]^-{c\vert_Z}& \pi^*(E^{1-i}\vert_Z). 
}
\] 
\end{defn}

To describe the components of $\mathfrak{c}_N$ in the local frames $\{h_A, \sigma_k, f_l\}_{A, k, l}$, we observe that, 
\begin{equation}
\label{eq:clifford_matrix_restricted}
\tilde c = c_k^{A l} \tau_A\otimes \sigma^k\otimes f_l.
\end{equation}
Since the Clifford multiplication and the local frames are $\tilde\nabla^{\tilde E}$-parallel, the components $c_k^{A l}:U \to \mathbb{R}$ are the components of the representation matrices $c^A, \, A=j,\alpha$ of the restricted Clifford multiplication $c\vert_Z: T^*X\vert_Z \otimes E^i\vert_Z\to E^{1-i}\vert_Z,\ i =0,1$. It's pullback therefore give a map with components 
\[
( c_k^{A l} \circ \pi)  e_A \otimes \sigma^k \otimes f_l 
\]
that is, the same components as those in equation \eqref{eq:clifford_matrix_restricted}. Because $H^*( h^A) = e^A,\ A=j,\alpha$, we obtain 
\[
\mathfrak{c}_N= ( c_k^{A l} \circ \pi) h_A\otimes \sigma^k \otimes f_l,
\]
to be the expression of $\mathfrak{c}_N$ in these local frames.

\medskip
 
\subsection{The total space \texorpdfstring{$TN\to N$}{}.}
\label{subApp:The_total_space_of_the_normal_bundle_N_to_Z}

In this subsection we introduce the connections and the Riemmanian metric we will be using for the analytical part of the paper. Recall the splitting $TN = {\mathcal V} \oplus {\mathcal H}$ into vertical and horizontal distributions, introduced in Appendix subsection~\ref{subApp:tau_j_tau_a_frames}. The splitting is crucial in the expansion of Lemma~\ref{lem:Dtaylorexp}. The vertical and horizontal distributions are used in Definition~\ref{defn:vrtical_horizontal_Dirac} to introduce the vertical and horizontal operators $\slashed D_0$ and $\bar D^Z$ respectively. Under the rescaling $\{w_\alpha = \sqrt{s} x_\alpha\}_\alpha$, the fields $\{h_\alpha\}_\alpha$ re-scale by a factor of $\sqrt{s}$ while the fields $\{h_j\}_j$ are invariant. To account for different rates in the re-scaling, we use the bundle isomorphism $H:\pi^*(TZ \oplus N) \to TN$ from \eqref{eq:horizontal_lift} to push forward the Riemannian metric and connection into the total space $TN$: we define a new Riemannian metric on the total space $TN$, as the push-forward
\begin{equation}
\label{eq:g_TN}
g_{TN}:TN\otimes TN \to \mathbb{R},\quad  Hv \otimes Hw \mapsto g(v,w), \quad v,w\in \pi^*(TZ \oplus N).
\end{equation}
The pullback bundle $\pi^*(T\mathcal{N}\vert_Z)=\pi^*(TZ \oplus N) $ with the pullback connection $\nabla^{\pi^*N}  \oplus \nabla^{\pi^*TZ}$ induce through the bundle isomorphism $H$, a new connection $\bar\nabla^{TN}$, 
\begin{equation}
\label{eq:connection}
\bar\nabla^{TN}= H \circ (\nabla^{\pi^*N}  \oplus \nabla^{\pi^*TZ}) \circ H^{-1}. 
\end{equation}
The frames $\{h_A = He_A\}_{A=j,\alpha}$ of $TN$ with dual coframes $\{h^A\}_{A=j,\alpha}$ are $g_{TN}$-orthonormal and the connection writes,
\begin{equation}
\label{eq:local_components_for_bar_nabla_TN}
\bar\nabla^{TN}_{h_j} h_k = \bar\omega_{jk}^l h_l, \quad \bar\nabla^{TN}_{h_j} h_\alpha = \bar\omega_{j \alpha}^ \beta h_\beta, \quad \text{and}\quad \bar\nabla^{TN}_{h_\alpha} h_A = 0, 
\end{equation}
for every $j,k,\alpha, A$, where $\omega^A_{B\Gamma}$ are the components in \eqref{eq:connection_comp_rates}. 

Similarly, the definition for the connection in the dual bundles is given by,
\begin{equation}
\label{eq:dual_connection}
\bar\nabla^{T^*N}= (H^*)^{-1} \circ (\nabla^{\pi^*N^*}  \oplus \nabla^{\pi^*T^*Z}) \circ H^*. 
\end{equation}

\medskip

\subsection{The expansion of the volume from along \texorpdfstring{$Z$}{} when \texorpdfstring{$r\to 0^+$}{}.}\hfill
\label{subApp:The_expansion_of_the_volume_from_along_Z}

Recall that $g = \exp^* g_X$. The volume element of $g$ is expressed in local coordinates $d\mathrm{vol}^\mathcal{N} =  (\det g)^{1/2}\, \bigwedge_j dx^j \wedge \bigwedge_\alpha dx^\alpha $. Since $\partial_A (\det g)^{1/2} = ({\rm div}_g \partial_A) (\det g)^{1/2}$, by using \cite{g}[Theorem 9.22] and relations \eqref{eq:Christoffel_relations_0}, we have,
\[
{\rm div}_g \partial_\alpha = \sum_j \Gamma_{\alpha j}^j + \sum_\beta \Gamma_{\alpha\beta}^\beta =  H_\alpha + O(r) 
\]
as $r \to 0^+$. Therefore, in Fermi coordinates, we have an expansion,
\begin{equation}
\label{eq:volume-expansion}
d\mathrm{vol}^\mathcal{N}= \left(1 + \sum_\alpha x_\alpha H_\alpha + O(r^2)\right)\, (\det \bar g)^{1/2}\bigwedge_j dx^j \wedge  \bigwedge_\alpha  dx^\alpha.
\end{equation}
The local volume element  $ (\det \bar g)^{1/2} \bigwedge_j dx^j \wedge  \bigwedge_\alpha  dx^\alpha$ is independent of the bundle coordinates and defines a volume element $d\mathrm{vol}^N$ that is expressed in the frames $\{h^A\}_A$ in \eqref{eq:dual_horizontal_lift_frames} as $h^1 \wedge \dots \wedge h^n $, on the total space of the normal bundle $N$. More to the point $d\mathrm{vol}^N$ is the associated volume element induced by $g_{TN}$ in an orthonormal coframe of lifts $\{h^A\}_A$. Finally there exist $\varepsilon_0>0$ so that for every $0<\varepsilon< \varepsilon_0$, the corresponding densities satisfy, 
\begin{align}
\label{eq:density_comparison}
\frac{1}{2}\lvert d\mathrm{vol}^\mathcal{N}\rvert_q \leq \lvert d\mathrm{vol}^N\rvert_q \leq 2\lvert d\mathrm{vol}^\mathcal{N}\rvert_q, 
\end{align}
for every $q\in \mathcal{N}_\varepsilon$.

\medskip
 
\subsection{The Clifford bundle \texorpdfstring{$(\pi^*(E\vert_Z), \bar\nabla, \mathfrak{c}_N )$}{}.}
\label{subApp:The_pullback_bundle_E_Z_bar_nabla_E_Z_c_Z}

Recall the connection $\bar\nabla^{E\vert_Z}$, introduced in Definition \ref{eq:connection_bar_nabla}. The restriction bundle $(E\vert_Z, \bar\nabla^{E\vert_Z}, c_Z )$ pullback through the normal bundle map $\pi:N\to Z$, inducing new bundles over the total space $N$, so that the diagram    
\[
\xymatrix{
\left(\pi^*(E\vert_Z),\, \bar\nabla^{\pi^*(E\vert_Z)} \right)\ar[r] \ar[d] & \left(E\vert_Z,\, \bar\nabla^{E\vert_Z}\right) \ar[d]
\\
N \ar[r]^\pi & Z}
\] 
commute. According to \cite{kn}[Ch.II, Prop. 6.2], the main property of the pullback connection $\bar\nabla^{\pi^*(E\vert_Z)}$ is that, for every section $\xi :Z \to E\vert_Z$,
\begin{equation}
\label{eq:pullback_connection_universal_property}
\bar\nabla^{\pi^*(E\vert_Z)}_v (\pi^*\xi) = \pi^*(\bar\nabla^{E\vert_Z}_{\pi_*v} \xi), \quad \text{for every $v\in TN$.}
\end{equation}
Therefore in bundle coordinates $(N_U, (x_j, x_\alpha)_{j, \alpha})$, using the frames $\{h_A = H e_A\}_{A=j,\alpha}$ and the frames $\{\sigma_\ell, f_k\}_{\ell, k}$,
\begin{equation}
\label{eq:pullback_bar_connection_components}
\bar\nabla_{h_\alpha}^{\pi^*(E^i\vert_Z)} = h_\alpha \quad \text{and} \quad  \bar\nabla_{h_j}^{\pi^*(E^i\vert_Z)} = h_j + \phi_j^i,\ i=0,1,
\end{equation}
where the matrices $\phi_j^i$ are introduced in \eqref{eq:comparing_tilde_nabla_orthonormal_to_bar_nabla_orthonormal}. Likewise, the curvature $F^{\pi^*(E\vert_Z)}$ of the connection $\bar\nabla^{\pi^*(E\vert_Z)})$ satisfies,
\begin{equation}
\label{eq:pullback_curvature_universal_property}
F^{\pi^*(E\vert_Z)}(u,v)(\pi^*\xi) = \pi^*( F^{E\vert_Z}( \pi_*u, \pi_*v) \xi),
\end{equation}
where $F^{E\vert_Z}$ denotes the curvature of the bundle $(E\vert_Z, \bar\nabla)$. For notational simplicity, every section $\xi: Z \to E\vert_Z$ that pulls back to $\pi^* \xi: N \to \pi^*E$, will be denoted by again by $\xi$.  We also denote $\bar\nabla^{\pi^*(E\vert_Z)}$ as simply $\bar\nabla$. 

Proposition~\ref{prop:basic_restriction_connection_properties} has the following analogue for the pullback connection: 
\begin{prop}
\label{prop:basic_extension_bar_connection_properties}
\begin{enumerate}
\item 
\label{prop:basic_extension_bar_connection_properties2}
The connection $\bar\nabla^{TN}$ is compatible with $g^{TN}$. Its torsion $T:=T^{\bar\nabla^{TN}}: TN \otimes TN \to TN$ is given by 
\[
((p,v), X \otimes Y) \mapsto HF^{\nabla^N}( \pi_*X, \pi_*Y)v, 
\]
for every $(p,v) \in N$.

\item
\label{prop:basic_extension_bar_connection_properties1} 
For every $\theta\in C^\infty(N; T^*N)$
\begin{equation}
    \label{eq:basic_extension_bar_connection_properties1}
[\bar\nabla, \mathfrak{c}_N(\theta)] = \mathfrak{c}_N( \bar\nabla^{T^*N} \theta).  
\end{equation}

\item
\label{prop:basic_extension_bar_connection_properties3}
The connection $\bar\nabla$ is compatible with the inner product on the bundle $\pi^*(E\vert_Z)$. Furthermore it reduces to a sum of connections, each of which is associated to each summand of the decomposition of $\pi^*(S^i\vert_Z)$ into eigenbundles $\pi^*(S^i_{\ell k})$ introduced in decompositions \eqref{eq:eigenspaces_of_Q_i} and \eqref{eq:decompositions_of_S_ell}.

\item 
\label{prop:basic_extension_bar_connection_properties4}
We define 
\[
\bar\nabla^{\mathcal V}: C^\infty(N; \pi^*E) \to C^\infty(N; {\mathcal V}^* \otimes \pi^*E),\quad  \bar\nabla^{\mathcal V} = h^\alpha \otimes \bar\nabla_{h_\alpha},
\]
and 
\[
\bar\nabla^{\mathcal H}: C^\infty(N;  \pi^*E) \to C^\infty(N; {\mathcal H}^* \otimes \pi^*E),\quad \bar\nabla^{\mathcal H} = h^j\otimes \bar\nabla_{h_j}.
\]
Then the expressions are independent of the frames used. Their adjoints are given by 
\[
\bar\nabla^{{\mathcal V}*}(h^\alpha \otimes s_\alpha)  =-\bar\nabla_{h_\alpha} s_\alpha \quad \text{and}\quad \bar\nabla^{{\mathcal H}*}(h^j \otimes s_j) = - \bar\nabla_{h_j} s_j .
\] 

\item 
\label{prop:basic_extension_bar_connection_properties5}
We define the Hessian operator with respect to the connection $\bar\nabla \otimes 1 + 1\otimes \bar\nabla$ of $TN \otimes \pi^* E$ as 
\[
\mathrm{Hess}(u,v)\xi = \bar\nabla_u \bar\nabla_v \xi - \bar\nabla_{ \bar\nabla_u v} \xi, 
\]
for every $u,v\in TN$. Then, the skew-symmetrization of the Hessian operator is related to the curvature $F^{\pi^*(E\vert_Z)}$ of the bundle $(\pi^* (E\vert_Z), \bar\nabla)$ by the equation 
\begin{equation}
\label{eq:basic_extension_bar_connection_properties5}
\mathrm{Hess}(u,v)\xi - \mathrm{Hess}(v,u)\xi = F^{\pi^*(E\vert_Z)}(u,v) \xi - \bar\nabla_{T(u,v)} \xi, 
\end{equation}
for every $u,v\in TN$. Moreover, the Hessian is symmetric when at least one of the $u$ and $v$, belong in ${\mathcal V}$. 
\end{enumerate}
\end{prop}

\begin{proof}
We use local orthonormal frames $\{e_j\}_j$ that are $\nabla^{TZ}$-parallel at $p \in Z$ and $\{e_\alpha\}_\alpha$ that are $\nabla^N$-parallel at $p$ with their dual frames $\{e^A\}_A$. Then equations \eqref{eq:connection_comp_rates} hold. Introduce lifts $\{h_A\}_A$ with dual frames $\{h^A\}_A$ and relations $\eqref{eq:local_components_for_bar_nabla_TN}$ hold. Finally introduce frames $\{\sigma_\ell\}_\ell,\ \{f_k\}_k$. The Clifford action obeys $\mathfrak{c}_N(h^A) = c(e^A) =: c^A$, the matrix representation with respect to the preceding frames. 

Part \eqref{prop:basic_extension_bar_connection_properties3} follows by the definition of $\bar\nabla$. To prove part \eqref{prop:basic_extension_bar_connection_properties2} we write $g^{TN} = h^k\otimes h^k + h^\alpha \otimes h^\alpha$. By the local expressions \eqref{eq:local_components_for_bar_nabla_TN} it follows that $\bar\nabla_{h_\alpha} g^{TN} = 0$ and that  
\begin{align*}
\bar\nabla_{h_j} g^{TN} = - (\bar\omega_{jk}^l + \bar\omega_{jl}^k) h^k \otimes h^l - (\bar\omega_{j\alpha}^\beta + \bar\omega_{j \beta}^\alpha) h^\alpha \otimes h^\beta =0, \qquad (\text{by  \eqref{eq:connection_comp_rates}}). 
\end{align*}
We calculate the torsion $T$ of the connection $\bar\nabla^{TN}$ over the fiber $N_p$ of the total space $N$ at $p$. We have that 
\[
\bar\nabla_{h_A} h_B=0\quad \text{and}\quad [h_\alpha, h_\beta] = [\partial_\alpha , \partial_\beta] =0, 
\]
for every $A,B,\alpha,\beta$. Therefore $T(h_\alpha, h_\beta)=0$. Also $[ h_j, h_\alpha]= \bar\omega_{j \alpha}^\beta h_\beta  = 0$ and therefore $T(h_j, h_\alpha)=0$. Finally one of the definitions of the curvature of $(N, \nabla^N)$ is as the obstruction to the integrability of the horizontal distribution induced by $\nabla^N$ in $TN$. We review the calculation, 
\begin{align*}
[h_j, h_k] &= [ d_j^l \partial_l , d_k^\ell \partial_\ell] + [x_\beta \bar\omega_{j \beta}^\alpha \partial_\alpha , x_\gamma \bar\omega_{k \gamma}^\delta \partial_\delta] 
- [ d_j^l \partial_l ,x_\gamma \bar\omega_{k \gamma}^\delta \partial_\delta] - [x_\beta \bar\omega_{j \beta}^\alpha \partial_\alpha  , d_k^\ell \partial_\ell] 
\\
&= \left( e_j(d_k^l) -  e_k(d_j^l) \right)  \partial_l +  x_\alpha\left(e_k(\bar\omega_{j\alpha}^\beta) - e_j(\bar \omega_{k \alpha}^\beta)\right) \partial_\beta 
\\
&= 0+  H(F^{\nabla^N}( e_k , e_j) ( x_\alpha e_\alpha)),
\end{align*}
over $N_p$ and we obtain $T( h_j, h_k) =  H(F^{\nabla^N}( \pi_* h_j , \pi_* h_k) ( x_\alpha e_\alpha))$. This finishes the computation of the torsion.

To prove part \eqref{prop:basic_extension_bar_connection_properties1} we recall the decomposition $T^*N = {\mathcal H}^*\oplus {\mathcal V}^*$. First we consider sections $\theta\in C^\infty(N; {\mathcal H}^*)$ so that $H^* \theta = r_1 w_1 + r_2 w_2$ with $w_1\in C^\infty(Z;T^*Z)$ and  $w_2\in C^\infty(Z;N^*)$ and $r_1, r_2 \in C^\infty(N)$. Since $\pi_* h_j = e_j$, 
\begin{align*}
[\bar\nabla_{h_j}, \mathfrak{c}_N(\theta)] = \sum_{i =1}^2 \left( h_j(r_i)  c(w_i)+ r_i  [\bar\nabla_{e_j}, c(w_i)] \right),
\end{align*}
by definition of $\bar\nabla$. From equation \eqref{prop:basic_restriction_connection_properties2}
\[
[\bar\nabla_{e_j}, c(w_i)] = \begin{cases} c( \nabla^{T^*Z}_{e_j} w_1),& \quad \text{if $i=1$,}
\\ c(\nabla^{N^*}_{e_j}w_2),& \quad \text{if $i=2$.}
\end{cases} .
\]
Substituting back and using \eqref{eq:dual_connection}, we obtain 
\[
[\bar\nabla_{h_j}, \mathfrak{c}_N(\theta)] =  c\left((\nabla^{N^*}_{h_j} \oplus \nabla^{T^*Z}_{h_j})( r_1w_1 + r_2w_2)\right)=  \mathfrak{c}_N( \bar\nabla^{T^*N}_{h_j} \theta).  
\]
Finally we consider section $\theta\in C^\infty(N; {\mathcal V}^*)$ so that $\theta = \theta_A h^A$. Then $\mathfrak{c}_N(\theta) = \theta_A c(e^A) = \theta_A c^A$ and therefore, 
\begin{align*}
[\bar\nabla_{h_\alpha}, \mathfrak{c}_N(\theta_A h^A)] &= (h_\alpha\theta_A) c^A +  \theta_A [\partial_\alpha, c^A] =  \mathfrak{c}_N(\bar\nabla_{h_\alpha} (\theta_A h ^A)).
\end{align*}
Any other section of $T^*N$ is locally a linear combination of sections $\theta$ of the forms examined. Therefore part \eqref{prop:basic_extension_bar_connection_properties1} hold in general.

To prove part \eqref{prop:basic_extension_bar_connection_properties4} let $\xi\in C^\infty(N; \pi^*E)$ and $s\in C^\infty(N; {\mathcal H}^*\otimes \pi^*E)$. In local frames $s = h^j\otimes s_j$. We define $Y = \sum_j Y^j h_j$ where $Y^j= \langle \xi, s_j\rangle$ and set $\omega = \iota_Y d\mathrm{vol}^N$. 

By Cartan's identity $d \omega = {\mathcal L}_Y\, d \mathrm{vol}^N$ where $d\mathrm{vol}^N = h^1 \wedge \dots \wedge h^n$ so that
\[
{\mathcal L}_Y d \mathrm{vol}^N = \sum_{j,A} h^1 \wedge \dots \wedge {\mathcal L}_{Y^j h_j} h^A \wedge \dots \wedge h^n.   
\]
Calculating over $N_p$, we have that 
\begin{align*}
{\mathcal L}_{Y^j h_j} h^k &= - h^k( [Y^j h_j, h _l]) h^l  - h^k([Y^j h_j, h_\beta]) h^\beta = dY^k,
\\
{\mathcal L}_{Y^j h_j} h^\alpha &= - h^\alpha( [Y^j h_j, h _l]) h^l  - h^\alpha([Y^j h_j, h_\beta]) h^\beta = - h^\alpha([h_j,h_l])Y^j h^l,
\end{align*}
so that ${\mathcal L}_Y d \mathrm{vol}^N = h_j(Y^j) \, d \mathrm{vol}^N$. Then 
\begin{align*}
h_j(Y^j) &= \langle \bar\nabla_{h_j} \xi , s_j \rangle + \langle \xi, \nabla_{h_j} s_j\rangle
\\
&=    \langle h^j\otimes \bar\nabla_{h_j} \xi , h^j\otimes s_j \rangle + \langle \xi, \bar\nabla_{h_j} s_j\rangle, \qquad (\text{since $\lvert h^j\rvert=1$})
\\
&= \langle \bar\nabla^{\mathcal H} \xi, s\rangle - \langle \xi, \bar\nabla^{{\mathcal H}*} s \rangle,
\end{align*}
that is 
\[
d\omega= (\langle \bar\nabla^{\mathcal H} \xi, s\rangle - \langle \xi, \bar\nabla^{{\mathcal H}*} s \rangle)\,  d \mathrm{vol}^N
\]
This finishes the proof of the local formula for $\nabla^{{\mathcal H}*}$. The formula for $\nabla^{{\mathcal V}*}$ follows the same pattern. The formula for the skew-symmetrization of the Hessian in part \eqref{prop:basic_extension_bar_connection_properties5} is a straightforward calculation. Also the asserted symmetry of the Hessian follows by the formula, since $T(h_\alpha, h_\beta)= T(h_j, h_\alpha)=0$ and $F^{\pi^*(E\vert_Z)}(h_\alpha, h_\beta) = F^{\pi^*(E\vert_Z)}(h_j, h_\alpha) =0$ by \eqref{eq:pullback_curvature_universal_property}.
\end{proof}

\bigskip

\setcounter{equation}{0}
\renewcommand\theequation{C.\arabic{equation}}
\section{Various analytical proofs}
\label{subApp:various_analytical_proofs}

\begin{proof}[Proof of Proposition~\ref{prop:improved_concentration_Prop}(The improved concentration principle)]
The proof uses a technique described in \cite{as}[Ch.4, Th.4.4,pp.59]. Let $\rho :X \to \mathbb{R}$ be smooth and $\xi \in C^\infty(X;E^0)$. Using that $D_s(\rho \xi) = d\rho_\cdot \xi + \rho D_s \xi$, we calculate
\begin{align*}
\langle D_s \xi, D_s(\rho^2 \xi)\rangle_{L^2(X)} &=  \langle \rho D_s \xi, D_s(\rho\xi)\rangle_{L^2(X)} + \langle D_s \xi, \rho d\rho_\cdot \xi\rangle_{L^2(X)}
\\
&= \|D_s(\rho \xi)\|_{L^2(X)}^2 - \langle  d\rho_\cdot\xi, D_s(\rho \xi)\rangle_{L^2(X)} + \langle  d\rho_\cdot \xi, \rho D_s \xi\rangle_{L^2(X)}
\\
&=  \|D_s(\rho \xi)\|_{L^2(X)}^2  - \| \lvert d\rho\rvert \xi\|_{L^2(X)}^2. 
\end{align*}

Let now $\delta>0$ a number satisfying the assumptions of the statement. Assume additionally that $\mathrm{supp\,}\rho \subset \Omega(\delta)$. Then, by a concentration estimate, for every $\xi \in C^\infty(X;E)$,            
\begin{equation*}
\| D_s (\rho \xi)\|_{L^2(X)}^2 \geq s\langle B_{\mathcal A} \xi, \rho^2 \xi \rangle_{L^2(X)} + s^2\|{\mathcal A}(\rho \xi)\|_{L^2(X)}^2 \geq (s^2\kappa_\delta^2 - \lvert s\rvert C_0)\|\rho \xi\|_{L^2(X)}^2,
\end{equation*}
where we set,
\begin{equation}
\begin{aligned}
\label{eq:useful_constants}
C_0 &= \sup\{\lvert B_{\mathcal A}\rvert_p: p\in X\} 
\\
\kappa_\delta &= \inf\{\lvert{\mathcal A} v\rvert_p : p\in \Omega(\delta),\ v \in E_p,\ \lvert v\rvert_p=1 \}.
\end{aligned}
\end{equation}
Combining the preceding estimates, we obtain
\begin{multline*}
\int_{\Omega(\delta)}[( s^2 \kappa_\delta^2 - \lvert s\rvert C_0 - \lambda_s) \rho^2 - \lvert d\rho\rvert^2] \lvert\xi\rvert^2\, d\mathrm{vol}^X \leq  \int_{\Omega(\delta)}( \lvert D_s(\rho \xi)\rvert^2 - \lvert d\rho\rvert^2 \lvert\xi\rvert^2 - \lambda_s\lvert\rho\xi\rvert^2)\, d\mathrm{vol}^X
\\
= \langle D_s \xi , D_s(\varrho^2 \xi)\rangle_{L^2(X)} - \lambda_s \| \rho\xi\|^2_{L^2(X)} = \langle (D_s^* D_s - \lambda_s) \xi, \rho^2 \xi \rangle_{L^2(X)},
\end{multline*}
that is 
\[
\int_{\Omega(\delta)}[( s^2 \kappa_\delta^2 - \lvert s\rvert C_0 - \lambda_s) \rho^2 - \lvert d\rho\rvert^2] \lvert\xi\rvert^2\, d\mathrm{vol}^X \leq 0.
\]
We simplify this inequality by choosing $\lvert s\rvert> 2(C_0 + C_1)/ \kappa_\delta$, so that 
\begin{equation}
\label{eq:basic_exponential_decay}
\int_{\Omega(\delta)}\left(\frac{1}{2}s^2 \kappa_\delta^2 \rho^2 - \lvert d\rho\rvert^2\right) \lvert\xi\rvert^2\, d\mathrm{vol}^X \leq 0.
\end{equation}
Note that, by an approximation argument, the following inequality holds for every $\rho$ Lipschitz.  
 
Let now $\chi : [0, \infty) \to [0,1]$ a cutoff with $\chi\vert_{[0,1]} \equiv 1,\ \mathrm{supp\,} \chi \subset [0,2]$ and $\lvert\chi'\rvert \leq 1$. Then define
\[
\rho(p) = \begin{cases} 
\left(1- \chi\left(\frac{r}{\delta}\right) \right) e^{\tfrac{1}{2} \lvert s\rvert r^2 \epsilon}, \quad \text{if $p \in Z_{\mathcal A}(4\delta)$,}
\\
e^{8\lvert s\rvert\delta^2 \epsilon}, \quad \text{if $p\in \Omega(4\delta)$,} 
\end{cases}
\]
where $r = r(p)$ is the distance of $p$ from the core $Z_{\mathcal A}$. Note that
\[
d\rho(p) = \begin{cases} 
0, \quad \text{if $p \in Z_{\mathcal A}(\delta) \cup \Omega(4\delta)$,}
\\
-\frac{1}{\delta} \chi'\left(\frac{r}{\delta}\right) e^{\tfrac{1}{2} \lvert s\rvert r^2 \epsilon} dr + \lvert s\rvert r\epsilon \rho dr, \quad \text{if $p\in \Omega(\delta; 2\delta)$,}
\\
\lvert s\rvert r\epsilon \rho dr, \quad \text{if $p\in \Omega(2\delta; 4\delta)$,} 
\end{cases}
\]
where $\Omega( \delta_1;\delta_2) = Z_{\mathcal A}(\delta_2) \setminus Z_{\mathcal A}(\delta_1)$, for every $0 < \delta_1 < \delta_2$. Substituting back to \eqref{eq:basic_exponential_decay} and using that $(a + b)^2 \leq 2(a^2 + b^2)$,
\begin{equation*} 
\frac{1}{2}s^2 \kappa^2_\delta e^{16\lvert s\rvert\delta^2 \epsilon}\int_{\Omega(4\delta)} \lvert\xi\rvert^2\, d\mathrm{vol}^X + s^2\int_{\Omega(\delta, 4\delta)}\left(\frac{1}{2}\kappa^2_\delta - 2 r^2 \epsilon^2\right)  \lvert\rho \xi\rvert^2\, d\mathrm{vol}^X \leq \frac{2}{\delta^2} \int_{\Omega(\delta; 2\delta)} e^{ \lvert s\rvert r^2 \epsilon} \lvert\xi\rvert^2\, d\mathrm{vol}^X.
\end{equation*} 
The second integral of the left hand side is positive (and can be ignored) provided that $8\delta \epsilon < \kappa_\delta$ that is when we choose $0<\epsilon < \kappa_\delta/( 8\delta)$. We obtain that
\[
\int_{\Omega(4\delta)} \lvert\xi\rvert^2\, d\mathrm{vol}^X \leq \frac{4}{s^2 \delta^2\kappa_\delta^2} e^{- 12\lvert s\rvert\delta^2 \epsilon} \|\xi\|_{L^2(X)}^2. 
\]
An analogue to the bootstrap argument given in \cite{m}[Cor.2.6] proves now the proposition, where $\delta$ is changing to $2\delta$ or $3\delta$ in each iteration and there are as many iterations as needed for application of Morrey's inequality with the norm of the space $C^{\ell, \alpha}(\Omega(\ell \delta))$. This explains the dependence of the constants $C', \varepsilon_0$ and $s_0$ in $\ell$ and $\alpha$.
\end{proof}

We include here an approximation theorem, for the spaces $W_{k, l}^{1,2}(N)$ introduced in Definition~\ref{defn:eighted_norms_spaces}:

\begin{theorem}
\label{thm:approximation_theorem_for_weighted_spaces}
We have alternative descriptions:
\begin{align}
L^2_k(N) &= \{ \xi \in L^2(N): \|\xi\|_{0,2,k,0} < \infty\}, \label{eq:L_2_k_description}
\\
W^{1,2}_{k,0}(N) &= \{\xi \in L^2(N): \bar\nabla^{\mathcal V} \xi \in L^2(N)\ \text{and}\ \|\xi\|_{1,2,k,0} <\infty\},\label{eq:W_2_k_0_description}
\\
W^{1,2}_{k,l}(N) &= \{ \xi \in W^{1,2}(N): \|\xi\|_{1,2,k,l} < \infty\}. \label{eq:W_2_k_l_description}
\end{align}
\end{theorem}
\begin{proof}
For every $\xi \in C^\infty_c(N)$, we have $\|\xi\|_2 \leq \|\xi\|_{0,2,k,0}$. Therefore a Cauchy sequence in $L^2_k(N)$ is also Cauchy in $L^2(N)$ and converges to the same limit. This proves the one inclusion in \eqref{eq:L_2_k_description}. For the other inclusion, we introduce a bump function $\rho:\mathbb{R} \to [0,1]$, with $\rho([0,1]) = \{1\}$ and $\rho^{-1}(\{0\}) = [2,\infty)$ so that $\lvert\rho'(t)\rvert \leq 1$. Given $\xi \in L^2(N)$ with $\|\xi\|_{0,2,k,0} < \infty$, we introduce $\xi_j(v) = \xi(v) \cdot \rho(\lvert v\rvert -j),\ v \in N,\ j \in \mathbb{N}$ and estimate
\[
\|\xi_j - \xi\|_{0,2,k,0}^2 \leq \int_{\{v: \lvert v\rvert \geq j\}} \lvert \xi\rvert^2(1+ r^2)^{k+1}\, d\mathrm{vol}^N \to 0,
\]
as $j \to \infty$.  Let $\varepsilon>0$ and choose $j$ so that $\xi_j$ is in $\varepsilon/2$ distance from $\xi$, we can find a $\phi\in C^\infty_c(B(Z, j))$ so that $\|\xi_j - \phi\|_{0,2,k,0} < \varepsilon/2$. Hence $\phi$ is $\varepsilon$-close to $\xi$. This proves  \eqref{eq:L_2_k_description}. 

Similarly, we have that $\|\bar\nabla^{\mathcal V}\xi\|_2 \leq \|\xi\|_{1,2,k,0}$ and that $\|\xi\|_{1,2} \leq \|\xi\|_{1,2,k,l}$. Hence if $\xi \in W^{1,2}_{k,l}(N)$, then $\xi$ poses weak derivatives in normal directions if $l=0$ and in every direction if $l>0$ and the weak derivatives belong in $L^2(N)$. On the other hand, given $\xi \in W^{1,2}(N)$ with $\|\xi\|_{1,2,k,l} < \infty$ we form the sequence $\xi_j$ as above and estimate
\[
\int_N \lvert \bar \nabla (\xi_j - \xi)\rvert^2(1 + r^2)^{k+1}\, d\mathrm{vol}^N \leq 2\int_{\{v: \lvert v\rvert \geq j\}} ( \lvert\xi\rvert^2+  \lvert \bar \nabla \xi\rvert^2)(1 + r^2)^{k+1}\, d\mathrm{vol}^N \to 0, 
\]
as $j \to \infty$. The same argument as with the proof of \eqref{eq:L_2_k_description} works to prove \eqref{eq:W_2_k_0_description} and \eqref{eq:W_2_k_l_description}. 
\end{proof}


\end{document}